\documentclass[a4paper,12pt]{amsart}
\usepackage[utf8]{inputenc}
\usepackage[english]{babel}
\usepackage{amssymb}
\usepackage{amsmath}
\usepackage{amsthm}
\usepackage[foot]{amsaddr}
\usepackage{enumitem}

\usepackage[toc,page]{appendix}

\usepackage[hmargin=3cm,vmargin=2.5cm]{geometry} %% onesided
\newcommand{\qmatrix}[1]{ \left( \begin{matrix} #1 \end{matrix} \right) }

\newcommand{\loc}{\mathrm{loc}}

\newcommand{\printcbf}[1]{%
\stepcounter{#1}{\bf \arabic{#1})}
}%

\def\[#1\]{\begin{align*}#1\end{align*}}
\def\be#1\ee{\begin{align}#1\end{align}}
\def\bea#1\eea{\begin{align}#1\end{align}}
\def\ben#1\een{\begin{align*}#1\end{align*}}

\newcommand{\supp}{\mathrm{supp}}
\newcommand{\diff}{\mathrm{d}}
\newcommand{\pr}{\mathrm{pr}}

\newcommand{\n}[1]{\left\Vert #1\right\Vert}
\newcommand{\la}{\left\langle}
\newcommand{\ra}{\right\rangle}

\newcommand{\R}{\mathbb{R}}

\newcommand{\N}{\mathbb{N}}

\newcommand{\mc}[1]{\mathcal{#1}}
\newcommand{\s}{\mc}
\newcommand{\p}[2]{\frac{\partial #1}{\partial #2}}
\newcommand{\pa}[1]{{{\partial}\over{\partial #1}}}
\newcommand{\dif}[1]{\frac{\diff}{\diff #1}}
\def\ol#1{\overline{#1}}

\newcommand{\q}[2]{{{\partial^2 #1}\over{\partial #2}^2}}

\newtheoremstyle{theorem}{0.5cm}{0.5cm}%
   {}%         Body font
   {}%         Indent amount (empty = no indent, \parindent = para indent)
   {\bfseries}% Thm head font
   {}%        Punctuation after thm head
   {2ex}%     Space after thm head (\newline = linebreak)
   {\thmname{#1}\thmnumber{ #2}\thmnote{ #3}}%         Thm head spec
\theoremstyle{theorem}
\newtheorem{theorem}{Theorem}[section]

\newtheorem{proposition}[theorem]{Proposition}
\newtheorem{example}[theorem]{Example}
\newtheorem{corollary}[theorem]{Corollary}
\newtheorem{remark}[theorem]{Remark}
\newtheorem{definition}[theorem]{Definition}
\newtheorem{lemma}[theorem]{Lemma}

\title[On Existence of $L^1$-solutions for BTE and RT Treatment Optimization]
{On Existence of $L^1$-solutions for Coupled Boltzmann Transport Equation and Radiation Therapy Treatment Optimization}

\author{J. Tervo$^1$}
\address{$^1$University of Eastern Finland, Department of Applied Physics, P.O.Box 1627, FI-70211 Kuopio, Finland}

\author{P. Kokkonen$^2$}
\address{$^2$Varian Medical Systems Finland Oy, Paciuksenkatu 21, 00270 Helsinki, Finland}
\email{pvkokkon@gmail.com}

\begin{document}
\renewcommand{\diff}{d}

\begin{abstract}
The paper considers a linear system of  Boltzmann transport equations modelling the evolution of three species of particles, photons, electrons and positrons. The system is coupled because of the collision term (an integral operator). The model is intended especially for dose calculation (forward problem) in radiation therapy.
It, however, does not apply to all relevant interactions in its present form.
We show under physically relevant assumptions that the system has a unique solution in appropriate ($L^1$-based) spaces and that the solution is non-negative when the data (internal source and inflow boundary source) is non-negative.
In order to be self-contained as much as is practically possible,
many (basic) results and proofs have been reproduced in the paper.
Existence, uniqueness and non-negativity of solutions for the related time-dependent coupled system are also proven. Moreover, we deal with inverse radiation treatment planning problem (inverse problem) as an optimal control problem both for external and internal therapy (in general $L^p$-spaces). Especially, in the case $p=2$ variational equations for an optimal control related to an appropriate differentiable convex object function are verified. Its solution can be used as an initial point for an actual (global) optimization.
\end{abstract}

\maketitle

%%%%%%%%%%%%%%%%%%%%%%%%%%%%%%%%%%%%%%%%%%%%%%%%%%%%%%%%%%%%%%%%%%%%%%%%
\section{Introduction}\label{intro}
%%%%%%%%%%%%%%%%%%%%%%%%%%%%%%%%%%%%%%%%%%%%%%%%%%%%%%%%%%%%%%%%%%%%%%%%

The Boltzmann transport equation (BTE) is an integro-partial differential equation which physically is based on the conservation laws. It has applications in many fields of scientific computation, including in among others optical tomography, cosmic radiation,
nanotechnology (e.g. plasma physics) and radiation therapy, which is considered in this paper.
For general mathematical theory of BTE with relevant boundary conditions we refer to \cite{agoshkov} and \cite{dautraylionsv6}. See also \cite{casezweifel}, \cite{cercignani}, \cite{duderstadt79}, \cite{pomraning} where the subject is considered from more physical point of view.
For more recent issues related to BTE can be found \cite{mokhtarkharroubi}, and some non-linear aspects in \cite{bellamo}.
Finally, for topics related to Monte Carlo methods in the context of BTE,
both from theoretical and practical points of view,
we refer to \cite{lapeyre03} and \cite{seco13}.

From the computational point of view,
the primary goal in radiation therapy is to
generate dose distributions in such a way that the prescribed dose conforms
as well as possible to the target volume,
while healthy tissue, and especially
the so-called critical organs, achieve as low dose as possible.
One considers the desired dose distribution in the patient domain to be known.
In {\it external radiotherapy} the problem is to find the optimal dose by defining the field intensity, that is the incoming particle flux, on (patches of) the patient surface, which can be regulated (controlled)
by the relative position and orientation of the patient and the accelerator head,
as well as by different mechanical devices therein such as jaws, wedges and multileaf collimators (MLCs).
In {\it internal radiotherapy} the radioactive sources are to be
positioned inside the patient tissue such that the desired dose distribution is achieved.
The determination of the external particle flux (or the distribution of internal sources) required to deliver the desired dose distribution is called
the \emph{inverse treatment planning problem} (IRTP),
which from mathematical point of view is an {\it inverse problem}.
The calculation of particle fluxes or dose in the patient tissue when the incoming fluxes or internal sources are known, is called {\it dose calculation},
and it is considered as a {\it forward problem}.

The solution of IRTP always requires some \emph{dose calculation model}.
Classical examples of such models,
and whose popularity is mainly explained by the
limitations in computer technology until quite recently,
are the so-called pencil beam calculation models (cf. \cite{gustafsson}, \cite{mayles07}).
These models are based on the idea that the incident radiation beam is divided
into beamlets (pencil beams) and the total dose is obtained as a superposition
(e.g. by convolution)
of doses contributed by these beamlets (pencil beam kernels),
which themselves are calculated by a Monte Carlo simulation
or other methods such as Fermi-Eyges theory, which is based on a rough approximation of BTE (\cite{blais90}, \cite{mayles07}, \cite{seco13}).
Even though various kinds of corrections for pencil beam models have been proposed (see references mentioned in \cite{larsen97}, for example),
they remain inaccurate especially in regions which are highly non-homogeneous.
On the other hand, one can develop various dose calculation
models based on point (spread) kernels (see \cite{mayles07})
that result from a single interaction by an incident photon (for example)
at a given point in homogeneous material.
Like the pencil beam kernels above,
the point kernels are typically calculated by Monte-Carlo simulations.

In radiation therapy BTE describes how radiation is scattered and absorbed in a 
tissue. The sources of high energy particles, such as photon or electrons may be on 
the surface of the patient (external therapy)
or inside the patient close to the cancer tissue (internal 
therapy). In any case they mobilize three kinds of particles, photons, electrons 
and positrons, whose simultaneous evolution should be taken into account in the 
transport model.
In this setting, the potential creation of (or contamination by)
other heavy particles (such as neutrons)
will not be taken into account since their
contribution is negligible (cf. \cite{seco13}) when the source beam
consists of photons or electrons in the relevant range (say 6-15MeV) of energies.

Dose calculation models governed by the (linear) BTE are valid in inhomogeneous material. They take rigorously into account the scattering and absorption effects
(phenomena emerging from particle/nuclear physics) in physically solid way.
We assume here that the transport of radiation particles is ruled by the following linear coupled system of three BTEs
(for a derivation of the linear BTE, see \cite{allaire12}, \cite{duderstadt79}, \cite{agoshkov})
\be\label{i1}
\omega\cdot\nabla\psi_j(x,\omega,E)+\Sigma_j(x,\omega,E)\psi_j
-(K_j\psi)(x,\omega,E)=f_j(x,\omega,E),\quad j=1,2,3
\ee
together with an inflow boundary condition
\be\label{i2}
{\psi_j}_{|\Gamma_-}=g_j,\quad j=1,2,3
\ee
where
\[
(K_{j}\psi)(x,\omega,E)=\sum_{k=1}^3\int_S\int_I\sigma_{kj}(x,\omega',\omega,E',E)\psi_k(x,\omega',E') dE' d\omega',\quad j=1,2,3.
\]
The first term on the left in (\ref{i1}) is called a convection (or advection) operator,
the second term is a scattering operator and the third one is a collision operator.
On the right, the functions $f_j$ represent the (internal) sources,
and $g_j$ in \eqref{i2} are boundary sources.
The the system is \emph{coupled} through the operators $K_j$
(unless, of course, $\sigma_{jk}=0$ for $j\neq k$).
Here a solution $\psi=(\psi_1,\psi_2,\psi_3)$ of the problem \eqref{i1}-\eqref{i2} is a vector-valued function whose components
describe the particle number densities of photons, electrons and positrons, respectively. 
Its dynamical counterpart is given by (see section \ref{tds1})
\be\label{i1a}
\frac{1}{\n{v_j}}{\p {\psi_j}t}+\omega\cdot\nabla\psi_j+\Sigma_j(x,\omega,E)\psi_j
-K_j\psi=f_j(x,\omega,E,t),\quad j=1,2,3
\ee
where $\n{v_j}$ is the speed of the particle $j$,
together with inflow boundary and initial condition
\begin{align}
{\psi_j}_{|\Gamma_-\times [0,T]}=&g_j, \label{i2a} \\ 
\psi_j(0)=&\psi_{0j}, \quad j=1,2,3. \label{i2aa}
\end{align}
The dynamical solution is defined in seven-dimensional
phase space $(x,\omega,E,t)$; position, angle (direction of velocity), energy of particles and time.
In the steady state (stationary state) there is no time dependence,
and so the phase space $(x,\omega,E)$ is six-dimensional.
This is basically always the case in applications related to radiotherapy,
because the relevant radiation field(s) $\psi$ anyway reach the steady state
nearly instantly (\cite{borgers99}).

We remark that the model (\ref{i1}), (\ref{i2}) (and correspondingly (\ref{i1a}), (\ref{i2a}), (\ref{i2aa})) {\it is valid only for certain interactions} such as for Compton scattering that 
is, "photons to photons" and "photons to electrons" scattering.
In addition, it covers e.g. some parts of M\o ller scattering, namely "secondary electrons to electrons" scattering (and similarly for positrons).
It is not applicable to e.g. "primary electrons to electrons" scattering in M\o ller interaction.
We shall describe in more detail this subject in section \ref{sec:coll} (see also \cite{tervo16-up}). The point is that  for some interactions
the differential cross-sections
may have singularities or even {\it hyper-singularities} which implies that the nature of transport equations dramatically
changes by bringing extra first-order pseudo-differential (or approximately partial differential)
terms into the transport equation.
A vastly applied approximation to cover these problematic interactions is the so-called Continuous Slowing Down-Boltzmann Transport Equation
(CSDA-BTE) which we have analysed in \cite{tervo16-up}.
CSDA-BTE means that the following equation
(cf. \cite{frank10}, \cite{larsen97})
\bea\label{sda}
-{\p {(S_{j,r}\psi_j)}E}+
\omega\cdot\nabla\psi_j+\Sigma_{j,r}(x,\omega,E)\psi_j
-K_{j,r}\psi=f_j(x,\omega,E),\quad j=2,3
\eea
is used instead of (\ref{i1}) for $j=2,3$ (electrons and positrons) where
\[
K_{2,r}\psi=&{}
\int_S\int_I\sigma_{1,2}(x,\omega',\omega,E',E)\psi_1(x,E',\omega') dE' d\omega'\\
&{}+\int_S\int_I\sigma_{2,2,r}(x,\omega',\omega,E',E)\psi_2(x,E',\omega') dE' d\omega' \\
&{}+\int_S\int_I\sigma_{3,2}(x,\omega',\omega,E',E)\psi_3(x,E',\omega') dE' d\omega',
\]
and
\[
K_{3,r}\psi=&{}
\int_S\int_I\sigma_{1,3}(x,\omega',\omega,E',E)\psi_1(x,E',\omega') dE' d\omega'\\
&{}+
\int_S\int_I\sigma_{2,3}(x,\omega',\omega,E',E)\psi_2(x,E',\omega') dE' d\omega' \\
&{}+\int_S\int_I\sigma_{3,3,r}(x,\omega',\omega,E',E)\psi_3(x,E',\omega') dE' d\omega'.
\]
Above, for $j=2,3$, functions $\Sigma_{j,r}(x,E)$ are the \emph{restricted total cross-sections},
$\sigma_{j,j,r}(x,E',E,\omega',\omega)$ are the \emph{restricted differential cross-sections},
and the factors $S_{j,r}=S_{j,r}(x,E)$ are the so-called \emph{restricted stopping powers}.
The model neglects soft inelastic interactions.
Besides the inflow boundary condition (\ref{i2}),
one must demand from this model that the solution satisfy
\[
\psi_2(x,\omega,E_{\rm m})=\psi_3(x,\omega,E_{\rm m})=0,
\]
or at least that
\[
\lim_{E\to\infty}\psi_2(x,\omega,E)=\lim_{E\to\infty}\psi_3(x,\omega,E)=0,
\]
where in the last case we naturally assume that $I=[E_0,\infty[$.
This requirement makes the overall problem mathematically well-defined that is, under relevant physical assumptions  the problem has an unique solution.  
Similarly we can replace the equation \eqref{i1a} for $j=2,3$ in the time-dependent case to 
obtain a time-dependent CSDA-BTE. For stationary single CSDA-BTE equation existence of solutions 
and some optimal control results in $L^2$-spaces are recently shown in \cite{tervo16-up} (see 
also \cite{frank10} where one assumes that the stopping power is independent of the spatial
coordinate $x$ and that the collision operator has a special form).

In this paper we consider the existence of solutions for the above coupled system
\eqref{i1}, \eqref{i2} in spaces $L^p(G\times S\times I)^3$ especially for $p=1$.
Here $G\subset\R^3$ is the spatial domain, $S\subset\R^3$ is the unit sphere and $I=[E_0,E_{\rm m}]$ is the energy interval.
The energy $E$ and the angle $\omega$ are kept everywhere separated that is, the phase space is $G\times S\times I$.
At first we consider the so-called escape time mapping $t=t(x,\omega)$
and recall of its analytical properties, which are useful e.g. in investigations of regularity of the solutions. After that we reproduce the well-known solutions (modified to our situation) of the convection/scattering equation (i.e. without the collision operator),
by using the Lagrange's method i.e. the method of characteristics.
Then the $m$-dissipativity property of the convection operator under homogeneous inflow boundary data is shown for one type of particles (again modified to our case).
It follows from these considerations that the convection operator (under homogeneous inflow boundary data)
is $m$-dissipative in the spaces $L^1(G\times S\times I)^3$ 
related to system (\ref{i1}), which is still uncoupled (section \ref{cos1}). Under 
certain, physically relevant assumptions we show the dissipativity of the (coupled) 
scattering-collision operator. Putting these together and applying the properties 
of $m$-dissipative operators and the lifting results of inflow boundary data, the 
existence and uniqueness result of solutions for coupled system \eqref{i1}, 
\eqref{i2} is proved. In addition, we verify non-negativity of solutions when the data is non-negative. In section \ref{tds1} the existence of solutions of the time-dependent coupled system \eqref{i1a}, \eqref{i2a}, \eqref{i2aa} is studied.

The above model of transport is linear and, therefore, neglects any non-linear interaction (cf. \cite{duderstadt79}, \cite{friedlander02}). 
In addition the inflow boundary condition (\ref{i2}) is not exactly correct because minor part of the  particles return to the patient domain $G$.
The reflection boundary conditions of the form
${\psi_j}_{|\Gamma_-}=R_j({\psi_j}_{|\Gamma_+})+g_j,\ j=1,2,3$, where $R_j$ are appropriate (unbounded) operators might be more accurate (see \cite{tervo16-up}, \cite{dautraylionsv6}, \cite{lions71}).

The use of BTE in dose calculation needs the choice of {\it total and differential cross-sections}.
In radiation therapy the cross-sections of primary interest are those for water (tissue), bone and air (void-like regions). 
For a more thorough discussion on the cross sections relevant to radiation therapy, we refer to \cite{bomanthesis}, \cite{hensel}.

The analytical (explicit) solution of BTE is known only when the
underlying geometrical settings, the structures of cross-sections,
the sources and the incoming fluxes of particles (boundary conditions) are rather simple (see e.g. \cite{duderstadt79}, Ch. 2).
Hence in practical situations one  must apply appropriate numerical schemes for obtaining the solutions.
Various kinds of numerical methods can be utilized for solving the transport equation (see \cite{ackroyd}),
for instance the combination of finite element method (or collocation method) and discrete ordinate method,
or Monte Carlo.

In section \ref{srs} we consider the above mentioned IRTP problem.
In solving the IRTP problem 
one may use physical or biological criteria for optimization (for some general backgrounds see e.g. \cite{shepard99}). We consider here only physical criteria which are common in practical planning.
Biological criteria are not considered,
because their grounds from the modelling perspective have not been
well established.
We notice, however that the optimization schemes given in this paper can also be founded on the biological criteria in an analogous manner,
although the resulting object function is likely to be more multiextremal.

The patient domain $G\subset\R^3$ consists of tumor volume ${\bf T}$, critical organ's region ${\bf C}$ and the normal tissue's region ${\bf N}$. 
Hence  $G={\bf T}\cup {\bf C}\cup {\bf N}$
where the union is mutually disjoint.
The tumor volume (that is, the target) includes the tumor and some safety margin. Critical organs and normal tissue are build up of healthy tissue,
and should receive as low a dose as possible.

Typically the resulting object function  based on the physical criteria is,
in the stationary case, of the form (see section \ref{RTP})
\bea\label{io}
J(f,g)=&c_{\bf T}\n{D_0-{\s D}(f,g)}_{L^p({\bf T})}^p
+
c_{\bf C}\n{(D_C-{\s D}(f,g))_-}_{L^p({\bf C})}^p \nonumber \\
&+
c_{\bf N}\n{(D_N-{\s D}(f,g))_-}_{L^p({\bf N})}^p \nonumber \\
&+
c_{\rm DV}\Big(\Big(v_C-\frac{1}{\mc{L}^3({\bf C})}\int_{\bf C}H(({\s D}(f,g))(x)-d_C) dx\Big)_-\Big)^p
\eea
and where $c_{\bf T}, c_{\bf C}, c_{\bf N}, c_{\rm DV}$ are positive weights,
$\mc{L}^3$ is the 3-dimensional Lebesgue measure, $H$ is the Heaviside function
and $a_-$ denotes the negative part of $a\in\R$.
Here  ${\s D}(f,g)=D(\psi(f,g))$ where $D$ is the dose operator
(see section \ref{RTP1})
and $\psi=\psi(f,g)$ is the solution of (\ref{i1})-(\ref{i2}).
We note that $f=0$ for external therapy and $g=0$ for internal therapy. 
In (\ref{io}) the first three terms are convex and (locally) Lipschitz continuous (for $p=2$ the first term is also differentiable), and the last term is both non-convex and non-differentiable in general.
Moreover, the last term can be replaced a by Lipschitz continuous counterpart by replacing Heaviside function $H$ with its definition by a Lipschitz continuous approximation (which in practice is reasonable).
After this replacement the whole object function (\ref{io}) is (locally) Lipschitz continuous. In addition the admissible sets (as given in section \ref{RTP}) are convex.

In practice, solving deterministically (i.e. without Monte Carlo methods) the discretized BTE is a quite formidable numerical task because in three spatial dimensions we have in total $3+2+1=6$ phase space variables (i.e. 3 spatial, 2 angular and 1 energy dimensions).
In time-dependent case one would also have to take the time variable $t$  into 
account, further increasing the state space dimension by one.
We also notice that the collision term of BTE necessitates,
in general, the consideration of two additional variables $\omega'$, $E'$ which, 
however, are not phase space variables \emph{per se}.
Hence the numerical dimension of the problem is very large
in the sense that the total number of grid points needed in any (deterministic) discretization scheme
grows fast (to $6$th or $7$th power, say)
with the number of grid points used for discretizing each individual dimension (if assumed to be proportional).
There have only been a few attempts to solve BTE using deterministic methods for radiotherapy needs in three spatial dimensions without further approximations
and/or geometrical simplifications (see e.g.  \cite{bomanthesis}).
In \cite{hensel} computationally less complex algorithm is developed 
and some simulations are carried out in slab $3$D-geometry.

Since the object function (\ref{io}) contains non-convex terms, {\it global optimization} (\cite{pinter}) for (locally) Lipschitz continuous functions in convex domains is needed. Moreover, the applied optimization method should be reasonably fast.
The initialization (determination of an {\it initial solution} for global optimization scheme) is necessary since the determination of a carefully chosen initial point for a large dimensional global optimization 
scheme is very essential for achieving (time savings and) satisfactory results (\cite{pinter14}).
We prove in section \ref{RTP} (in the case $p=2$) that for a certain related (convex) object function,
the optimal control exists, and formulas  for it in a variational form are given.
We suggest that this solution is used as an initial solution.
Preliminary simulations show that the computation of the initial solution in this 
way is fast enough (\cite{bomanthesis}).
In section \ref{srs} we bring up some challenges and problems related to IRTP.  

Finally, we remark that an optimization scheme can be formulated in such a way
that in external therapy the device (such as MLC) parameters are {\it directly} as decision parameters both in static and dynamical  delivery techniques.
This is based on the fact that the incoming flux $g$ can be expressed as a function of these parameters, say $g=g({\bf q})$ (see \cite{tervo02} for a certain  implementation related to MLC).
Substitution of this expression $g=g({\bf q})$ to 
${\s D}(0,g)$ gives the object function (\ref{io}) as a function of  ${\bf q}$. 
The resulting object function is, however, highly multiextremal, and
thus seeking its global minimum is rendered more difficult.

The authors would, moreover, like to mention that in \cite{tervo02}, p.121 one must add for functions in $H$ the requirement $f_{|\Gamma}\in L_2(\Gamma,|\omega\cdot\nu|d\sigma dE d\Omega)$ 
which is erroneously omitted there. In addition, in \cite{tervo07}, p. 824 the space $H$ should be the completion of $C^1(\ol{V}\times I\times S)$ with respect to $\la\cdot,\cdot\ra_H$-inner product; not only the intersection $H_1\cap H_2$ as it 
was erroneously defined there.

{\bf Acknowledgements.}
The authors would like to thank C. Boylan and T. Torsti from Varian Medical Systems Finland
for their valuable comments and suggestions.

%%%%%%%%%%%%%%%%%%%%%%%%%%%%%%%%%%%%%%%%%%%%%%%%%%%%%%%%%%%%%%%%%%%%%%%%
\section{Preliminaries}\label{ls1}
%%%%%%%%%%%%%%%%%%%%%%%%%%%%%%%%%%%%%%%%%%%%%%%%%%%%%%%%%%%%%%%%%%%%%%%%

We assume that $G$ is an open bounded set in $\R^3$, equipped with Lebesgue measure $\mc{L}^3$, with piecewise smooth (orientable) $C^1$ boundary $\partial G$,
that is, $\partial G$ is $2$-dimensional (orientable) piecewise $C^1$-manifold
(i.e. a $C^1$-manifold with corners)
such that $G$ lies locally only on one side of $\partial G$, see e.g. \cite{grisvard}, \cite{lee03}.
For example, $G$ may be a parallelepiped. 

The unit outward pointing normal on $\partial G$ is denoted by $\nu$,
and the surface measure on $\partial G$ is $\sigma$.
Let $S$ be the unit sphere in $\R^3$ equipped with the
surface measure $\mu_S$,
and let $I$ be the (energy) interval $[E_0,E_{\rm m}]$, $E_0\geq 0$ (which we assume to be bounded),
equipped with the Lebesgue-measure $\mc{L}^1$.
The variable in $S$ (in $I$) is denoted by $\omega$ (by $E$).
The surface measures $\sigma$ and $\mu_S$ are induced by the Lebesgue measure,
and we write these in integrals over $\partial G$ and $S$,
respectively,
as $\diff\sigma(y)$ and $\diff\omega:=\diff\mu_S(\omega)$,
where $y\in \partial G$ and $\omega\in S$.
In integrals over the spatial domain $G$ and the energy interval $I$, 
we write $\diff x=\diff\mc{L}^3(x)$, and $\diff E:=\diff\mc{L}^1(E)$, where $x\in G$, $E\in I$.

\begin{remark}\label{pre-re1}
A.
In the definition of the (single) collision operator
\be\label{re1-1}
(K\psi)(x,\omega,E)=\int_{S\times I}\sigma(x,\omega',\omega,E',E)\psi(x,\omega',E') dE' d\omega'
\ee
it would be possible (instead of the above particular measures) to choose
{\it any positive Radon measures} $\rho_I$ and $\rho_S$ on (Borel sets of) the interval $I$ and on the unit sphere $S$
that is,
\be\label{re1-2}
(K\psi)(x,\omega,E):=\int_S\int_I\sigma(x,\omega',\omega,E',E)\psi(x,\omega',E') d\rho_I(E') d\rho_S(\omega').
\ee
This kind of more general choice of measures has obvious benefits.

B. Our considerations are based on the fact that $K$ is a bounded operator 
under assumption \eqref{colleh} (see Theorem \ref{NoLabel})
\[
K:L^1(G\times S\times I,dx d\omega dE)\to L^1(G\times S\times I,dx d\omega dE).
\]
The results of this paper can be straightforwardly generalized to the case where the measure $d\omega dE$, i.e. $\mu_S\otimes \mc{L}^1$, on $S\times I$ is replaced with a more general positive Radon measure $\ol{\rho}$ on $S\times I$
that is, we may seek solutions in the space $L^1(G\times S\times I,dx d\ol{\rho}(\omega,E))$. In this setting $K$ would be a linear operator
\be\label{re1-3}
K:L^1(G\times S\times I,dx d\ol{\rho}(\omega,E))\to L^1(G\times S\times I,dx d\ol{\rho}(\omega,E)).
\ee
Especially, the measure $\rho$ may be of the form
$\ol{\rho}=\ol{\rho}_S\otimes\ol{\rho}_I$ where $\ol{\rho}_S$ ($\ol{\rho}_I$) is a positive Radon measure on $S$ (on $I$).
Only the boundedness and dissipativity (coercitivity) criteria for $K$
(see \eqref{scateh}, \eqref{colleh}, \eqref{colleha}, \eqref{co2a}, \eqref{co2aa})
must be modified in a suitable way.

We emphasise that in the definition \eqref{re1-2} of $K$
there may be a different Radon measure $\rho=\rho_S\otimes\rho_I$ on $S\times I$, 
but when $K$ is considered as a linear operator between 
$L^1$-spaces one must have same Radon measure $\ol{\rho}$ on $S\times I$
in both domain and range as in (\ref{re1-3}). The same 
observation concerns the case where the solutions are sought in more general spaces $L^p(G\times S\times I,dx d\ol{\rho}(\omega,E))$, $1\leq p<\infty$.

C. When defining $L^p$-spaces as in case B. above,
using general Radon measure $\ol{\rho}$ on $S\times I$,
it is also important to assume (see \eqref{eq:def_Gamma} below)
that the subset $\Gamma_0$ of $\Gamma$ has zero measure
with respect to the measure $\sigma\otimes \ol{\rho}$ on $\Gamma$.
It can be shown that it follows from this assumption that the set $N_0$ (see \eqref{eq:N_0})
has zero measure with respect to $\mc{L}^3\otimes\ol{\rho}$.
However, the argument used in the proof of Theorem \ref{th:bardos}
is no longer applicable in this more general setting,
and must be replaced by one using, for example, Fubini's theorem.

D. In fact, the Radon-measures $\rho_I$ (on $I$) and $\rho_S$ (on $S$) might even depend on parameters $(x,\omega',\omega,E)$ and $(x,\omega,E',E)$, 
respectively.
For example, the collision operator related to elastic scattering has the form
\[
(K\psi)(x,\omega,E)=\int_{S}\sigma(x,\omega',\omega,E)\psi(x,E,\omega') d\omega'
\]
which can be given as
\[
(K\psi)(x,\omega,E)=\int_{S}\int_I\sigma(x,\omega',\omega,E)\psi(x,E',\omega') d\rho_I(E'|E)d\omega',
\]
where for every $E$, $\rho_I(\cdot|E)$ is a Radon-measure on $I$ defined by
\[
\rho_I(A|E):=\delta_0(A-E)
=\begin{cases}
1, & \textrm{if}\ E\in A \\
0, & \textrm{if}\ E\notin A
\end{cases}
\]
for all Borel sets $A\subset U$. 
\end{remark}

Let $(\partial G)_r$ be the $C^1$-part of $\partial G$.
We use the following abbreviations
\begin{align}\label{eq:def_Gamma}
\Gamma&=\{(y,\omega,E)\in \partial G\times S\times I\} \nonumber\\
\tilde \Gamma&=\{(y,\omega,E)\in (\partial G)_r\times S\times I\} \nonumber\\
\Gamma_+&=\{(y,\omega,E)\in (\partial G)_r\times S\times I\ |\ \omega\cdot\nu(y)>0\} \nonumber\\
\Gamma_-&=\{(y,\omega,E)\in (\partial G)_r\times S\times I\ |\ \omega\cdot\nu(y)<0\} \nonumber \\
\tilde\Gamma_0&=\{(y,\omega,E)\in (\partial G)_r\times S\times I\ |\ \omega\cdot\nu(y)=0\}
\nonumber \\
\Gamma_0&=\tilde\Gamma_0\cup(\Gamma\backslash\tilde{\Gamma}).
\end{align}
Then $\Gamma=\Gamma_+\cup\Gamma_-\cup\Gamma_0\cup (\Gamma\backslash\tilde{\Gamma})$ and the union is mutually disjoint.
Notice that $\Gamma_{\pm}$ are open sets in $\partial G\times S\times I$ and
$\Gamma\backslash\tilde{\Gamma}=(\partial G\backslash (\partial G)_r)\times S\times I$ is zero-measurable in $\Gamma$.
Moreover, $\Gamma_{0}$ is a closed set in $\Gamma$,
and it is in fact zero measurable as demonstrated in the following lemma.

\begin{lemma}
The set $\Gamma_0$ is zero-measurable in $\Gamma$ with respect to $\sigma\otimes\mu_S\otimes \mc{L}^1$.
\end{lemma}

\begin{proof}
It is enough to show that $\tilde{\Gamma}_0$ is zero-measurable in $\tilde{\Gamma}$.
For each $y\in(\partial G)_r$ the set
\[
{S}_0(y,E)=\{\omega\in S\ |\ \omega\cdot\nu(y)=0\}
\]
is zero-measurable in $S$ (w.r.t the measure $d\omega$) and $\tilde{\Gamma}_0$ can be written as
\[
\tilde{\Gamma}_0=\{(y,\omega,E)\in\partial G\times S\times I\ |\ (y,E)\in(\partial G)_r\times I,\ \omega\in {S}_0(y,E)\}
\]
Hence by Fubini's theorem (with $dE$ is the Lebesgue measure on $I$, again with a slight abuse of notation)
\[
(\sigma\otimes \mu_S\otimes \mc{L}^1)(\tilde{\Gamma}_0)
=\int_{(\partial G)_r\times I} \Big(\int_{{S}_0(y)} d\omega\Big) d\sigma(y)dE
=0.
\]
This finishes the proof.
\end{proof}

\begin{remark}
If $\partial G$ happened to be piecewise $C^2$-manifold,
the proof of the zero-measurability of $\tilde{\Gamma}_0$ in $\tilde{\Gamma}$
above could also be proved in the following way
that is more differential geometric in flavor:

Let $f:\tilde{\Gamma}\to\R;\ f(y,\omega,E)=\omega\cdot\nu(y)$, which is $C^1$-smooth and $\tilde{\Gamma}_0=f^{-1}(0)$. 
The differential $Df$ of $f$ on $\tilde{\Gamma}$ can be seen to be
\[
Df(y,\omega,E)(a,b,c)=(D\nu(y)a)\cdot \omega+\nu(y)\cdot b, \quad (a,b,c)\in T_y (\partial G)_r\times T_\omega S\times \R.
\]
Clearly, if $f(y,\omega,E)=0$, then $Df(y,\omega,E)\neq 0$.
Indeed, otherwise, as $\nu(y)\cdot b=0$ for all $b\in T_\omega S$,
there would be an $\alpha\in\R$ such that $\nu(y)=\alpha\omega$,
hence $0=f(y,\omega,E)=\alpha\n{\omega}^2_{\R^3}=\alpha$, which is impossible as
this would imply that $\nu(y)=0$.
Then $\tilde{\Gamma}_0$
is a $C^1$-submanifold of $\tilde{\Gamma} = (\partial G)_r\times S\times I $ of codimension 1, and so has measure zero.

\end{remark}

Furthermore, let
\[
W^1(G\times S\times I)=\{\psi\in L^1(G\times S\times I)\ |\ \omega\cdot\nabla\psi\in L^1(G\times S\times I)\}
\]
(here $\nabla$ is taken with respect to $x$-variable only and $\nabla\psi$ is understood in the distributional sense).
The space $W^1(G\times S\times I)$ is equipped with the norm
\[
\n{\psi}_{W^1(G\times S\times I)}=\n{\psi}_{L^1(G\times S\times I)}+
\n{\omega\cdot\nabla\psi}_{L^1(G\times S\times I)}.
\]
Then $W^1(G\times S\times I)$ is a Banach space.
It is known that the space
\[
\mc{D}(\ol G\times S\times I)=\{\phi_{| G\times S\times I}|\ \phi\in C_0^\infty(\R^3\times S\times \R)\}
\]
is dense in $W^1(G\times S\times I)$ (cf. \cite{dautraylionsv6}, p. 221 and the references mentioned therein).

For $\Gamma_-$ we use can define the space of $L^1$-functions with respect to the measure
$|\omega\cdot\nu|\ d\sigma d\omega dE=-\omega\cdot\nu\ d\sigma d\omega dE$ which is denoted by $T^1(\Gamma_-)$
that is, $T^1(\Gamma_-)=L^1(\Gamma_-,|\omega\cdot\nu|\ d\sigma d\omega dE)$.
The norm in   $T^1(\Gamma_-)$   is
\[
\n{h}_{T^1(\Gamma_-)}=\int_{\Gamma_-}|h(y,\omega,E)|\ |\omega\cdot\nu|\ d\sigma d\omega dE.
\]
The space $T^1(\Gamma_+)$ of $L^1$-functions (and its norm) on $\Gamma_+$ with respect to the measure
$|\omega\cdot\nu|\ d\sigma d\omega dE=\omega\cdot\nu\ d\sigma d\omega dE$ is defined similarly.
Moreover, one has the following
trace theorem (see \cite{dautraylionsv6}, pp. 230-231).

\begin{theorem}\label{th:trace}
For any compact set $K\subset \Gamma_{\pm}$ there exists a constant $C_K>0$ such that
\[
\int_K|\psi(y,\omega,E)|\ |\omega\cdot\nu|\ d\sigma d\omega dE\leq C_K\n{\psi}_{W^1
(G\times S\times I)}\ {\rm for\ all}\ \psi\in \mc{D}(\ol G\times S\times I).
\]
\end{theorem}

\begin{proof}
We assume here, for definiteness, that $K\subset\Gamma_{+}$, the case $K\subset\Gamma_-$ being proven in analogous way.
For $(\omega,E)\in S\times I$, let $K_{(\omega,E)}=\{x\in\R^3\ |\ (x,\omega,E)\in K\}$ which is a compact subset of $\partial G$.
Choose a function $\theta_K\in\mc{D}(\R^3\times S\times\R)$
such that $0\leq \theta_K\leq 1$ everywhere, $\theta_K|_K=1$
and $(\supp\ \theta_K) \cap \Gamma_-=\emptyset$.
Then as $\theta_K(x,\omega,E)(\omega\cdot\nu(x))\geq 0$ for all $x\in (\partial G)_r$ and
$\theta_K(x,\omega,E)=1$ for all $x\in K_{(\omega,E)}$,
and because $|\psi(\cdot,\omega,E)|$ belongs to the standard Sobolev space
of index $(1,1)$ on $G$ i.e. to $W^{1,1}(G)=\{f\in L^1(G)\ |\ \nabla f\in L^1(G)\}$, we have
\[
& \int_{K_{(\omega,E)}} |\psi(\cdot,\omega,E)|(\omega\cdot\nu) d\sigma
=\int_{K_{(\omega,E)}} |\psi(\cdot,\omega,E)|\theta_K(\cdot,\omega,E)(\omega\cdot\nu) d\sigma \\
\leq & \int_{\partial G} |\psi(\cdot,\omega,E)|\theta_K(\cdot,\omega,E) (\omega\cdot\nu) d\sigma
=\Big|\int_G \omega\cdot\nabla\big(\theta_K(\cdot,\omega,E)|\psi(\cdot,\omega,E)|\big)\diff x\Big| \\
\leq& \int_G |\omega\cdot\nabla\theta_K(\cdot,\omega,E)||\psi(\cdot,\omega,E)| dx
+\int_G |\theta_K(\cdot,\omega,E)| \big|\omega\cdot \nabla|\psi(\cdot,\omega,E)|\big| dx,
\]
where in the third phase we used the Stokes' Theorem, along with the fact that the integral over $\partial G$ on the left is non-negative.
Letting $C_K>0$ be such that
$|\theta_K|\leq C_K$ and $\n{\nabla\theta_K}_{\R^3}\leq C_K$
on $\R^3\times S\times\R$,
taking into account that
$\nabla|\psi(\cdot,\omega,E)|=\mathrm{sgn}(\psi(\cdot,\omega,E))\nabla\psi(\cdot,\omega,E)$
(cf. \cite{grigoryan09}, Section 5.1)
we have, by integrating the above inequalities over $S\times I$,
\[
0\leq \int_K |\psi(x,\omega,E)|(\omega\cdot\nu)d\sigma(x)d\omega dE
\leq& C_K\big(\n{\psi}_{L^1(G\times S\times I)}+\n{\omega\cdot\nabla\psi}_{L^1(G\times S\times I)}).
\]
The right hand side being equal to $C_K\n{\psi}_{W^1(G\times S\times I)}$,
this finishes the proof.
\end{proof}

\begin{remark}
Since $|\omega\cdot\nu|$ is bounded from below on a compact $K\subset\Gamma_-$,
the previous inequality implies
\[
\int_K|\psi(y,\omega,E)|d\sigma d\omega dE\leq \tilde{C}_K\n{\psi}_{W^1(G\times S\times I)},
\]
for some constant $\tilde{C}_K>0$ depending on $K$.
\end{remark}

As a direct consequence of Theorem \ref{th:trace},
any element $\psi\in W^1(G\times S\times I)$ has a well defined trace
$\psi_{|\Gamma_-}$ in $L^1_{\rm loc}(\Gamma_-,|\omega\cdot\nu|\ d\sigma d\omega dE)$ defined by
\[
\psi_{|K}:=\lim_{j\to\infty}\ {\phi_j}_{|K}\ {\rm for\ any\ compact\ subset\ K\subset\Gamma_+},
\]
where $\{\phi_j\}\subset \mc{D}(\ol G\times S\times I)$ is any sequence such that $\lim_{j\to\infty}\n{\phi_j-\psi}_{W^1(G\times S\times I)} =0$. In addition the trace mapping $\gamma_-:W^1(G\times S\times I)\to \ L^1_{\rm loc}(\Gamma_-,
|\omega\cdot\nu|\ d\sigma d\omega dE);\
\gamma_-(\psi)=\psi_{|\Gamma_-}$ is continuous.

In a similar way one has a continuous trace mapping $\gamma_+:W^1(G\times S\times I)\to \ L^1_{\rm loc}(\Gamma_+,|\omega\cdot\nu|\ d\sigma d\omega dE)$. Hence we can define 
(a.e. unique) the trace  $\gamma(\psi)$ on $\Gamma$ for $\psi\in W^1(G\times S\times I)$ by setting (recall that $\Gamma_0$ has a measure zero)
\[
\gamma(\psi)(y,\omega,E)=\begin{cases}
\gamma_+(\psi)(y,\omega,E),& (y,\omega,E)\in\Gamma_+\\
\gamma_-(\psi)(y,\omega,E),& (y,\omega,E)\in\Gamma_-\\
0,& (y,\omega,E)\in\Gamma_0
\end{cases}. 
\]
Finally we denote by $T^1(\Gamma)$ the space of $L^1$-functions on $\Gamma$ with respect to the measure
$|\omega\cdot\nu|\ d\sigma d\omega dE $ that is,
\[
T^1(\Gamma)=L^1(\Gamma,|\omega\cdot\nu|\ d\sigma d\omega dE).
\]
The norm in $T^1(\Gamma)$ is
\[
\n{h}_{T^1(\Gamma)}= \int_{\Gamma}|h(y,\omega,E)|\ |\omega\cdot\nu|\ d\sigma d\omega dE.
\]

Evidently, the spaces $T^1(\Gamma_-)$ and $T^1(\Gamma_+)$ are isometrically embedded into $T^1(\Gamma)$ through the operation of extension by zero.
From the fact that $\Gamma_0$ has a measure zero in $\Gamma$
it follows that one can isometrically identify $T^1(\Gamma)$
with $T^1(\Gamma_-)\times T^1(\Gamma_+)$ (as can be seen easily).
One can, moreover, identify the spaces $T^1(\Gamma_-)$ and $T^1(\Gamma_+)$ with each other isometrically,
a result whose justification will be postponed until Corollary \ref{cor:gamma-pm}
(altenatively, see\cite{cipolatti06}, Corollary 2.2 or \cite{cessenat85}).
As a consequence of these remarks,
the space $T^1(\Gamma)$ can be identified isomorphically (i.e. with equivalent norm)
with $T^1(\Gamma_-)$, or with $T^1(\Gamma_+)$.

The trace $\gamma_{\pm}(\psi)$ for $\psi\in W^1(G\times S\times I)$ is not necessarily in the space $T^1(\Gamma)$.  Hence it is reasonable to define the space
\[
\tilde W^1(G\times S\times I)=\{\psi\in W^1(G\times S\times I)\ |\  \gamma(\psi)\in T^1(\Gamma) \}.
\]
The space $\tilde W^1(G\times S\times I)$ is equipped with the norm
\[
\n{\psi}_{\tilde W^1(G\times S\times I)}=\n{\psi}_{W^1(G\times S\times I)}+ \n{\gamma(\psi)}_{T^1(\Gamma)}.
\]
As the spaces $T^1(\Gamma)$, $T^1(\Gamma_-)$ and $T^1(\Gamma_+)$
are mutually isomorphic, so are also the spaces
$\tilde W^1(G\times S\times I)$, $\tilde W_-^1(G\times S\times I)$ and $\tilde W_+^1(G\times S\times I)$, where
\[
\tilde W_{\pm}^1(G\times S\times I):=\{\psi\in W^1(G\times S\times I)\ |\ \gamma_{\pm}(\psi)\in T^1(\Gamma_{\pm})\},
\]
equipped with the norms defined similarly as $\n{\cdot}_{\tilde{W}^1(G\times S\times I)}$
above.

\begin{proposition}
The spaces $\tilde{W}^1(G\times S\times I)$ and $\tilde{W}_{\pm}^1(G\times S\times I)$ are Banach space.
\end{proposition}

\begin{proof}
We give the proof only for $\tilde{W}^1(G\times S\times I)$,
as the spaces $\tilde{W}_{\pm}^1(G\times S\times I)$ are handled similarly.
If $\{\psi_n\}$ is a Cauchy sequence in $\tilde W^1(G\times S\times I)$,
then $\psi_n\to \psi$ in $W^1(G\times S\times I)$,
$\gamma(\psi_n)\to g$ in $T^1(\Gamma)$,
thus $\gamma_\pm(\psi_n)\to \gamma_\pm(\psi)$
and $\gamma(\psi_n)|_{\Gamma_\pm}\to g|_{\Gamma_\pm}$ in $L^1_\loc(\Gamma_\pm,|\omega\cdot\nu|\diff\sigma\diff\omega\diff E)$,
and so $g=\gamma(\psi)$, which implies that $\psi\in \tilde W^1(G\times S\times I)$.
\end{proof}

For $v\in \mc{D}(\ol G\times S\times I)$ and $u\in\tilde W^1(G\times S\times I)$ one has the following Green's formula
\begin{align}\label{green}
\int_{G\times S\times I}(\omega\cdot \nabla u)v\ dxd\omega dE
+\int_{G\times S\times I}(\omega\cdot \nabla v)u\ dxd\omega dE
=\int_{\partial G\times S\times I}(\omega\cdot \nu) uv d\sigma d\omega dE
\end{align}
which is obtained from the Stokes' Theorem for $u,v\in \mc{D}(\ol G\times S\times I)$, and then by the limiting process for general $v\in \mc{D}(\ol G\times S\times I)$ and $u\in\tilde W^1(G\times S\times I)$.
Similarly, one deduces from Stokes' theorem
the following special case of it that holds for $u\in\tilde{W}^1(G\times S\times I)$,
\begin{align}\label{eq:stokes}
\int_{G\times S\times I} (\omega\cdot \nabla u) dx d\omega dE=\int_{\partial G\times S\times I} u (\omega\cdot \nu)d\sigma d\omega dE.
\end{align}
As a straightforward application of this formula, we have another trace theorem which we shall need later on (cf. \cite{cessenat84} Th\'eor\`eme de trace 2, \cite{cipolatti06} Lemma 2.1, or \cite{dautraylionsv6} Theorem 1, p. 252).

\begin{theorem}\label{th:trace:2}
For $\psi\in \tilde{W}^1_{\pm}(G\times S\times I)$,
we have
\[
\n{\gamma_{\mp}(\psi)}_{T^1(\Gamma_{\mp})}\leq \n{\psi}_{\tilde{W}^1_{\pm}(G\times S\times I)}.
\]
\end{theorem}

\begin{proof}
We consider here the case $\psi\in \tilde{W}_{+}(G\times S\times I)$,
as the other case is handled analogously.
From \eqref{eq:stokes} and the fact that $\Gamma_0$ is zero-measurable, we have
\[
& \n{\gamma_+(\psi)}_{T^1(\Gamma_+)}-\n{\gamma_-(\psi)}_{T^1(\Gamma_-)}
=\int_{\Gamma_+} |\psi| |\omega\cdot\nu| d\sigma d\omega dE-\int_{\Gamma_-} |\psi| |\omega\cdot\nu| d\sigma d\omega dE \\
=&\int_{\Gamma_+} |\psi| (\omega\cdot\nu) d\sigma d\omega dE+\int_{\Gamma_-} |\psi| (\omega\cdot\nu) d\sigma d\omega dE
=\int_{\partial G\times S\times I} |\psi| (\omega\cdot\nu)d\sigma d\omega dE \\
=&\int_{G\times S\times I} \omega\cdot\nabla|\psi| d\sigma d\omega dE,
\]
from which, by taking into account that $\omega\cdot \nabla |\psi|=\mathrm{sgn}(\psi)(\omega\cdot \nabla\psi)$ a.e. on $G\times S\times I$ (cf. the proof of Theorem \ref{th:trace}),
we have
\[
\n{\gamma_-(\psi)}_{T^1(\Gamma_-)}\leq &\n{\gamma_+(\psi)}_{T^1(\Gamma_+)}+\int_{G\times S\times I} \big|\omega\cdot\nabla|\psi|\big| d\sigma d\omega dE \\
\leq & \n{\gamma_+(\psi)}_{T^1(\Gamma_+)}+\n{\omega\cdot \nabla\psi}_{L^1(G\times S\times I)}.
\]
This completes the proof since the right hand side is clearly less than $\n{\psi}_{\tilde{W}^1_{+}(G\times S\times I)}$.
\end{proof}

In particular, we have as sets the equalities
\[
\tilde{W}^1(G\times S\times I)=\tilde{W}^1_{+}(G\times S\times I)=\tilde{W}_{-}^1(G\times S\times I).
\]
Moreover, these spaces are equivalent as normed spaces, with their respective norms introduced above, since for all $\psi\in \tilde{W}^1(G\times S\times I)$,
\[
\frac{1}{2}\n{\psi}_{\tilde{W}^1_{\mp}(G\times S\times I)}\leq \n{\psi}_{\tilde{W}^1_{\pm}(G\times S\times I)}\leq \n{\psi}_{\tilde{W}^1(G\times S\times I)}
\leq 2\n{\psi}_{\tilde{W}^1_{\mp}(G\times S\times I)}.
\]

\begin{remark}
A. Similarly one can define the spaces $W^p(G\times S\times I),\ T^p(\Gamma_-)$ and so on
for any $1\leq p<\infty$.

B. In the above we could replace the space 
$ W^1(G\times S\times I)$ with a more general (weighted) space
\[
W^1_\rho (G\times S\times I)=\{\psi\in L^1(G\times S\times I)\ |\ \rho(\omega,E)\omega\cdot\nabla\psi\in L^1(G\times S\times I)\}
\]
where $\rho$ is a positive measurable function (cf. Section 6 where $\rho=\sqrt{E}$).
Similar note is valid for any $1\leq p<\infty$. 
We omit these generalizations here.
\end{remark}

%%%%%%%%%%%%%%%%%%%%%%%%%%%%%%%%%%%%%%%%%%%%%%%%%% 
\section{On Collision Operators}\label{sec:coll}
%%%%%%%%%%%%%%%%%%%%%%%%%%%%%%%%%%%%%%%%%%%%%%%%%% 

%%%%%%%%%%%%%%%%%%%%%%%%%%%%%%%%%%%%%%%%%%%%%%%%%% 
\subsection{Hyper-singular and (pseudo-)differential nature of certain collision operators}\label{coll-sec1}
%%%%%%%%%%%%%%%%%%%%%%%%%%%%%%%%%%%%%%%%%%%%%%%%%% 

The differential cross-sections
may have singularities, or even hyper-singularities, which would lead to extra partial differential
and pseudo-differential terms in the transport equation (\cite{hsiao}, Sec. 7.1, pp. 353-394).
Instead of explaining systematically the underlying theory, the following slightly informal description suffices for the purposes of this work.

First of all, in the case where $\sigma(x,\omega',\omega,E',E)$ has hyper-singularities (like  M\o ller differential cross 
section given in the below example) the integral $\int_S\int_I$ occurring in the collision operator must be understood in the 
sense of {\it Cauchy principal value} ${\rm p.v.}\int_S\int_I$ or more generally in the sense of {\it Hadamard finite part integral} ${\rm p.f.}\int_S\int_I$ (\cite{hsiao}, Sec. 3.2., \cite{martin-rizzo}, \cite{chan}, \cite{schwartz}, pp. 104-105.
We remark that one encounters this kind of hyper-singularities frequently in physical models.
In addition, we must assume that $E_0>0$ in the energy interval $I=[E_0,E_{\rm m}]$,
because otherwise $K\psi$, for $\psi\in C_0^\infty(G\times S\times I^\circ)$,
might turn out to be (strictly) a distribution, which would increase the complexity of what is presented here.
In \cite{lorence}, p. 7, it is reported that the differential cross sections are not necessarily valid for very small energies which supports this assumption.

Consider the following partial hyper-singular integral operator,
\be\label{coll-1}
(K\psi)(x,\omega,E)={\rm p.f.}\int_{I}\int_{S}\sigma(x,\omega',\omega,E',E)\psi(x,\omega',E')d\omega' dE'.
\ee
The simplest case is where  $\sigma=\sigma_0(x,\omega',\omega,E',E)$ 
is a measurable non-negative function $G\times S\times S\times (I\times I\setminus D)\to\R$,
where $D=\{(E,E)\ |\ E\in I\}$ is the diagonal of $I\times I$,
obeying for $E\neq E'$ the estimates
\bea
&
{\rm esssup}_{(x,\omega)}\int_{S}\sigma_0(x,\omega',\omega,E',E)d\omega'\leq  {C\over{|E-E'|^\kappa}},
\label{coll-3} \\
&
{\rm esssup}_{(x,\omega')}\int_{S}\sigma_0(x,\omega',\omega,E',E)d\omega\leq  {C\over{|E-E'|^\kappa}},
\label{coll-3a}
\eea
where $\kappa<1$, meaning that $\sigma_0(x,\omega',\omega,E',E)$ has a so-called \emph{weak singularity}.
We see that
\bea
&
{\rm esssup}_{(x,\omega,E)\in G\times S\times I}\int_{I}\int_{S}\sigma_0(x,\omega'\omega,E',E)d\omega'dE'
\leq \sup_{E} C\int_{I}{1\over{|E-E'|^\kappa}}dE' \\
&
=
\sup_{E} C{1\over{1-\kappa}}[(E_{\rm m}-E)^{1-\kappa}+(E-E_0)^{1-\kappa}]\leq {{2CE_{\rm m}^{1-\kappa}}\over{1-\kappa}},
\label{coll-4a}
\eea
and similarly for $\int_{I}\int_{S}\sigma_0(x,\omega'\omega,E',E)d\omega dE$. Hence we see 
that $\sigma_0(x,\omega'\omega,E',E)$ satisfies the {\it Schur conditions} given in section \ref{coss2}, and the corresponding collision operator 
\be 
(K_0\psi)(x,\omega,E)
=
\int_{I}\int_{S}\sigma_0(x,\omega',\omega,E',E)\psi(x,\omega',E')d\omega' dE',
\ee
is the usual partial  (singular)  integral operator.
It is bounded $L^2(G\times S\times I)\to L^2(G\times S\times I)$.

Nevertheless, the  collision operator $K$ is not generally of the above form $K_0$. $(E',E)$-dependence in differential cross section 
$\sigma(x,\omega'\omega,E',E)$ may contain hyper-singularities of higher order,
${1\over{(E'-E)}^m}$, for $m=1,2$. For example; see Example \ref{moller} below. 

Moreover, the $(\omega',\omega)$-dependence in differential cross-sections typically contain Dirac's $\delta$-distributions (on $\R$). 
More precisely, in  $\sigma(x,\omega'\omega,E',E)$ there may occur terms
like $\delta(\omega\cdot\omega'-\mu(E',E))$ which require special treatment.
We remark, however that $\delta$-distribution can be approximated by smooth functions $\eta_{\epsilon}\in C_0^\infty(\R)$ in the sense that 
\be 
|\delta(\phi)-\la\eta_\epsilon,\phi\ra_{L^2(\R)}|
\leq
\n{\delta-\eta_\epsilon}_{H^{-1}(\R)}\n{\phi}_{H^1(\R)},\quad \phi\in H^1(\R),
\ee
where $\n{\delta-\eta_\epsilon}_{H^{-1}(\R)}\to 0$ when $\epsilon\to 0^+$.
Typically $\eta_\epsilon$ is chosen to be the convolution $\eta_\epsilon:=\delta\star\theta_\epsilon$,
where $\theta\in C_0^\infty(\R)$ such that $\int_{\R}\theta(x) dx=1$.
Hence we are able to replace $\delta(\omega\cdot\omega'-\mu(E',E))$,
with $\eta_\epsilon(\omega\cdot\omega'-\mu(E',E))$ which is a well-behaved (smooth) function.

We shall see that the cross section $\sigma$ may be the form (e.g. in {M\o ller} electron-electron cross-section)
\begin{multline}
\sigma(x,\omega',\omega,E',E)
=
\chi(E',E)\Big(
{1\over{(E'-E)^2}}\sigma_2(x,\omega',\omega,E',E)\\
+{1\over{E'-E}}\sigma_1(x,\omega',\omega,E',E)+\sigma_0(x,\omega',\omega,E',E)\Big)
\label{coll-2}
\end{multline}
where $\chi(E',E):=\chi_{\R_+}(E-E_0)\chi_{\R_+}(E_m-E)$.
Here each of $\sigma_j(x,\omega',\omega,E',E)$, $j=0,1,2$ may contain the above explained $\delta$-distributions,
and hence they are not necessarily measurable functions.
Denote for $j=0,1,2$,
\[
(\ol {\s K}_j\psi)(x,\omega,E',E):={}&\int_{S}\sigma_j(x,\omega',\omega,E',E)\psi(x,\omega',E') d\omega', \\[2mm]
(\widehat {\s K}_j\psi)(x,\omega,E',E):={}&\chi(E',E)(\ol {\s K}_j\psi)(x,\omega,E',E)
\]
where $\int_{S}$ is understood, if needed, as a distribution.
We find, according to the examples below, that at worst $K$ can be of the
form (this is corresponding the {M\o ller} scattering for electrons)
\begin{multline}
({K}\psi)(x,\omega,E)
=
{\s H}_2\big((\ol {\s K}_2\psi)(x,\omega,\cdot,E)\big)(E) 
\\
+
{\s H}_1\big((\ol {\s K}_1)\psi)(x,\omega,\cdot,E)\big)(E)
+\int_{I}(\widehat {\s K}_0\psi)(x,\omega,E',E) dE',
\label{co-bb}
\end{multline}
where ${\s H}_m$, $m=1,2$, are the {\it Hadamard finite part operators} with respect to $E'$-variable defined by
\[
({\s H}_m u)(E):={\rm p.f.}\int_{E}^{E_m}{1\over{(E'-E)^m}}u(E')dE'.
\]
The expression \eqref{co-bb} is the hyper-singular integral form of $K$.

In \cite{tervo16-up} we verified that \eqref{co-bb} can be equivalently given in the "pseudo-differential form" by 
\bea\label{co-cc}
({K}\psi)(x,\omega,E)
={}&
{\partial\over{\partial E}}\Big(
{\s H}_1\big((\ol{\s K}_2\psi)(x,\omega,\cdot,E)\big)(E)\Big)
-
{\s H}_1\big(({\p {(\ol{\s K}_2\psi)}E}(x,\omega,\cdot,E)\big)(E)
\nonumber\\
{}&
+{\partial\over{\partial E'}}\Big(
(\ol{\s K}_2\psi)(x,\omega,E',E)\Big)_{|E'=E}\nonumber\\
{}&
+
{\s H}_1\big((\widehat {\s K}_1\psi)(x,\omega,\cdot,E)\big)(E)
+\int_{I}({\s K}_0\psi)(x,\omega,E',E) dE'
\eea
where only ${\s H}_1$ appears. This formulation reveals the nature of charged particles' collisions.
Recall that
${\s H}_1$ is well-defined (at least) for all $u\in C^\alpha(I),\ \alpha>0$ and (cf. \cite{chan})
\be\label{c-0-a}
({\s H}_1u)(E)=\int_{E}^{E_m}{{u(E')-u(E)}\over{E'-E}}dE'+
u(E)\ln({E_m-E}) .
\ee
Moreover, it can perhaps be shown that 
${\s H}_1$ is a zero-order pseudo-differential operator (cf. \cite{hsiao}, Chapter 7).

As a conclusion we find that some 
interactions produce the first-order partial derivatives
with respect to energy $E$ combined with the "zero-order" Hadamard part operator.
A closer analysis of the operators $\widehat{\s K}_j$ reveals that in addition, partial derivatives with respect to $\omega$ may appear.
We  demonstrated that in \cite{tervo16-up} for $n=2$. 
The problematic interactions are the electron-electron (considered below)
and positron-positron collisions, and bremsstrahlung (for which see \cite{koch59}, \cite{lorence}).

The operator (\ref{co-cc}) (or equivalently (\ref{co-bb}))  contains two features that require further study: i) The analysis of operators $\widehat {\s K}_j,\ j=0,1,2$, and ii) the analysis of the Hadamard finite part operator ${\s H}_1$ which is a hyper-singular integral operator.
The analysis of the existence of the solutions for the transport problem,
in the case where these operators are included in the transport operator remains to our understanding open.

\begin{remark}
We remark that the operators of the form
\be\label{r-1}
(Pu)(x,E):={}&
{\rm p.f.}\int_{E_0}^{E_m}{{\sigma_0(x,E,E')}\over{E'-E}}u(x,E,E')dE',\quad u\in C_0^\infty(G\times I^\circ\times I^\circ),
\ee
can be treated as in \cite{hsiao}, Chapter 7. Note that in (\ref{r-1}) the integration is over the whole interval $[E_0,E_m]$.
Under relevant criteria the operators (\ref{r-1}) can likely be shown to be pseudo-differential operators.
In particular we recall that the {\it partial Hilbert transform} 
\[
(Hu)(x,E):={}&
{\rm p.f.}\int_{E_0}^{E_m}{{u(x,E,E')}\over{E-E'}}dE'
\]
is a pseudo-differential operator with symbol $-{\rm i}\ {\rm sign}(\xi)$.
Simplifying the operator ${\s H}_1$ is of the form
\[
({\s H}_1u)(x,E):={}&
{\rm p.f.}\int_{E}^{E_m}{{\sigma_0(x,E,E')}\over{E'-E}}u(x,E,E')dE',\quad u\in C_0^\infty(G\times I^\circ\times I^\circ),
\]
The problematic feature in the expression of $({\s H}_1u)(x,E)$ is that the integration is over $[E,E_m]$.
Similar observations concern the operator ${\s H}_2$. 
Consistently to \cite{hsiao}  formal computations suggest that the prospective symbols of ${\s H}_j$ are respectively
\be 
p_1(x,E,\xi)=
{\rm p.f.}\int_E^{\infty}{{\sigma_0(x,E,E')}\over{E'-E}}
e^{{\rm i}(E'-E)\xi}dE'
=
{\rm p.f.}\int_0^{\infty}{{\sigma_0(x,E,E+z)}\over{z}}
e^{{\rm i} z\xi}dz
\ee
and
\be 
p_2(x,E,\xi)=
{\rm p.f.}\int_0^{\infty}{{\sigma_0(x,E,E+z)}\over{z^2}}
e^{{\rm i} z\xi}dz
\ee
The exact analysis of these operators (whether they are zero-order pseudo-differential operators, for example) remains to our knowledge open.
\end{remark}

The next example of the M\o ller-collision operator,
one of the relevant operators in e.g. radiation therapy,
illustrate the above observations.

\begin{example}\label{moller}
{\it Electron-electron scattering - M\o ller}.
We denote the corresponding differential cross section by $\sigma_{22}(x,\omega',\omega,E',E)$. It has a decomposition  (\cite{duclous}, \cite{lorence}, \cite{bomanthesis}, \cite{hensel})
\be\label{i-e-e1} 
\sigma_{22}(x,\omega',\omega,E',E)
=\sigma^p_{22}(x,\omega',\omega,E',E)+\sigma^s_{22}(x,\omega',\omega,E',E).
\ee
where $\sigma^p_{22}(x,\omega',\omega,E',E)$ is corresponding to the (new) primary electrons
and  $\sigma^s_{22}(x,\omega',\omega,E',E)$ is corresponding to the secondary electrons.
In this scattering process the spins have been averaged out,
and the two electrons completely lose their identity.
Therefore, categorizing the electrons as "primary" and "secondary" is simply done
by assigning the electron exiting the scattering event with the higher energy to be the primary one.
The scattering cross section for primary electron $\sigma^p_{22}(x,\omega',\omega,E',E)$ has an expression
\begin{multline}\label{eq:sigma_p_22}
\sigma^p_{22}(x,\omega',\omega,E',E)
=
\sigma_{0}(x){{(E'+1)^2}\over{E'(E'+2)}}\Big({1\over{E^2}}
+{1\over{(E'-E)^2}}+{1\over{(E'+1)^2}} \\
-{{2E'+1}\over{(E'+1)^2E(E'-E)}}\Big)\chi_{22,p}(E',E)\delta(\omega'\cdot\omega-\mu_{22,p}(E',E)),
\end{multline}
where $\sigma_{0}(x)$ depends on the background material, and
\[
& \mu_{22,p}(E',E):=\sqrt{{E(E'+2)}\over{E'(E+2)}} \\
& \chi_{22,p}(E',E):=\chi_{\R_+}(E-E_0)\chi_{\R_+}(E-{E'\over 2})\chi_{\R_+}(E'-E),
\]
while the cross section for the secondary electron $\sigma^s_{22}(x,\omega',\omega,E',E)$ is
\begin{multline*}
\sigma^s_{22}(x,\omega',\omega,E',E)
=
\sigma_0(x){{(E'+1)^2}\over{E'(E'+2)}}\Big({1\over{E^2}}
+{1\over{(E'-E)^2}}+{1\over{(E'+1)^2}}\\
-{{2E'+1}\over{(E'+1)^2E(E'-E)}}\Big)\chi_{22,s}(E',E)\delta(\omega'\cdot\omega-\mu_{22,s}(E',E))
\end{multline*}
where 
\[
& \mu_{22,s}(E',E):=\mu_{22,p}(E',E-E), \\
& \chi_{22,s}(E',E):=\chi_{\R_+}({{E'}\over 2}-E)\chi_{\R_+}(E-E_0).
\]
Since $\sigma^s_{22}(x,\omega',\omega,E',E)=0$ for $E'\leq 2E$ the 
singularities at $E'=E$ do not cause any problems for the secondary electrons.
 
Write
\[
\chi_{22}(E',E)
:={}&
\chi_{22,p}(E',E)+\chi_{22,s}(E',E), \\[2mm]
\mu_{22}(E',E)
:={}&
\begin{cases} \mu_{22,p}(E',E),& E'\leq 2E \\[2mm]
\mu_{22,s}(E',E),& E'\geq 2E \end{cases},  \\[2mm]
\hat\sigma_{22,0}(x,E',E)
:={}&
\sigma_{0}(x){{(E'+1)^2}\over{E'(E'+2)}}\Big({1\over{E^2}}+{1\over{(E'+1)^2}}\Big), \\[2mm]
\hat\sigma_{22,1}(x,E',E)
:={}&
-\sigma_{0}(x){{2E'+1}\over{E'(E'+2)E}}, \\[2mm]
\hat\sigma_{22,2}(x,E',E)
:={}&
\sigma_{0}(x){{(E'+1)^2}\over{E'(E'+2)}}.
\]
Then we find that 
\begin{multline}\label{i-e-e2} 
\sigma_{22}(x,\omega',\omega,E',E)
=
\chi_{22}(E',E)\Big({1\over{(E'-E)^2}}\hat\sigma_{22,2}(x,E',E)\delta(\omega'\cdot\omega-\mu_{22}(E',E))\\
+
{1\over{E'-E}}\hat\sigma_{22,1}(x,E',E)\delta(\omega'\cdot\omega-\mu_{22}(E',E))\\
+
\hat\sigma_{22,0}(x,E',E)\delta(\omega'\cdot\omega-\mu_{22}(E',E))\Big).
\end{multline}

The operators $\ol {\s K}_{22,j}$ are for any $j=0,1,2$,
\bea\label{eq:ol_s_K_22_j}
(\ol {\s K}_{22,j}\psi)(x,\omega,E',E)
=
{}&
\hat\sigma_{22,j}(x,E',E)
\int_{S}
\delta(\omega'\cdot\omega - \mu_{22}(E,E'))\psi(x,\omega',E')d\omega'
\nonumber\\
={}&\hat{\sigma}_{22,j}(x,E',E)\int_{0}^{2\pi}\psi(x,\gamma(s),E')ds,
\eea
where $\gamma=\gamma_{22}(E',E,\omega):[0,2\pi]\to S$
is a parametrization of the curve
\[
\Gamma(E',E,\omega)=\{\omega'\in S\ |\ \omega'\cdot\omega-\mu_{22}(E',E)=0\},
\]
with (constant) speed
\[
\n{\gamma'(s)}=\sqrt{1-\mu_{22}(E',E)^2},\quad s\in [0,2\pi].
\]
For example, we can choose
\be
\gamma(s)=R(\omega)\big(\sqrt{1-\mu_{22}^2}\cos(s),\sqrt{1-\mu_{22}^2}\sin(s),\mu_{22}\big),\quad s\in [0,2\pi],
\ee
where $\mu_{22}=\mu_{22}(E',E)$, and $R(\omega)$ is any rotation matrix which maps the vector $(0,0,1)$ into $\omega$.

When the spatial dimension $n=2$ the operators $\widehat{\s K}_{22,j}$ are simply,
\begin{multline*}
(\widehat {\s K}_{22,j}\psi)(x,\omega,E',E) \\
= 
\hat{\sigma}_{22,j}(x,E',E)\chi_{22}(E',E)\big(\psi(x,\mu_{22}(E',E)\omega+\sqrt{1-\mu_{22}(E',E)^2}\omega^\perp,E')
\\
+
\psi(x,\mu_{22}(E',E)\omega-\sqrt{1-\mu_{22}(E',E)^2}\omega^\perp,E')\big),
\end{multline*}
where $\omega^\perp:=(-\omega_2,\omega_1)$ (the tangent vector of the unit circle $S=S^1$ at $\omega$).

Writing for $j=0,1,2$,
\[
(\widehat {\s K}_{22,j}\psi)(x,\omega,E',E)=\chi_{22}(E',E)
(\ol {\s K}_{22,j}\psi)(x,\omega,E',E),
\]
the collision operator $K_{22}$ decomposes into
\[
K_{22}=K_{22,2}+K_{22,1}+K_{22,0}
\]
where
\be\label{k220}
(K_{22,0}\psi)(x,\omega,E)
=
\int_{I}(\widehat{\s K}_{22,0}\psi)(x,\omega,E',E)
dE'
\ee
and
\bea\label{k221}
(K_{22,1}\psi)(x,\omega,E)
={}&
{\rm p.f.}\int_{I}{{(\widehat{\s K}_{22,1}\psi)(x,\omega,E',E)}\over{E'-E}}
dE' \nonumber \\
={}&
{\rm p.f.}\int_{E}^{E_m}{{(\ol{\s K}_{22,1}\psi)(x,\omega,E',E)}\over{E'-E}}
dE'\nonumber\\
={}&
{\s H}_1((\ol{\s K}_{22,1}\psi)(x,\omega,\cdot,E))(E).
\eea

The operator $K_{22,2}$ gets a hyper-singular form
\bea\label{k222}
(K_{22,2}\psi)(x,\omega,E)
={}&
{\rm p.f.}\int_{I}{{(\widehat{\s K}_{22,2}\psi)(x,\omega,E',E)}\over{(E'-E)^2}}dE' \nonumber\\
={}&
{\rm p.f.}\int_{E}^{E_m}{{(\ol{\s K}_{22,2}\psi)(x,\omega,E',E)}\over{(E'-E)^2}}dE' \nonumber\\
={}&
{\s H}_2((\ol{\s K}_{22,2}\psi)(x,\omega,\cdot,E))(E).
\eea
Hence by \eqref{k220}, \eqref{k221}, \eqref{k222},
\bea\label{k22}
(K_{22}\psi)(x,\omega,E)
={}&
{\s H}_2((\ol{\s K}_{22,2}\psi)(x,\omega,\cdot,E))(E)
+
{\s H}_1((\ol{\s K}_{22,1}\psi)(x,\omega,\cdot,E))(E)
\nonumber\\
{}&
+\int_{I}(\widehat{\s K}_{22,0}\psi)(x,\omega,E',E)
dE'
\eea
which is the hyper-singular integral form of $K_{22}$.
In \cite{tervo16-up} we showed that
we see that
\bea\label{k22-a}
&
{\s H}_2((\ol{\s K}_{22,2}\psi)(x,\omega,\cdot,E))(E)\nonumber\\
={}&
{\partial\over{\partial E}}\Big(
{\rm p.f.}\int_{E}^{E_m}{{(\ol{\s K}_{22,2}\psi)(x,\omega,E',E)}\over{E'-E}}dE'\Big)
-
{\rm p.f.}\int_{E}^{E_m}{1\over{E'-E}}{\p {(\ol{\s K}_{22,2}\psi)}E}(x,\omega,E',E)dE'
\nonumber\\
{}&
+{\partial\over{\partial E'}}\Big(
(\ol{\s K}_{22,2}\psi)(x,\omega,E',E)\Big)_{|E'=E}
\nonumber\\
={}&
{\partial\over{\partial E}}\Big(
{\s H}_1\big((\ol{\s K}_{22,2}\psi)(x,\omega,\cdot,E)\big)(E)\Big)
-
{\s H}_1\Big({\p{(\ol{\s K}_{22,2}\psi)}E}(x,\omega,\cdot,E)\Big)(E)
\nonumber\\
{}&
+{\partial\over{\partial E'}}\Big(
(\ol{\s K}_{22,2}\psi)(x,\omega,E',E)\Big)_{|E'=E},
\eea

Finally, we remark that
the approximative  $ \widehat{\s K}_{22,j}$ for $j=0,1,2$ are
\bea
( \widehat{\s K}_{22,j}\psi)(x,\omega,E)\approx {}&
\int_{S}\hat\sigma_j(x,E',E)\chi_{22}(E',E)
\eta_\epsilon(\omega'\cdot\omega-\mu_{p}(E,E'))\psi(x,\omega',E')d\omega'\nonumber\\
=:&
\int_{S}\tilde\sigma_j(x,\omega',\omega,E',E)\psi(x,\omega',E')d\omega'
=:(\widetilde {\s K}_{22,j}\psi)(x,\omega,E',E).
\eea
Note that 
$\widetilde K_{22,0}$ is the usual partial Schur integral operator.
The approximations  $\widetilde {\s K}_{22,j}$ are useful from theoretical and practical point of view.
\end{example}

\begin{remark}
In \cite{tervo16-up} we computed  further some of the  terms appearing in the above example,
limiting ourselves to the case for $n=2$. As we mentioned above, in this case
\begin{multline}\label{n2k-1}
(\ol {\s K}_{22,2}\psi)(x,\omega,E',E) \\
=
\hat{\sigma}_{22,2}(x,E',E)\big(\psi(x,\mu_{22}(E',E)\omega+\sqrt{1-\mu_{22}(E',E)^2}\omega^\perp,E')
\\
+
\psi(x,\mu_{22}(E',E)\omega-\sqrt{1-\mu_{22}(E',E)^2}\omega^\perp,E')\big).
\end{multline}
It was found that
the {M\o ller} collision term produces first order partial differential terms with respect to $E$, combined with the Hadamard finite part operator,
and, in addition, it produces terms containing $\nabla_\omega$, i.e. angular derivatives.

The exact form of {M\o ller} collision operator
allows accessing relevant approximation schemes for which the error analysis 
can be carried out. In \cite{tervo16-up} we derived the CSDA-BTE-type approximation,
which however, does not take into account the change of
angle for the (new) primary electron during transport,
since the angular derivative ($\nabla_\omega$) is missing from it. 
On the other hand, CSDA-Focker-Plank approximation 
contains also second order partial derivatives (with respect to angle)
which do not occur in the results of \cite{tervo16-up}.
\end{remark}

As a conclusion, we find that the complete Boltzmann operator
in its exact, general form is given by
\bea\label{T}
T\psi=&{}
-{\partial\over{\partial E}}\Big(
{\s H}_1((\ol{\s K}_2\psi)(x,\omega,\cdot,E))(E)\Big)
+
{\s H}_1(({\p {(\ol{\s K}_2\psi)}E}(x,\omega,\cdot,E))(E)\nonumber\\
&{}
-{\partial\over{\partial E'}}\Big(
(\ol{\s K}_2\psi)(x,\omega,E',E)\Big)_{|E'=E}+
\omega\cdot\nabla_x\psi+F\cdot\nabla_\omega\psi
\nonumber\\
&{}
-
{\s H}_1((\ol {\s K}_1\psi)(x,\omega,\cdot,E))(E)+\Sigma\psi-K_0\psi
\eea
where
\be\label{K}
K_0\psi=
\int_{I}({\s K}_0\psi)(x,\omega,E',E) dE'
\int_{S}\sigma_{00}(x,\omega',\omega,E)\psi(x,\omega',E) d\omega'.
\ee
In addition, the real, physical model is a coupled system $T=(T_1,T_2,T_3)$ of the operators like (\ref{T}) and some terms may 
be missing in $T_j$ (cf. the system considered in this paper and in \cite{tervo16-up}).
The existence and uniqueness properties for the exact transport equation $T\psi=f$
with the given inflow boundary and initial conditions remain to be analysed.
Potential methods are  Lions-Lax-Milgram Theorem, theory of maximally dissipative operators (as in this paper) and theory of evolution operators.
In \cite{tervo16-up} we have applied these methods to the CSDA-BTE-problem.

It is also important to understand regularity of solutions of BTE in the mixed-norm (anisotropic) Sobolev-Slobodevskij spaces.
This is needed e.g. in approximation analysis and, in particular, in numerical analysis (e.g. FEM).
In existence, uniqueness and regularity analysis the above mentioned pseudo-differential form-like expressions of collision operators  might be useful.

\begin{remark}\label{tuned FEM}
At least in existence and uniqueness analysis of solutions,
it is more fruitful to use the partial differential (pseudo-differential) form  of the exact transport equation.  
Nevertheless, the numerical methods may apply directly the hyper-singular partial integral equation (\ref{co-bb}).
For instance, the Galerkin (discontinuous) finite element methods (FEM) are able to consider hyper-singular partial integral terms.  
These techniques are well-known e.g. in field of boundary element methods (BEM) where the hyper-singular
integral kernels are emerging from single and double layer potentials.
We remark that carefully chosen (special) numerical integration schemes and the choice of bases functions are needed in computing element matrices for hyper-singular integral operators.
\end{remark}

%%%%%%%%%%%%%%%%%%%%%%%%%%%%%%%%%%%%%%%%%%%%%%%%%%%%%%%%%%%%%%
\subsection{On the Choice of More General Radon Measures in the Definition of Collision Operators}\label{coll-sec2}
%%%%%%%%%%%%%%%%%%%%%%%%%%%%%%%%%%%%%%%%%%%%%%%%%%%%%%%%%%%%%%

We give the following computations regarding a more general
measures $\rho_I$ and/or $\rho_S$ in the definition of $K$, instead of typical  measures $\mc{L}^1$ and/or  $\mu_S$.

As mentioned in Remark \ref{pre-re1},
for elastic scattering $K$ can be written in the form
\bea\label{r-1-a}
(K\psi)(x,\omega,E)={}&\int_{S}\int_I\sigma(x,\omega',\omega,E)\psi(x,\omega',E') d\rho_I(E'|E)d\omega'\nonumber\\
={}&\int_{S}\sigma(x,\omega',\omega,E)\psi(x,\omega',E)d\omega',
\eea
where for every $E$, and Borel set $A\subset I$,
$\rho_I(A|E)=1$ if $E\in A$ and $\rho_I(A|E)=0$ otherwise.
This is the first motivation of more general measures being used in the definition of the collision operator $K$.

The second one can be depicted as follows. 
On $I$ we use the Lebesgue measure $dE$.
The above example \ref{moller} (see also Example 2.27 and Remark 2.26 in \cite{tervo16-up})
shows that for some interactions, the collision operator $K$ may be of the form
\bea\label{r-1-aa}
(K\psi)(x,\omega,E)
={}&
\int_{I}
\chi(E',E)\hat\sigma(x,E',E)
\int_{0}^{2\pi}\psi(x,\gamma(E',E,\omega)(s),E')ds\nonumber\\
={}&
\int_{I}
\chi(E',E)\hat\sigma(x,E',E)
\frac{1}{\sqrt{1-\mu(E',E)^2}}\int_{\Gamma(E',E,\omega)} \psi(x,\gamma,E')d\ell(\gamma).
\eea
Here $\chi(E',E)$ is a product of characteristic functions, 
the integral $\int_{\Gamma(E',E,\omega)} (\cdots)d\ell$ is the path integral along the curve
\[
\Gamma(E',E,\omega)=\{\omega'\in S\ |\ \omega'\cdot\omega-\mu(E',E)=0\}.
\]
We can further write $K$ as
\bea\label{r-2}
(K\psi)(x,\omega,E)
=
\int_I \int_{S}\chi(E',E)\hat\sigma(x,E',E)\psi(x,\omega',E')d\rho_S\big(\omega'|(\omega,\mu(E',E))\big)dE',
\eea
where for all $-1<\mu<1$ and $\omega\in S$,
the quantity $\rho_S(\cdot|(\omega,\mu))$ is the Radon measure defined
on Borel sets $A\subset S$ by
\[
\rho_S(A|(\omega,\mu)):=\frac{\mc{H}^1(A\cap \Gamma_{(\omega,\mu)})}{\sqrt{1-\mu^2}},
\]
the measure $\mc{H}^1$ being the 1-dimensional Hausdorff measure on $S$ (see \cite[pp. 7--10]{falconer86}), and
\[
\Gamma_{(\omega,\mu)}:=\{\omega\in S\ |\ \omega'\cdot\omega-\mu=0\}.
\]
Notice that $\Gamma(E',E,\omega)=\Gamma_{(\omega,\mu(E',E))}$.

%%%%%%%%%%%%%%%%%%%%%%%%%%%%%%%%%%%%%%%%%%%%%%%%%%%%%%%%%%%%%%%%%%%%%%%%%%%%%%%%%%%%%%%%%%%%
\subsection{On the Boundedness of $K$ and Dissipativity of $\Sigma-K$ for a Single Collision Operator Containing More General Measures}\label{gen-bound-k}
%%%%%%%%%%%%%%%%%%%%%%%%%%%%%%%%%%%%%%%%%%%%%%%%%%%%%%%%%%%%%%%%%%%%%%%%%%%%%%%%%%%%%%%%%%%%

In the cases where $K$ contains more general  measures the boundedness of $K$ and the dissipativity of $\Sigma-K$ in $L^p(G\times S\times I)$-spaces can be shown in the similar fashion as in the below Theorems \ref{NoLabel} and \ref{dfsco} for Lebesgue based measures.
For example, when $K$ is of the form (\ref{r-1-a}) we get for the boundedness of $K$ (here $L^p(G\times S\times I)=L^p(G\times S\times I,dx d\omega dE)$ and in formulations we restrict ourselves to the cases $p=1,2$ only):

Suppose that ${\sigma}:G\times S^2\times I\to\R$ is a measurable non-negative function. Then we have (cf. \cite{halmos}, p. 20, \cite{dautraylionsv6}, pp. 227-228 and the below Theorem \ref{NoLabel})

\begin{enumerate}
\item
Suppose that 
\be\label{s-1}
\int_{S}
{\sigma}(x,\omega,\omega',E)d\omega'
\leq M<\infty\quad {\rm a.e.}\ (x,\omega,E)\in G\times S\times I.
\ee
Then $K:L^1(G\times S\times I)\to L^1(G\times S\times I)$ is bounded and $\n{K}\leq M$.

\item
Suppose that 
\bea\label{s-2}
&
\int_{S}
{\sigma}(x,\omega',\omega,E)d\omega'
\leq M_1<\infty\quad {\rm a.e.}\ (x,\omega,E)\in G\times S\times I,
\nonumber\\
&
\int_{S}
{\sigma}(x,\omega,\omega',E)d\omega'
\leq M_2<\infty\quad {\rm a.e.}\ (x,\omega,E)\in G\times S\times I.
\eea
Then $K:L^2(G\times S\times I)\to L^2(G\times S\times I)$ is bounded and $\n{K}\leq \sqrt{M_1M_2}$.

\end{enumerate}

The conditions (\ref{s-1}), (\ref{s-2}) are called Schur conditions for the boundedness and we call the corresponding collision operator {\it Schur partial integral operators}.

For the dissipativity of $\Sigma-K$, in this case, we have the following.
Suppose that ${\sigma}:G\times S^2\times I\to\R$ is a measurable non-negative function and that $\Sigma\in L^\infty(G\times S\times I)$. Then we have (cf. \cite{dautraylionsv6}, p. 241 and the below Theorem  \ref{dfsco})

\begin{enumerate}
\item
Suppose that, for a.e. $(x,\omega,E)\in G\times S\times I$,
\be\label{d-1}
&
\Sigma(x,\omega,E)-
\int_{S}
\sigma(x,\omega,\omega',E) d\omega'\geq c\geq 0.
\ee
Then for all $\psi\in L^1(G\times S\times I)$ and $\lambda>0$,
\be\label{r-6}
\n{\big(\lambda I-(-\Sigma+K+cI)\big)\psi}_{L^1(G\times S\times I)}\geq \lambda\n{\psi}_{L^1(G\times S\times I)}.
\ee

\item
Suppose that, for a.e. $(x,\omega,E)\in G\times S\times I$,
\bea\label{r-9}
\Sigma(x,\omega,E)-
\int_{S}
\sigma(x,\omega',\omega,E) d\omega' &\geq c\geq 0,
\nonumber\\
\Sigma(x,\omega,E)-
\int_{S}
\sigma(x,\omega,\omega',E) d\omega' &\geq c\geq 0.
\eea
Then for all $\psi\in L^2(G\times S\times I)$,
\be\label{r-7}
\la (\Sigma-K-cI)\psi,\psi\ra_{L^2(G\times S\times I)}\geq 0.
\ee
\end{enumerate}

Additionally, we will deal below with
the boundedness of $K$ and dissipativity of $\Sigma-K$ in the case of single collision operator is of the form (\ref{r-1-aa}).
For simplicity we restrict ourselves to the {\it case $n=2$}, but similar analysis can be performed for $n=3$, or even for general $n$.

%%%%%%%%%%%%%%%%%%%%%%%%%%%%%%%%%%%%%%%%%%%%%%%%%%%%%%%%%%%%%%%%
\subsubsection{Boundedness}
%%%%%%%%%%%%%%%%%%%%%%%%%%%%%%%%%%%%%%%%%%%%%%%%%%%%%%%%%%%%%%%%

Suppose that the spatial dimension $n=2$. In this case $\hat {\s K}$ is of the form
\begin{multline}\label{re-b-k-1}
(\hat {\s K}\psi)(x,\omega,E',E) \\
=
\hat{\sigma}(x,E',E)\chi(E',E)\big(\psi(x,\mu(E',E)\omega+\sqrt{1-\mu(E',E)^2}\omega^\perp,E') \\
+
\psi(x,\mu(E',E)\omega-\sqrt{1-\mu(E',E)^2}\omega^\perp,E')\big),
\end{multline}
where $\omega^\perp:=(-\omega_2,\omega_1)$ (the tangent vector of the unit circle $S=S_1$ at $\omega$).
Hence the collision operator (\ref{r-2}) $K$ is
\bea\label{re-b-k-2}
&
(K\psi)(x,\omega,E)
=\int_{I}
(\hat {\s K}\psi)(x,\omega,E',E)dE'\nonumber \\
={}&
\int_{I}\hat{\sigma}(x,E',E)\chi(E',E)\psi(x,\mu(E',E)\omega+\sqrt{1-\mu(E',E)^2}\omega^\perp,E')dE'
\nonumber\\
&
+
\int_{I}\hat{\sigma}(x,E',E)\chi(E',E)
\psi(x,\mu(E',E)\omega-\sqrt{1-\mu(E',E)^2}\omega^\perp,E')dE'\nonumber\\
=:{}&
(K_+\psi)(x,\omega,E)+(K_-\psi)(x,\omega,E).
\eea

Consider the operator $K_+$. We find that
\bea\label{re-b-k-3}
&
\n{K_+\psi}_{L^1(G\times S\times I)}\nonumber\\
\leq{}&
\int_G\int_S\int_I\int_{I}
\hat{\sigma}(x,E',E)\chi(E',E)|\psi(x,\mu(E',E)\omega+\sqrt{1-\mu(E',E)^2}\omega^\perp,E')|dE' dE d\omega dx\nonumber\\
={}&
\int_G\int_I\int_{I}
\hat{\sigma}(x,E',E)\chi(E',E)\Big(\int_S|\psi(x,\mu(E',E)\omega+\sqrt{1-\mu(E',E)^2}\omega^\perp,E')|d\omega\Big)dE' dE dx.
\eea
We see that for any fixed $E,\ E'\in I$ the mapping $h:S\to S$ defined by
\[
h(\omega):={}&\mu(E',E)\omega+\sqrt{1-\mu(E',E)^2}\omega^\perp \\
={}&\qmatrix{\mu(E',E) &-\sqrt{1-\mu(E',E)^2}\\ \sqrt{1-\mu(E',E)^2} &\mu(E',E)\\}
\qmatrix{\omega_1\\ \omega_2}=:\omega''
\]
is a diffeomorphism, and
\[
h^{-1}(\omega'')=\mu(E',E)\omega''-\sqrt{1-\mu(E',E)^2}(\omega'')^\perp.
\]
Furthermore, 
we find that  $h:S\to S$ (and $h^{-1}$) is  isometric. Indeed,
for all $v\in T_\omega(S)$,
\[
dh_\omega(v)=((dh_1)_\omega(v),(dh_1)_\omega(v))
=(\la\nabla h_1(\omega),v\ra,\la\nabla h_2(\omega),v\ra
=H_\mu(\omega)v=H_\mu v
\]
where $\mu=\mu(E',E)$ and
\[
H_\mu:=\qmatrix{\mu &-\sqrt{1-\mu^2}\\ \sqrt{1-\mu^2} &\mu\\}.
\]
Let $G_{\omega}(v,v')=G(v,v')=\la v,v'\ra$ be the Riemannian metric on $S$ induced by the Euclidean metric on $\R^2$,
and let $h^*G(\cdot,\cdot)$ be the pull-back of $G$ onto $S$ along $h$.
Then
\bea
&
(h^*G)(v,v')=G(h_*v, h_*v')=G(dh_\omega(v), dh_\omega(v'))
=G(H_\mu v,H_\mu v')\nonumber\\
={}&
\la H_\mu v,H_\mu v'\ra=\la v, H_\mu^T H_\mu v'\ra
=\la v,v'\ra
=G(v,v')
\eea
as desired. 
The Change of Variables Theorem implies (\cite{lee03}) that
\be\label{c-v}
\int_{S} (f\circ h)d\mu_G=\int_{S} (f\circ h)d(h^*\mu_G)=\int_{S} fd\mu_{G}
\ee
where $\mu_G$ is the Riemannian measure on $S$ induced by $G$.

Using the formula (\ref{c-v}) we obtain
\bea
\int_S|\psi(x,\mu(E',E)\omega+\sqrt{1-\mu(E',E)^2}\omega^\perp,E')|d\omega
={}&\int_{S}|(\psi\circ h)(x,\omega,E')| d\omega
\nonumber\\
={}&
\int_{S}|\psi(x,\omega'',E')|d\omega''. \label{re-b-k-4}
\eea

Combining (\ref{re-b-k-3}) and (\ref{re-b-k-4}) we find that
\bea\label{re-b-k-4b}
\n{K_+\psi}_{L^1(G\times S\times I)}
\leq{}&
\int_G\int_I\int_{I}
\hat{\sigma}(x,E',E)\chi(E',E)\Big(\int_{S}|\psi(x,\omega'',E')|d\omega''\Big)dE' dE dx\nonumber\\
\leq{}&
\Big(\sup_{(x,E')\in G\times I}\int_I\hat{\sigma}(x,E',E)\chi(E',E)dE\Big)\n{\psi}_{L^1(G\times S\times I)},
\eea
implying that $K_+:L^1(G\times S\times I)\to L^1(G\times S\times I)$ is bounded.
Similarly one can show that $K_-:L^1(G\times S\times I)\to L^1(G\times S\times I)$ is bounded as well, and 
\be\label{re-b-k-4a}
\n{K_-\psi}_{L^1(G\times S\times I)}
\leq
\Big(\sup_{(x,E')\in G\times I}\int_I\hat{\sigma}(x,E',E)\chi(E',E)dE\Big)\n{\psi}_{L^1(G\times S\times I)}.
\ee
Hence $K:L^1(G\times S\times I)\to L^1(G\times S\times I)$ is bounded and 
\be\label{re-b-k-5}
\n{K}\leq 
2\sup_{(x,E')\in G\times I}\int_I\hat{\sigma}(x,E',E)\chi(E',E)dE.
\ee

When $p=2$, one obtains analogous boundedness criteria by applying H\"older inequality.
In this case we obtain
\be\label{re-b-k-6}
\n{K}\leq 2\sqrt{M_1M_2}
\ee
where
\bea\label{re-b-k-7}
M_1:={}&\sup_{(x,E')\in G\times I}\int_I\hat{\sigma}(x,E',E)\chi(E',E)dE,\nonumber\\
M_2:={}&\sup_{(x,E)\in G\times I}\int_{I}\hat{\sigma}(x,E',E)\chi(E',E)dE'.
\eea

%%%%%%%%%%%%%%%%%%%%%%%%%%%%%%%%%%%%%%%%%%%%%%%%%%%%%%%%
\subsubsection{Dissipativity}
%%%%%%%%%%%%%%%%%%%%%%%%%%%%%%%%%%%%%%%%%%%%%%%%%%%%%%%%

Consider the dissipativity of $cI-(\Sigma-K):L^1(G\times S\times I)\to L^1(G\times S\times I)$. Suppose that for a.e. $(x,\omega,E)$ we have
\be\label{re-b-k-10}
\Sigma(x,\omega,E)
-
2\int_I\hat{\sigma}(x,E,E')\chi(E,E')dE'\geq c\geq 0.
\ee
Using \eqref{re-b-k-4} we find that
\be\label{re-b-k-5:2}
\int_S|(K\psi)(x,\omega,E)| d\omega \leq
2\int_{I}\int_{S}\hat{\sigma}(x,E',E)\chi(E',E)|\psi(x,\omega'',E')|d\omega'' dE',
\ee
and so for any $\lambda>0$
\bea\label{re-b-k-8}
&
\n{(\lambda I-(cI-(\Sigma-K))\psi}_{L^1(G\times S\times I)} \\
={}&
\int_G\int_S\int_I\Big|(\lambda-c+\Sigma(x,\omega,E))\psi(x,\omega,E)-
(K\psi)(x,\omega,E)\Big|dE d\omega dx\nonumber\\
\geq {}&
\int_G\int_S\int_I\Big((\lambda-c+\Sigma(x,\omega,E))|\psi(x,\omega,E)|-
|(K\psi)(x,\omega,E)|\Big)dE d\omega dx
\nonumber\\
\geq{}&
\int_G\int_S\int_I(\lambda-c+\Sigma(x,\omega,E))|\psi(x,\omega,E)|dE d\omega dx\nonumber\\
{}&
-
\Big(2\int_G\int_I\int_{I}\int_{S}\hat{\sigma}(x,E',E)\chi(E',E)dE\Big)|\psi(x,\omega'',E')|d\omega'' dE' dE dx
\nonumber\\
={}&
\int_G\int_S\int_I\Big(\lambda-c+\Sigma(x,\omega,E)
-
2\int_I\hat{\sigma}(x,E,E')\chi(E,E')dE'\Big)|\psi(x,\omega,E)|dE d\omega dx\nonumber\\
\geq{}& 
\lambda \n{\psi}_{L^1(G\times S\times I)},\nonumber
\eea
where we used the inequality \eqref{re-b-k-5:2} and the fact that $\lambda-c+\Sigma(x,\omega,E)\geq 0$ for a.e. $(x,\omega,E)\in G\times S\times I$.

For $p=2$, an analogous dissipativity result is implied by the assumptions
that for a.e. $(x,\omega,E)\in G\times S\times I$,
\bea\label{re-b-k-11}
&
\Sigma(x,\omega,E)
-
2\int_I\hat{\sigma}(x,E,E')\chi(E,E')dE'
\geq c\geq 0,
\nonumber\\
&
\Sigma(x,\omega,E)
-
2\int_{I}\hat{\sigma}(x,E',E)\chi(E',E)dE'
\geq c\geq 0,
\eea
in combination with the H\"older inequality.

For $n=3$ and for the coupled system (cf. section \ref{coss2}) the boundedness of $K$ and the 
dissipativity of $cI-(\Sigma-K)$ in spaces $L^1(G\times S\times I)$ and $L^2(G\times S\times I)$ (and more generally in spaces $L^p(G\times S\times I)$, $1\leq p<\infty$) can be
derived in a similar fashion, from analogous assumptions.
The details will be given in a future work.

In this paper we 
shall for simplicity assume that the single collision operator is of the form
\be\label{tassaK}
(K\psi)(x,\omega,E)=\int_S\int_I\sigma(x,\omega',\omega,E',E)\psi(x,\omega',E')dE' d\omega'.
\ee
The inclusion of other types of collision operators (\ref{r-1-a}), (\ref{r-1-aa}) in the coupled system is a technicality and we omit it.

%%%%%%%%%%%%%%%%%%%%%%%%%%%%%%%%%%%%%%%%%%%%%%%%%%%%%%%%%%%%%%%%%%%%%%%%
\section{Solving the Convection Equation by the Method of Characteristics}\label{ls2}
%%%%%%%%%%%%%%%%%%%%%%%%%%%%%%%%%%%%%%%%%%%%%%%%%%%%%%%%%%%%%%%%%%%%%%%%

%%%%%%%%%%%%%%%%%%%%%%%%%%%%%%%%%%%%%%%%%%%%%%%%%%%%%%%%%%%%%%%%%%%%%%%%
\subsection{On the Escape Time Map}\label{lss2}
%%%%%%%%%%%%%%%%%%%%%%%%%%%%%%%%%%%%%%%%%%%%%%%%%%%%%%%%%%%%%%%%%%%%%%%%

In the following we need the concept of "escape time" $t(x,\omega)$ where $x\in G$ and $\omega\in S$. We define for $(x,\omega)\in G\times S$,
\be\label{eq:escape_time}
t(x,\omega)=t_-(x,\omega):=&\inf\{s>0\ |\ x-s\omega\not\in G\}\\ \nonumber
=&\sup\{t>0\ |\ x-s\omega\in G\ {\rm for\ all}\ 0< s <t\}.
\ee
For some simple cases this mapping $t$ can be given explicitly.

We give a simple example in which $t$ can be computed explicitly.

\begin{example}\label{le1}
Let $G$ be the ball $B(0,r)\subset\R^3$. Suppose that $x\in G$.
We find that the point $y=x-s\omega$ belongs to $\partial G$ exactly when $\n{x-s\omega}=r$. This means that
\be\label{l9}
\n{x}^2-2s(x\cdot\omega) + s^2=r^2.
\ee
The solution of (\ref{l9}) is
\[
s=x\cdot\omega\pm \sqrt{(x\cdot\omega)^2+r^2-\n{x}^2}.
\]
Since $t(x,\omega)$ is positive, we have
\[
t(x,\omega)=x\cdot\omega + \sqrt{(x\cdot\omega)^2+r^2-\n{x}^2}.
\]
Note that the discriminant appearing in the expression of $t(x,\omega)$    is always positive for $x\in G$. Hence $t\in C^\infty( G\times S)$. 
We also remark that for $t(x.\omega)$ is defined for $y\in\partial G$ and $t(y,\omega)=0,\ y\in\partial G$. Hence we see that $t\in C(\ol G\times S)$.
\end{example}

As in Example \ref{le1} one sees generally that
\be\label{l10}
t(x,\omega)=x\cdot\omega\pm \sqrt{(x\cdot\omega)^2+\n{y}^2-\n{x}^2}
\ee
where $y=x-t(x,\omega)\omega\in\partial G$. By elementary geometric considerations one finds that the discriminant in the expression (\ref{l10}) is nonnegative for $x\in G$.
In the case when $G$ is convex (and bounded) and when it has $C^1$-boundary the mapping $t:\ol G\times S\to [0,\infty[$ is continuous and it has continuous partial derivatives ${\p t{x_j}}$ in $G\times S$ (see Prop. \ref{tp} below).
In the case where $G$ is convex $t(x,\omega)$ is the unique number $s$ such that
$y=x-t(x,\omega)\omega\in \partial G$.

We record here a simple general lemma.

\begin{lemma}\label{le:escape_time}
\begin{itemize}
\item[(i)] For all $(x,\omega)\in G\times S$, one has $x-t(x,\omega)\omega\in \partial G$.
\item[(ii)]
For every $(x,\omega)\in G\times S$ for which $y:=x-t(x,\omega)\omega$
is a regular point of $\partial G$, it holds $\omega\cdot \nu(y)\leq 0$.
\end{itemize}
\end{lemma}

\begin{proof}
(i)
By the first line \eqref{eq:escape_time}, one can choose
a decreasing sequence of positive numbers $(s_n)$ such that $s_n\to t(x,\omega)$
and $x-s_n\omega\notin G$. As $G$ is open, one has in the limit that $x-t(x,\omega)\omega\notin G$.
Similarly, choosing an increasing sequence $(t_n)$ of strictly positive numbers such that $t_n\to t(x,\omega)$,
by the second line in \eqref{eq:escape_time} we have $x-t_n\omega\in G$,
and therefore $x-t(x,\omega)\omega\in\ol{G}$.

(ii) We argue by contradiction.
Suppose that $y=x-t(x,\omega)\omega$ is a regular point of $\partial G$
but $\omega\cdot\nu(y)>0$.
Then as $\nu(y)$ points outward from $G$, we necessarily have $y+\tau \omega\notin G$ for all small enough $\tau>0$,
i.e. $x-(t(x,\omega)-\tau) \omega\notin G$, which contradicts
the above ($\inf$-)definition of $t(x,\omega)=t_-(x,\omega)$.
\end{proof}

For the needs of the next proposition,
we recall the concept of lower and upper semi-continuity
of a mapping, as well as a standard result regarding the
existence of supports for a convex (open) subset $C$ of $\R^n$.

Let $X$ be a metric space and let  $f:A\to \R, A\subset X$ be a mapping. Recall that $f$ is lower semi-continuous at $x_0\in A$ if
\[
\liminf_{x\to x_0}f(x)\geq f(x_0).
\]
Similarly $f$ is upper semi-continuous at $x_0\in A$ if
\[
\limsup_{x\to x_0}f(x)\leq f(x_0).
\]

\begin{proposition}\label{pr:support}
Let $C\subset\R^n$ be an open convex set and let $y\in\partial C$.
Then there exists a $\lambda=\lambda_{y}\in\R^n$ such that 
$\lambda\cdot (x-y)<0$ for all $x\in C$.

We call this $\lambda\in\R^n$ a \emph{support} of $C$ at $y\in\partial C$.
\end{proposition}

\begin{proof}
We will prove here a bit more than stated above, namely
that for every $y\notin C$ (which is the case if $y\in\partial C$),
there exist $\lambda$ such that $\lambda\cdot (x-y)<0$ for all $x\in C$.

Observe first that it is enough to prove this in the case $y=0$,
since otherwise one can simply replace $C$ with $C-y=\{x-y\ |\ x\in C\}$.

In the first place, we assume that $0\notin \ol{C}$.
As $\ol{C}$ is closed, there exists $y\in \ol{C}$ such that $\n{y}=\min\{\n{x}\ |\ x\in \ol{C}\}$.
If $x\in C$ we have by convexity $sx+(1-s)y\in\ol{C}$ for all $s\in [0,1]$, from which
\[
\n{y}^2\leq &\n{sx+(1-s)y)}^2
=s^2\n{x}^2+(1-s)^2\n{y}^2+2s(1-s)x\cdot y,
\]
i.e.
\[
2s(1-s)x\cdot (-y)\leq s^2(\n{x}^2+\n{y}^2)-2s\n{y}^2.
\]
Setting $\lambda:=-y$ and letting $s>0$ tend to zero,
we get
\[
x\cdot\lambda \leq \lim_{s\to 0^+} \frac{s(\n{x}^2+\n{y}^2)-2\n{y}^2}{2(1-s)}=-\n{\lambda}^2.
\]
Because $0\notin\ol{C}$, we have $\lambda=-y\neq 0$ and thus $x\cdot\lambda<0$.
Since $x\in C$ was arbitrary, this proves the claim in this special case, i.e. when $0\notin\ol{C}$.

It remains to consider the case where $0\in\partial C$.
Let $(x_n)$ be a sequence in the complement of $\ol{C}$ that converges to $0$
when $n\to\infty$.
Choose a sequence $(\lambda_n)$ in $\R^n$, corresponding to $x_n$ as above,
such that $\lambda_n\cdot x\leq 0$ for all $x\in C$.
We may normalize these vectors $\lambda_n$ by replacing them by $\frac{\lambda_n}{\n{\lambda_n}}$
i.e. assume that $\n{\lambda_n}=1$.
But then they lie on the compact unit sphere $S^{n-1}$ of $\R^{n}$,
and thus we may extract a subsequence $(\lambda_{n_i})$ converging to $\lambda\in S^{n-1}$
as $i\to\infty$.
If $x\in C$, then $\lambda_{n_i}\cdot x\leq 0$ for all $i$ implies that
$\lambda\cdot x\leq 0$ for all $x\in C$ as well.

It remains to show that the above inequality is actually a strict inequality.
Indeed, given $x\in C$,
the assumption that $C$ is open, implies that $x+\epsilon \lambda\in C$
for a small enough $\epsilon>0$.
But then by what we just proved, $\lambda\cdot (x+\epsilon\lambda)\leq 0$,
i.e.
\[
\lambda\cdot x\leq -\epsilon\n{\lambda}^2,\quad \forall x\in C,
\]
which, since $\epsilon>0$ and $\n{\lambda}^2>0$
implies that $\lambda\cdot x<0$ as claimed.
This completes the proof.
\end{proof}

\begin{remark}
As is well known, above proposition is actually true in any (possibly infinite dimensional) Hilbert space.
Moreover, it can be verified using the above proof with the exception
that one must recall that (a) the existence of (a unique) $y$ such that $\ol{y}=\{\n{x}\ |\ x\in\ol{C}\}$
remains true in this more general setting, because $\ol{C}$ is convex and closed,
and (b) that  $\n{\lambda_n}=1$ for all $n$
implies that a subsequence $\lambda_{n_i}$ converges \emph{weakly} to some $\lambda$.
Alternatively, it can be seen as a simple corollary of the Hahn-Banach theorem and Riesz representation theorems.
However, here we do not need this result in such a generality.
\end{remark}

In the next two propositions we formulate the basic continuity and differentiability properties
of the escape time $t$, respectively.

\begin{proposition}\label{pr:escape_time}
The escape time $t(x,\omega)$ has the following properties:
\begin{itemize}
\item[(i)] Function $t(x,\omega)$ is lower semi-continuous on $G\times S$.
\item[(ii)] Function $t(x,\omega)$ is continuous on $G\times S$
if and only if $G$ is convex.
In addition, in this case, for all $(x_0,\omega_0)\in G\times S$
if $y_0=x_0-t(x_0,\omega_0)\omega_0$, we have
\be\label{t0}
\lim_{(x,\omega)\to (y_0,\omega_0)} t(x,\omega)=0.
\ee
\end{itemize}
\end{proposition}

\begin{proof}
(i) We assume that $t(x,\omega)$ is not lower semi-continuous,
and show that this leads to a contradiction, thus proving the claim.
Indeed, if this is the case, there is a sequence $(x_n,\omega_n)$ in $G\times S$
converging to $(x,\omega)\in G\times S$
and $\ol{t}>0$ such that $t(x,\omega)>\ol t\geq \liminf_{n\to\infty} t(x_n,\omega_n)$.
Since $t(x_n,\omega_n)$ belongs to a bounded set $[0,\ol t]$,
there is a subsequence of $(x_n,\omega_n)$, which we still denote by $(x_n,\omega_n)$, such that $t(x_n,\omega_n)$
converges to a number $t_0\in [0,\ol t]$.

But then the limit $\lim_{n\to\infty} (x_n-t(x_n,\omega_n)\omega_n)$
exists and equals $x-t_0\omega\in \partial G$,
which by the definition of $t(x,\omega)$
implies that $t(x,\omega)\leq t_0$.
This gives us a contradiction since
\[
t(x,\omega)\leq t_0=\lim_{n\to\infty} t(x_n,\omega_n)\leq \ol t< t(x,\omega).
\]

(ii)
Assume first that $G$ is convex.
Let $(x,\omega)\in G\times S$
and choose any $\alpha\in\R$ and $\lambda\in \R^3$
such that $\lambda\cdot z<\alpha$ for all $z\in G$
and $\lambda\cdot (x-t(x,\omega)\omega)=\alpha$,
i.e. $\lambda\cdot x-\alpha=t(x,\omega)\lambda\cdot\omega$.
This is possible by Proposition \ref{pr:support},
choosing $y:=x-t(x,\omega)\omega\in\partial G$ there, and writing $\alpha:=\lambda\cdot y$.

Notice that because $t(x,\omega)>0$
and $t(x,\omega)\lambda\cdot\omega=\lambda\cdot x-\alpha<0$,
we must have $\lambda\cdot\omega<0$.
Then if $(x_n,\omega_n)$ is any sequence in $G\times S$
converging to $(x,\omega)$, one has
\[
\alpha\geq
\lambda\cdot (x_n-t(x_n,\omega_n)\omega_n)
=\lambda\cdot x_n-t(x_n,\omega_n)\lambda\cdot \omega_n
\]
i.e.
\[
t(x_n,\omega_n)\lambda\cdot \omega_n \geq \lambda\cdot x_n-\alpha,
\]
and therefore
\[
\limsup_{n\to\infty} t(x_n,\omega_n)\leq \frac{\lambda\cdot x-\alpha}{\lambda\cdot \omega} = t(x,\omega).
\]
This proves the upper semi-continuity of $t$,
which combined with the result of the case (i)
shows the continuity of $t$ on $G$.

In the opposite direction, let us then demonstrate that the convexity of $G$
follows from the continuity of $t$ on $G\times S$.
Indeed, if $G$ is not convex,
there are $x,y\in G$ such that the line $\ell_{x,y}:=\{x+t(y-x)\ |\ x\in [0,1]\}$
between them is not completely contained in $G$.
Since $G$ is open and connected,
there is a path $\gamma:[0,1]\to G$ such that $\gamma(0)=x$,
$\gamma(1)=y$ and $\gamma(s)\neq x$ for all $s\in ]0,1]$.
Define $s_0\in\R$ to be the the infimum of $s\in [0,1]$
such that $\ell_{x,\gamma(s)}\not\subset G$.
Clearly, $s_0>0$ and $\ell_{x,\gamma(s_0)}\not\subset G$.
We let $s_n\in ]0,1]$ be an increasing sequence
whose limit is $s_0$, and define $\omega_n:=\frac{x-\gamma(s_n)}{\n{\gamma(s_n)-x}}$,
$\omega_0:=\frac{x-\gamma(s_0)}{\n{\gamma(s_0)-x}}$.
But then for all $n$ and $s\in [0,\n{\gamma(s_n)-x}]$
we have $x-s\omega_n\in G$,
and therefore $t(x,\omega_n)\geq \n{\gamma(s_n)-x}$.
On the other hand, since $\ell_{x,\gamma(s_0)}\not\subset G$
and $\gamma(s_0)\in G$, one has $t(x,\omega_0)<\n{\gamma(s_0)-x}$.
Finally, because $(x,\omega_n)\to (x,\omega_0)$, we 
have
\[
\limsup_{n\to\infty} t(x,\omega_n)
\geq \limsup_{n\to\infty} \n{\gamma(s_n)-x}
=\n{\gamma(s_0)-x}>t(x,\omega_0),
\]
and thus, we conclude that $t$ is not upper semi-continuous on $G\times S$.
This completes the proof of the first part of (ii).

For the second part,
let $(x_0,\omega_0)\in G\times S$ and $y_0=x_0-t(x_0,\omega_0)\omega_0$.
As $y_0\in\partial G$ by Lemma \ref{le:escape_time},
there exists by Proposition \ref{pr:support} a number $\alpha\in\R$ and
a vector $\lambda\in\R^3$
such that $\lambda\cdot y<\alpha$ for all $y\in G$
and $\lambda\cdot y_0=\alpha$.

Therefore, as $x_0\in G$, we have $\lambda\cdot x_0<\alpha$
and since $t_0:=t(x_0,\omega_0)>0$,
we deduce that $\lambda\cdot \omega_0 = t_0^{-1}(\lambda\cdot x_0-\alpha)<0$.
Then if $(x_n,\omega_n)$ is a sequence in $G\times S$ that converges
to $(y_0,\omega_0)$ in $\R^3\times S$,
we have like earlier,
$t(x_n,\omega_n)\lambda\cdot \omega_n \geq \lambda\cdot x_n-\alpha$.
Combining this with the inequality $\lambda\cdot \omega_0<0$ allows us to conclude
\[
\limsup_{n\to\infty} t(x_n,\omega_n)\leq \frac{\lambda\cdot y_0-\alpha}{\lambda\cdot \omega_0} = 0,
\]
where in the last step we used again the equality $\lambda\cdot y_0=\alpha$.
The proof of case (ii) is finished.
\end{proof}

\begin{remark}
The results of Proposition \ref{pr:escape_time} remain true with the given proof
if $G$ is an arbitrary open bounded convex subset of $\R^3$, i.e. it does not necessarily have to have piecewise $C^1$-boundary.
\end{remark}

In particular, due to lower semicontinuity (case (i) of the above proposition)
the escape time $t$ is a Lebesgue-measurable map on $G\times S$.

\begin{proposition}\label{tp}
The mapping $t:G\times S\to\R$ is continuously differentiable on a neighbourhood of
every point $(x_0,\omega_0)\in G\times S$ for which $\omega\cdot\nu(y_0)<0$, where $y_0=x_0-t(x_0,\omega_0)\omega_0\in \partial G$ and where $y_0$ is a regular point of $\partial G$. Moreover, in this case \eqref{t0} holds at the point $(y_0,\omega_0)\in\Gamma_-$ and
\begin{align}\label{l11}
\omega_0\cdot (\nabla t)(x_0,\omega_0)=1.
\end{align}
\end{proposition}

\begin{proof}
Let $(x_0,\omega_0)\in G\times S$ be such that $y_0=x_0-t(x_0,\omega_0)\omega_0$
is a regular point of $\partial G$ and that $\omega_0\cdot \nu(y_0)<0$.
Choose $C^1$-diffeomorphism $H:D\to V$ from an open subset $D\subset\R^3$
onto an open subset $V\subset\R^3$ containing $y_0$ such that $V\cap\ol{G}=H(D_+)$,
with $D_+=\{(x_1,x_2,x_3)\in D\ |\ x_3\geq 0\}$
and $V\cap \partial G=H(D_0)$ where $D_0=\{(x_1,x_2,0)\in D\}$,
which we implicitly identify with the obvious subset of $\R^2$.
Such a $H$ exists since $y_0$ was assumed to be a regular point of $\partial G$.
Define
\[
F:D_0\times\R\times S\to\R^3\times S;\quad F(u,s,\omega)=(H(u,0)+s\omega,\omega),
\]
where $u=(u_1,u_2)\in D_0$, $s\in \R$, $\omega\in S$.
Clearly $F$ is $C^1$, and if $y_0=H(u_0,0)$ and $t_0=t(x_0,\omega_0)$, we have $F(u_0,t_0,\omega_0)=(x_0,\omega_0)$.
Moreover, identifying $T_{u_0} D_0$ with $\R^2$ as well,
\[
DF(u_0,t_0,\omega_0)(v,r,\theta)=\big(DH(u_0,0)v+r\omega_0+t_0\theta,\theta),
\quad (v,r,\theta)\in\R^2\times\R\times T_{\omega_0} S,
\]
and therefore $DF(u_0,t_0,\omega_0)(v,r,\theta)=0$ implies
\[
\theta=0\quad \textrm{and}\quad DH(u_0,0)v+r\omega_0=0.
\]
Using that $DH(u_0,0)v\in T_{y_0}(\partial G)$,
we have furthermore
\[
0=\nu(y_0)\cdot (DH(u_0,0)v+r\omega_0)=r\nu(y_0)\cdot \omega_0
\]
and since $\omega_0\cdot\nu(y_0)<0$, that $r=0$ and finally $v=0$.

It has thus been shown (by the Inverse Mapping Theorem) that $F$ is a diffeomorphism
from a small neighbourhood $W$ of $(u_0,t_0,\omega_0)$ onto a neighbourhood $U$ of $(x_0,\omega_0)$.
We claim that
\begin{align}\label{eq:tp:1}
t\big(F(u,s,\omega)\big)=s,\quad \forall (u,s,\omega)\in W.
\end{align}
Once this is established, it is clear that $t$ is a $C^1$-mapping on $U$ because
then
\[
t(x,\omega)=\pr_2\big(F^{-1}(x,\omega)\big),\quad \forall (x,\omega)\in U,
\]
where $\pr_2$ is the projection onto the second factor $D_0\times\R\times S\to D_0$.

It suffices to show that we can shrink $W$ around $(u_0,t_0,\omega_0)$ such that
the (half open) segment $\{H(u,0)+s\omega\ |\ 0<s\leq \tau\}$ lies completely in $G$
for all $(u,\tau,\omega)\in W$.
Indeed, once this has been established, \eqref{eq:tp:1} follows directly from the definition of $t$.
 
To show that such a modification of $W$ is possible,
we argue by contradiction.
Thus, suppose that there was a sequence $(u_n,\tau_n,\omega_n)$ in $W$
converging to $(u_0,t_0,\omega_0)$
and a sequence of numbers $s_n$, $0<s_n\leq \tau_n$,
such that $v_n:=H(u_n,0)+s_n\omega_n\notin G$.
Taking the numbers $s_n$ to be smaller if necessary, we may assume that $v_n\in\partial G$ for every $n$. 
As the sequence $(s_n)$ is bounded, we may pass to a subsequence $(s_{n_i})$
that converges to some $s'\in [0,t_0]$ (as $\tau_{n_i}\to t_0$),
and because $\partial G$ is closed, we have $\lim_{i\to\infty} v_{n_i}=H(u_0,0)+s'\omega_0\in \partial G$.
On the other hand, $H(u_0,0)=y_0=x_0-t_0\omega_0$
and hence the definition of $t_0=t(x_0,\omega_0)$
implies that $t_0\leq t_0-s'$ i.e. $s'=0$.
Thus $v_{n_i}\to y_0=H(u_0,0)$ in $\R^3$ when $i\to \infty$.

The boundary $\partial G$ being an embedded submanifold of $\R^3$, we have $v_{n_i}\to y_0$ also in $\partial G$,
and hence for all $i$ big enough, one has $v_{n_i}=H(\tilde{u}_{n_i},0)$ for some $\tilde{u}_{n_i}\in D_0$ such that $\tilde{u}_{n_i}\to u_0$ when $i\to\infty$.

Computing the differential of $F$ at $(u_0,0,\omega_0)$ as above,
we see that $DF(u_0,0,\omega_0)(v,r,\theta)=(DH(u_0,0)v+r\omega_0,\theta)$,
from which it is seen that $DF(u_0,0,\omega_0)$ is invertible,
and hence that $F$ is a $C^1$-diffeomorphism from an open neighbourhood
$\tilde{W}$ of $(u_0,0,\omega_0)$ onto an open neighbourhood $\tilde{U}$
of $F(u_0,0,\omega_0)=(y_0,\omega_0)$ in $\R^3\times S$.
But $(\tilde{u}_{n_i},0,\omega_{n_i})\to (u_0,0,\omega_0)$,
hence for $i$ large enough, points $(\tilde{u}_{n_i},0,\omega_{n_i})$
and $(u_{n_i},s_{n_i},\omega_{n_i})$ all belong to $\tilde{W}$.
On the other hand,
\[
F(u_{n_i},s_{n_i},\omega_{n_i})=(v_{n_i},\omega_{n_i})
=(H(\tilde{u}_{n_i},0),\omega_{n_i})=F(\tilde{u}_{n_i},0,\omega_{n_i}),
\]
and thus the injectivity of $F$ on $\tilde{W}$
implies that $(u_{n_i},s_{n_i},\omega_{n_i})=(\tilde{u}_{n_i},0,\omega_{n_i})$
for large enough $i$.
In particular $s_{n_i}=0$ for large $i$, which contradicts the fact
that $s_{n}>0$ for all $n$.

As explained before, this contradiction establishes \eqref{eq:tp:1}
and concludes our proof of $C^1$-differentiability property $t$ as announced in the statement of this proposition.

We shall next demonstrate the limiting property \eqref{t0}.
Let $(x_n,\omega_n)$ be a sequence in $G\times S$ converging to $(y_0,\omega_0)$,
where as before $y_0=x_0-t_0(x_0,\omega_0)\omega_0\in\partial G$.
Using the $F$ as defined previously, and setting $y_0=H(u_0,0)$,
we have as above that $DF(u_0,0,\omega_0)(v,r,\theta)=(DH(u_0,0)v+r\omega_0,\theta)$,
from which we again deduce that $DF(u_0,0,\omega_0)$ is invertible,
and therefore there exist an open subset $\tilde{W}\subset D_0\times\R\times S$
containing $(u_0,0,\omega_0)$ which is mapped diffeomorphically
onto an open neighbourhood $\tilde{U}\subset \R^3\times S$ of $(y_0,\omega_0)$.

Defining $u_n,s_n$ by requiring that $F(u_n,s_n,\omega_n)=(x_n,\omega_n)$,
we have $(u_n,s_n,\omega_n)\to F^{-1}(y_0,\omega_0)=(u_0,0,\omega_0)$.
On the other hand, as $x_n-s_n\omega_n = H(u_n,0)\in\partial G$,
it follows that $t(x_n,\omega_n)\leq s_n$ for all $n$
and therefore
\[
\limsup_{n\to\infty} t(x_n,\omega_n)\leq \lim_{n\to\infty} s_n = 0.
\]
Hence the validity of the limit \eqref{t0} is demonstrated (since $\liminf_{n\to\infty} t(x_n,\omega_n)\geq 0$).

Finally, to prove \eqref{l11}, we observe that whenever $|s|$ is small enough, $s\in\R$,
it holds that $t(x_0+s\omega_0,\omega_0)=s+t(x_0,\omega_0)$
and hence
\[
\omega_0\cdot (\nabla t)(x_0,\omega_0)=\dif{s}\big|_{s=0} t(x_0+s\omega_0,\omega_0)
=\dif{s}\big|_{s=0} \big(s+t(x_0,\omega_0)\big)
=1.
\]
\end{proof}

Since $\Gamma_0$ has measure zero on $\partial G\times S\times I$,
we have the following result (cf. Lemme 2.3.3 in \cite{allaire12}):

\begin{theorem}\label{th:bardos}
The set
\begin{align}\label{eq:N_0}
N_0:=\{(x,\omega,E)\in G\times S\times I\ |\ & \textrm{either}\ y\in\partial G\setminus(\partial G)_r,\ \textrm{or}\ y\in(\partial G)_r\ \textrm{and}\ \omega\cdot\nu(y)=0, \nonumber\\
&\textrm{where}\ y=x-t(x,\omega)\omega\in\partial G\}
\end{align}
has a measure zero in $G\times S\times I$.
\end{theorem}

\begin{proof}
First observe that
\[
N_0\subset P\big((\tilde\Gamma_0\times\R)\cup \big((\partial G\setminus(\partial G)_r)\times S\times I\times\R\big)\big),
\]
where
\[
P:\partial G\times S\times I\times\R\to \R^3\times S\times I;\quad P(y,\omega,E,t)=(y+t\omega,\omega,E).
\]
Since $(\tilde\Gamma_0\times\R)\cup ((\partial G\setminus(\partial G)_r)\times S\times I\times\R)$
has measure zero in $\partial G\times S\times I\times\R$
and $\dim (\partial G\times S\times I\times\R)=\dim (\R^3\times S\times I)$,
to prove the theorem, it suffices to
show that $P$ is locally Lipschitz-continuous, since then it maps sets of Lebesgue measure zero to sets of Lebesgue measure zero.
Indeed, we have
\[
& d_{\R^3\times S\times I}\big(P(y_1,\omega_1,E_1,t_1),P(y_2,\omega_2,E_2,t_2)\big) \\
=&\n{y_1-y_2}+\n{t_1\omega_1-t_2\omega_2}+d_S(\omega_1,\omega_2)+|E_1-E_2| \\
\leq&\n{y_1-y_2}+|t_1-t_2|+|t_1|\n{\omega_1-\omega_2}+d_S(\omega_1,\omega_2)+|E_1-E_2| \\
\leq & 
C_1d_{\partial G}(y_1,y_2)+|t_1-t_2|
+(1+|t_1|)C_2d_S(\omega_1,\omega_2)+|E_1-E_2|,
\]
where $d_{\partial G}$ and $d_{S}$ are the (intrinsic) geodesic metrics on $\partial G$ and $S$,
respectively, and $C_1,C_2>0$ are some constants coming from the fact that these geodesic metrics are equivalent with the
restrictions of metrics of the ambient space $\R^3$ (recall that $\partial G$ is compact).
Finally, for any bounded interval $J\subset\R$ there is a constant $C_3>0$
such that the last line in the above inequality is dominated by
$C_3d_{\partial G\times S\times I\times \R}\big((y_1,\omega_1,E_1,t_1),(y_2,\omega_2,E_2,t_2)\big)$, whenever $t_1,t_2\in J$.
This established the claim.
\end{proof}

%%%%%%%%%%%%%%%%%%%%%%%%%%%%%%%%%%%%%%%%%%%%%%%%%%%%%%%%%%%%%%%%%%%%%%%%
\subsection{Local Solution Obtained by Lagrange's Method}\label{lss1}
%%%%%%%%%%%%%%%%%%%%%%%%%%%%%%%%%%%%%%%%%%%%%%%%%%%%%%%%%%%%%%%%%%%%%%%%

We consider only the convection of one particle that is, the solution $\psi$ is scalar valued. The vector valued case ($\psi=(\psi_1,\psi_2,\psi_3)$) is then obtained (when required) easily and it is considered in section \ref{lss3}.

At first we apply the classical Lagrange's method (the method of characteristics) for the convection equation
\be\label{l1}
\omega\cdot \nabla\psi+\lambda\psi=f(x,\omega,E)
\ee
where $\lambda \in\R$ and $f\in C(\ol G\times S\times I )$. We demand that the solution
$\psi=\psi(x,\omega,E)$ satisfy the inflow boundary condition
\be
\psi(y,\omega,E)=0,\quad {\rm when }\ (y,\omega,E)\in \partial G_r\times S\times I ;\ \omega\cdot\nu(y)<0.
\ee

First, we seek a general solution for the equation \eqref{l1}, which we write as
\[
\sum_{j=1}^3 \omega_j{\p \psi{x_j}}+\lambda\psi=f(x,\omega,E).
\]
Denote $(x,\omega,E)=(x_1,x_2,x_3,\omega_1,\omega_2,\omega_3,E)$. Then  the augmented system of ordinary differential equations (the system of characteristics) is
\[
X_1'(s)&=\Omega_1,\ \ \ \Omega_1'(s)=0\\
X_2'(s)&=\Omega_2,\ \ \ \Omega_2'(s)=0\\
X_3'(s)&=\Omega_3,\ \ \ \Omega_3'(s)=0\\
{\s E}'(s)&=0 \\
\Psi'(s)&= \ f(X,\Omega,{\s E})-\lambda \Psi
\]
where we denoted $X=(X_1,X_2,X_3),\ \Omega=(\Omega_1,\Omega_2,\Omega_3)$.

We have $\Omega_j'(s)=0,\ {\s E}'(s)=0$ which implies
$\Omega_j(s)=C_j$ (a constant) and ${\s E}=C''$ (a constant). Hence we further get $X_j'(s)=C_j$ and so $X_j(s)=C_js+C_j'$ where $C_j'$ are constants. Denote $C=(C_1,C_2,C_3),\ C'=(C_1',C_2',C_3')$. Then
\[
\Psi'(s)=f(sC+C',C,C'')-\lambda \Psi
\]
whose solutions is
\be \label{l2}
\Psi(s)=e^{-\lambda s} \big(C_0+\int_0^sf(\tau C+C',C,C'')e^{\lambda\tau}\ d\tau\big),
\ee
where $C_0$ is again a constant.

Next we consider (locally) the initial value  for the augmented system. It must be of the form
$(X(0),\Omega(0),{\s E}(0),\Psi(0))=\Theta(w)$ where $w\in W\subset\circ\R^6$ and $\Theta: W\to \R^8$ is the (local) parametrization of the $6$-dimensional manifold $\zeta:=\Theta (W)$ through which the curve $(X,\Omega,{\s E},\Psi)$ goes at $s=0$. Let $h=(h_1,h_2,h_3):{\s V}\to \partial G$, ${\s V}\subset\R^2$ (open)
be a local parametrization of the boundary $\partial G_r$. Suppose that $y_0=h(v_0)\in (\partial G)_r$ such that $\omega_0\cdot\nu(y_0)=\omega_0\cdot\nu(h(v_0))<0$. Then there exist
an open neighbourhood $ {\s V'}\subset {\s V}$ and an open neighbourhood  $ {\s U}\subset\R^3$ such that $\omega\cdot\nu(h(v))<0$ for all
$(v,\omega)\in {\s V'}\times {\s U}$ (here we exceptionally assume $\omega$ belongs to an open subset of $\R^3$). Hence the local parametrization $\Theta$ is
\[
\Theta:{\s V'}\times {\s U}\times \Delta\to \R^8;\ \Theta(w)=(h(v),\omega,E,0), \ w=(v,\omega,E)
\]
where $\Delta\subset I$.

The initial condition is
\[
(X(0),\Omega(0),{\s E}(0),\Psi(0))=\Theta(w)=(h(v),\omega,E,0)
\]
which is equivalent to
\be\label{l3}
X(0)=h(v),\ \Omega(0)=\omega,\ {\s E}(0)=E,\ \Psi(0)=0.
\ee
Taking into account the above obtained general solutions, the condition (\ref{l3}) means that $
C_0=0,\ C'=h(v),\ C=\omega,\ C''=E$ and then
\be\label{l4}
X(s):=X_{(v,\omega,E)}(s)=s\omega+h(v),\ \Omega(s):=\Omega_{(v,\omega,E)}(s)=\omega,\
{\s E}(s):={\s E}_{(v,\omega,E)}(s)=E
\ee
and
\be\label{l4a}
\Psi(s):=\Psi_{(v,\omega,E)}(s)=e^{-\lambda s} \int_0^s f(\tau \omega+h(v),\omega,E)e^{\lambda\tau}\ d\tau.
\ee

The Lagrange's method proceeds as follows. We denote
\be\label{l5}
X(s)=x,\ \Omega(s)=\omega,\ {\s E}(s)=E
\ee
from which we must eliminate $s,\ v$, ($\omega$ and $E$). We find that the equations (\ref{l5}) mean that
\be\label{l6}
h(v)=x-s\omega\in\partial G\ {\rm or}\ v=h^{-1}(x-s\omega).
\ee
Since $x-s\omega=h(v)\in\partial G$  the definition of $t(x,\omega)$ implies that $s=t(x,\omega)$ in (\ref{l6}).
Hence the solution $\psi$
\bea\label{l7}
\psi(x,\omega,E)&=\Psi_{(h^{-1}(x-s\omega),\omega,E)}(t(x,\omega))\\
&=e^{-\lambda t(x,\omega)} \int_0^{t(x,\omega)}f(\tau \omega+x-t(x,\omega)\omega,\omega,E)e^{\lambda\tau}\ d\tau\\
&=\int_0^{t(x,\omega)}f(x-\tau \omega,\omega,E)e^{-\lambda\tau}\ d\tau.
\eea

The applied Lagrange's method gives (locally) a unique continuous solution (for which ${\p \psi{x_j}}$ are continuous)
when initial value manifold $\zeta$ is not characteristic at $\Theta(w_0)$ for the convection equation. This means that
at the given point $(x_0,\omega_0,E_0,0)=(h(v_0),\omega_0,E_0,0)=\Theta(w_0)\in \zeta$ one must have
\be \label{l8}
\det\qmatrix{
\partial_1\theta(w_0)\cr \vdots\cr\partial_6\theta(w_0)\cr a(x_0,\omega_0,E_0)\cr
}
\not=0
\ee
where $\theta(w):=(h(v),\omega,E)$ and $a(x,\omega,E):=(-\omega_1,-\omega_2,-\omega_3,0,0,0,0)$. Hence
\be \label{l8a}
\det\qmatrix{
\partial_1h_1(v_0)&\partial_1h_2(v_0)&\partial_1h_3(v_0)&0&0&0&0\cr
\partial_2h_1(v_0)&\partial_2h_2(v_0)&\partial_2h_3(v_0)&0&0&0&0\cr
0&0&0&1&0&0&0\cr
0&0&0&0&1&0&0\cr
0&0&0&0&0&1&0\cr
0&0&0&0&0&0&1\cr
\omega_{10}&\omega_{20}&\omega_{30}&0&0&0&0\cr
}\neq 0
\ee
that is
\[
\omega_{10}\det\qmatrix{
\partial_1h_2&\partial_1h_3\cr
\partial_2h_2&\partial_2h_3\cr
}
-
\omega_{20}\det\qmatrix{
\partial_1h_1&\partial_1h_3\cr
\partial_2h_1&\partial_2h_3\cr
}
+
\omega_{30}\det\qmatrix{
\partial_1h_1&\partial_1h_2\cr
\partial_2h_1&\partial_2h_2\cr
}\neq 0,
\]
where partial derivatives $\partial_i h_j$ are evaluated at $v_0$.
Since the normal $\nu(y_0)$ of the surface $\partial G_r$ at $y_0:=h(v_0)$ is parallel to
\[
&
\left(\det\qmatrix{\partial_1h_2(v_0)&\partial_1h_3(v_0)\cr
\partial_2h_2(v_0)&\partial_2h_3(v_0)\cr}
,
-{\det}\qmatrix{\partial_1h_1(v_0)&\partial_1h_3(v_0)\cr
\partial_2h_1(v_0)&\partial_2h_3(v_0)\cr}
,
{\det}\qmatrix{\partial_1h_1(v_0)&\partial_1h_2(v_0)\cr
\partial_2h_1(v_0)&\partial_2h_2(v_0)\cr
}\right) \\
&=(\partial_1h\times\partial_2h)(v_0)
\]
the condition (\ref{l8a}) is equivalent to $\omega_0\cdot \nu(y_0)\not=0$ which is satisfied on the manifold $\zeta$ where $\omega\cdot\nu(h(v))<0$. Hence the obtained solution $\psi$ exists locally and it is unique.

\begin{remark}\label{lr2}
The above expressed method of characteristics needs only the condition $\omega_0\cdot \nu(y_0)\not=0$ to guarantee the existence of the unique {\it local} solution such that ${\p \psi{x_j}}$ are continuous.
\end{remark}

\begin{remark}
Here we give formally a shorter argument
for the derivation of the explicit form 
of the solution to \eqref{l1}:
\[
\dif{s} \psi(x-s\omega,\omega,E)=-\omega\cdot\nabla \psi(x-s\omega,\omega,E)
=\lambda\psi(x-s\omega,\omega,E)-f(x-s\omega,\omega,E).
\]
Write $\Psi(s)=\psi(x-s\omega,\omega,E)$, $F(s)=f(x-s\omega,\omega,E)$,
and we have
\[
\Psi'(s)-\lambda\Psi(s)=-F(s),
\]
i.e.
\[
\dif{s}(e^{-\lambda s}\Psi(s))=-e^{-\lambda s}F(s),
\]
from which
\[
\Psi(s)=e^{\lambda s}\Big(\Psi(0)-\int_0^{s} e^{-\lambda \tau}F(\tau)\diff\tau\Big),
\]
or
\[
\psi(x-s\omega,\omega,E)=e^{\lambda s}\Big(\psi(x,\omega,E)-\int_0^{s} e^{-\lambda \tau}f(x-\tau\omega,\omega,E)\diff\tau\Big).
\]
Letting $s=t(x,\omega)$, we therefore obtain
\[
\psi(x,\omega,E)=e^{-\lambda t(x,\omega)}\psi(x-t(x,\omega)\omega,\omega,E)+\int_0^{t(x,\omega)} e^{-\lambda \tau}f(x-\tau\omega,\omega,E)\diff\tau,
\]
where the first term on the right hand size vanishes,
because of the assumption that $\psi=0$ on $\Gamma_-$.
\end{remark}

%%%%%%%%%%%%%%%%%%%%%%%%%%%%%%%%%%%%%%%%%%%%%%%%%%%%%%%%%%%%%%%%%%%%%%%%
\subsection{Global Solution Given by the Method of Characteristics}\label{lss3}
%%%%%%%%%%%%%%%%%%%%%%%%%%%%%%%%%%%%%%%%%%%%%%%%%%%%%%%%%%%%%%%%%%%%%%%%

The section \ref{lss2} suggests that the solution for the convection equation (for $\lambda\in\R$)
\be\label{l1aaa}
\omega\cdot \nabla\psi+\lambda\psi=f(x,\omega,E)
\ee
be
\be \label{l15}
\psi(x,\omega,E)=\int_0^{t(x,\omega)}f(x-s\omega,\omega,E)e^{-\lambda s}\ ds.
\ee
In the following we denote
\[
D:=(G\times S\times I)\setminus N_0,
\]
where $N_0$ is Lebesgue zero measurable set given in Theorem \ref{th:bardos}.

\begin{theorem}\label{lth1}
Suppose that $f\in C(\ol G\times S\times I)$ is such that ${\p f{x_j}}\in C(\ol G\times S\times I)$.
Then (\ref{l15}) is the unique solution of the equation (\ref{l1aaa}) in $D$ satisfying the inflow boundary condition $\psi(y,\omega,E)=0$ for $(y,\omega,E)\in\Gamma_-$, where $y:=x-t(x,\omega)\omega,\ x\in D$ 
(that is, $\psi$ is the classical solution in $D$).
\end{theorem}

\begin{proof}
Since $f\in C(\ol G\times S\times I)$ the expression (\ref{l15}) is defined for all $(x,\omega,E)\in
\ol G\times S\times I$.
Define
\[
F(x,\omega,E,t)=\int_0^t f(x-s \omega,\omega,E)e^{-\lambda s}\ ds,\quad \ (x,\omega,E)\in D,\ t\in [0,t(x,\omega)].
\]
Then for $(x,\omega,E)\in D$,
\bea\label{cont}
\psi(x,\omega,E)=F(x,\omega,E,t(x,\omega))
\eea
and so
\be\label{l16}
{\p \psi{x_j}}
={\p F{x_j}}(x,\omega,E,t(x,\omega))+{\p F{t}}(x,\omega,E,t(x,\omega)){\p t{x_j}}(x,\omega).
\ee
Hence (recall that by assumption ${\p f{x_j}}\in C(\ol G\times S\times I)$)
\bea\label{l17}
{\p \psi{x_j}}&=\int_0^{t(x,\omega)} {\p f{x_j}}(x-s\omega,\omega,E)e^{-\lambda s}ds
+f(x-t(x,\omega)\omega,\omega,E)e^{-\lambda t(x,\omega)}{\p t{x_j}}(x,\omega).
\eea
Hence we see that ${\p \psi{x_j}}(x,\omega,E)$ exists on $D$ and
\bea
\omega\cdot\nabla\psi
=&\int_0^{t(x,\omega)} \omega\cdot\nabla_x f(x-t\omega,\omega,E)e^{-\lambda t}dt \\
&+f(x-t(x,\omega)\omega,E)e^{-\lambda t(x,\omega)}\omega\cdot(\nabla_x t)(x,\omega).
\eea

Using Eq. \eqref{l11} i.e. $\omega\cdot(\nabla_x t)(x,\omega)=1$
and the basic fact that
\[
\dif{s} f(x-s\omega,\omega,E)=-\omega\cdot \nabla_x f(x-s\omega,\omega,E),
\]
we can simplify the above formula as follows:
\[
\omega\cdot\nabla\psi
=&
\int_0^{t(x,\omega)} -\Big(\dif{s} f(x-s\omega,\omega,E)\Big)e^{-\lambda s}ds
+f(x-t(x,\omega)\omega,\omega,E)e^{-\lambda t(x,\omega)} \\
=&-f(x-t(x,\omega)\omega,\omega,E)e^{-\lambda t(x,\omega)}+f(x,\omega,E) \nonumber \\
&-\lambda\int_0^{t(x,\omega)} f(x-s\omega,\omega,E) e^{-\lambda s}ds
+f(x-t(x,\omega)\omega,\omega,E)e^{-\lambda t(x,\omega)} \\
=&f(x,\omega,E)-\lambda F(x,\omega,E,t(x,\omega)) \\
=&f(x,\omega,E)-\lambda\psi.
\]
We thus see that the convection equation holds.

For $y\in\Gamma_-$ given in the assertion we have by Proposition \ref{tp}
that $t(y,\omega)=0$, and then $\psi(y,\omega,E)=0$ for the inflow boundary points given in the theorem.
This finishes the proof.

\end{proof}

By applying the similar methods as above we
get the following theorems (cf. \cite{dautraylionsv6}, p. 244-246).

\begin{theorem}\label{lth3}
Suppose that
$f\in C(\ol G\times S\times I)$ such that ${\p f{x_j}}\in C(\ol G\times S\times I)$,
and let $\Sigma\in C(\ol G\times S\times I)$ such that 
${\p {\Sigma}{x_j}}\in C(\ol G\times S\times I),\ j=1,2,3$  .
Then  the unique (classical) solution of the equation
\be\label{l19a}
\omega\cdot \nabla\psi+\Sigma\psi=f\quad {\rm in}\ D 
\ee
satisfying the homogeneous inflow boundary condition
\be\label{l20a}
\psi(y,\omega,E)=0\ {\rm for}\ (y,\omega,E)\in\Gamma_- \ 
\ee
is given by
\be \label{l21a}
\psi(x,\omega,E)=\int_0^{t(x,\omega)}e^{-\int_0^t\Sigma(x-s\omega,\omega,E)ds} f(x-t\omega,\omega,E) dt.
\ee
\end{theorem}

\begin{theorem}\label{lth3a}
Suppose that $g\in C(\Gamma_-)$ such that ${\p g{\tilde y_i}}\in C(\Gamma_-), i=1,2$,
where ${\partial\over{\partial \tilde y_i}}$ denotes any local basis of the
tangent space of $(\partial G)_r$
(and ${\p g{\tilde y_i}}\in C(\Gamma_-)$ is to be understood in a local sense),
and $\Sigma\in C(\ol G\times S\times I)$ such that
${\p {\Sigma}{x_j}}\in C(\ol G\times S\times I),\ j=1,2,3$.
Then  the unique (classical) solution of the equation
\be\label{l19}
\omega\cdot \nabla\psi+\Sigma\psi=0\quad {\rm in}\ D
\ee
satisfying the inhomogeneous inflow boundary condition
\be\label{l20a:1}
\psi(y,\omega,E)=g(y,\omega,E)\quad {\rm for}\ (y,\omega,E)\in \Gamma_-\ 
\ee
is given by
\be \label{l21}
\psi(x,\omega,E)=e^{-\int_0^{t(x,\omega)}\Sigma(x-s\omega,\omega,E)ds} g(x-t(x,\omega)\omega,\omega,E).
\ee
\end{theorem}

\begin{proof}
We denote $B:G\times S\times I\to \partial G$; $B(x,\omega,E)=x-t(x,\omega)\omega$,
which by Proposition \ref{tp} is $C^1$-smooth on $D$.
It then follows from the considerations in section \ref{lss2},
that $(x,\omega,E)\mapsto (B(x,\omega,E),\omega,E)$
is $C^1$-map with respect to $x$ from $D$ into $\Gamma_-$ ,
and hence, by the regularity assumptions imposed on $g$,
that the partial derivatives $\pa{x_i} g(B(x,\omega,E),\omega,E)$
exist and are continuous on $D$.

Taking $\psi$ to be defined by Eq. \eqref{l21}, which can be written as
\be\label{bsol}
\psi(x,\omega,E)=e^{-\int_0^{t(x,\omega)} \Sigma(x-s\omega,\omega,E) ds} g(B(x,\omega,E),\omega,E),
\ee
we have on $D$,
\[
&\nabla_x \psi(x,\omega,E) \\
=&\psi(x,\omega,E)\Big(-\Sigma(x-t(x,\omega,E)\omega,\omega,E)(\nabla_x t)(x,\omega)-\int_0^{t(x,\omega)}(\nabla_x \Sigma)(x-s\omega,\omega,E) ds\Big) \\
&+e^{-\int_0^{t(x,\omega)} \Sigma(x-s\omega,\omega,E) ds} \nabla_x\big(g(B(x,\omega,E),\omega,E)\big).
\]
We shall take an inner product of this formula with $\omega$.
To this end, recall that $\omega\cdot (\nabla_x t)(x,\omega)=1$ by \eqref{l11},
and notice that
$\omega\cdot (\nabla_x\Sigma)(x-s\omega,\omega,E)=-\dif{s} \Sigma(x-s\omega,\omega,E)$.
Moreover, for all $s$ near zero,
\[
B(x+s\omega,\omega,E)=&(x+s\omega)-t(x+s\omega,\omega)\omega \\
=&(x+s\omega)-\big(t(x,\omega)+s)\omega=x-t(x,\omega)\omega \\
=&B(x,\omega,E),
\]
and hence
\[
\omega\cdot \nabla_x\big(g(B(x,\omega,E),\omega,E)\big)
=&\dif{s}\big|_{s=0} g(B(x+s\omega,\omega,E),\omega,E)) \\
=&\dif{s}\big|_{s=0} g(B(x,\omega,E),\omega,E)) \\
=&0.
\]
Thus,
\[
&\omega\cdot \nabla_x \psi(x,\omega,E) \\
=&\psi(x,\omega,E)\Big(-\Sigma(x-t(x,\omega)\omega,\omega,E)+\Sigma(x-t(x,\omega)\omega, \omega,E)-\Sigma(x,\omega,E)\Big) \\
&+e^{-\int_0^{t(x,\omega)} \Sigma(x-s\omega,\omega,E) ds} \omega\cdot \nabla_x\big(g(B(x,\omega,E),\omega,E)\big) \\
=&-\Sigma(x,\omega,E)\psi(x,\omega,E),
\]
which is \eqref{l21}.

On the other hand, if $(y,\omega,E)\in\Gamma_-$,
then $t(y,\omega)=0$ by Proposition \ref{tp},
and hence $B(y,\omega,E)=y$, which gives $\psi(y,\omega,E)=g(y,\omega,E)$
i.e. \eqref{l20a:1}.
\end{proof}

With the assumptions of Theorems \ref{lth3} and \ref{lth3a} the (classical) solution of the problem
\be\label{l22a}
\omega\cdot \nabla\psi+\Sigma\psi=f\quad {\rm in}\ D 
\ee
satisfying the inhomogeneous inflow boundary condition (\ref{l20a:1})
is the sum $\psi+\phi$ of the solutions of the problems
\bea\label{l22}
\omega\cdot \nabla\psi+\Sigma\psi&=f\quad {\rm in}\ D\nonumber\\
\psi_{|\Gamma_-}=0,
\eea
and
\bea\label{l23}
\omega\cdot \nabla\phi+\Sigma\phi&=0\quad {\rm in}\ D\nonumber\\
\phi_{|\Gamma_-}&=g. 
\eea

Hence we obtain under the assumptions of Theorems
\ref{lth3}, \ref{lth3a}  a  (classical) solution $\psi$ in $D$ for the problem \eqref{l19a}, \eqref{l20a:1}
\be \label{l24}
\psi(x,\omega,E)=&\int_0^{t(x,\omega)} e^{\int_0^t-\Sigma(x-s\omega,\omega,E)ds}\cdot f(x-t\omega,\omega,E) dt \\
&+e^{\int_0^{t(x,\omega)}-\Sigma(x-s\omega,\omega,E)ds}\cdot g(x-t(x,\omega)\omega,\omega,E). \nonumber
\ee

For later needs we also formulate a generalization of Theorem \ref{lth3a}:

Suppose that $g_j\in C(\Gamma_-)$
such that ${\p{g}{\tilde y_i}}\in C(\Gamma_-), \ i=1,2$
(in the same sense as in Theorem \ref{lth3a} above)
and  $\Sigma_{lk}\in C(\ol G\times S\times I)$ such that
${\p {\Sigma_{lk}}{x_j}}\in C(\ol G\times S\times I),\ 1\leq l,k\leq 3$.
Let
\[
\Sigma\psi=\ol\Sigma\qmatrix{
\psi_1\cr\psi_2\cr\psi_3\cr
}
\]
where $\ol\Sigma$ is the matrix $(\Sigma_{lk}(x,\omega,E))$.
Then the unique classical solution $\psi=(\psi_1,\psi_2,\psi_3)$ of the coupled system of equations
\be\label{l19b}
\omega\cdot \nabla\psi_j+\Sigma(x,\omega,E)\psi=0 \ {\rm in}\ D,\ j=1,2,3
\ee
satisfying the inhomogeneous inflow boundary condition
\be\label{l19d}
\psi_j(y,\omega,E)=g_j(y,\omega,E)\ {\rm on}\ \Gamma_-\ ,\ j=1,2,3
\ee
is
\be \label{l21c}
\psi(x,\omega,E)=e^{-\int_0^{t(x,\omega)} \ol\Sigma(x-s\omega,\omega,E)ds}\cdot g(x-t(x,\omega)\omega,\omega,E).
\ee
In the case where $\ol\Sigma$ is a diagonal matrix $\ol\Sigma={\rm diag}(\Sigma_1,\Sigma_2,\Sigma_3)$ the (classical) solution of (\ref{l19b}-\ref{l19d}) is
\[
\psi(x,\omega,E)&=\Big(e^{\int_0^{t(x,\omega)}-\Sigma_1(x-s\omega,\omega,E)ds}\cdot g_1(x-t(x,\omega)\omega,\omega,E),\nonumber\\ 
& e^{\int_0^{t(x,\omega)}-\Sigma_2(x-s\omega,\omega.E)ds}\cdot g_2(x-t(x,\omega)\omega,\omega,E), \nonumber\\
& e^{\int_0^{t(x,\omega)}-\Sigma_3(x-s\omega,\omega,E)ds}\cdot g_3(x-t(x,\omega)\omega,\omega,E)\Big).
\]
In this article we need only this solution of uncoupled convection equation.  

Similarly we find (a generalization of Theorem \ref{lth3}) that when $f_j\in C(\ol G\times S\times I), \ j=1,2,3$ such that
${\p {f_j}{x_k}}\in C(\ol G\times S\times I),\ j, k=1,2,3$, the (classical) solution of the coupled system
\[
\omega\cdot \nabla\psi_j+\Sigma(x,\omega,E)\psi=f_j(x,\omega,E)\ {\rm in}\ D,\ j=1,2,3
\]
satisfying the homogeneous inflow boundary condition
\[
{\psi_j}_{|\Gamma_-}=0,\quad j=1,2,3,
\]
is
\be\label{l21cc}
\psi(x,\omega,E)=\int_0^{t(x,\omega)}e^{\int_0^t-\ol\Sigma(x-s\omega,\omega,E)ds}\cdot f(x-t\omega,\omega,E) dt.
\ee

The (classical) solution for the general coupled system
\[
\omega\cdot \nabla\psi_j+\Sigma\psi=f_j\quad {\rm in}\ D,\ j=1,2,3
\]
satisfying the inhomogeneous inflow boundary condition
\[
{\psi_j}_{|\Gamma_-}=g_j,\quad j=1,2,3\  
\]
is obtained as the sum of solutions  (\ref{l21c}) and (\ref{l21cc}) (when the stated assumptions are valid).

\begin{remark}\label{lre2}
The classical solution $\psi$ obtained above is continuous in $D$
(and its partial derivatives ${\p {\psi}{x_k}}, \ k=1,2,3$ are continuous in $D$) which can be immediately seen from the formulas like \eqref{cont} and (\ref{bsol}). Thus $\psi$ is continuous almost everywhere in $G\times S\times I$, which implies in particular that it is Lebesgue measurable in $G\times S\times I$.

Note that in the case where $G$ is convex such that $\partial G$ is $C^1$-boundary, the solution $\psi$ is in $C(\ol G\times S\times I)$ and ${\p {\psi}{x_k}}\in C(G\times S\times I)$.
\end{remark}

\begin{remark}
Notice that the formulas for $\psi$ in Theorems \ref{lth3} and \ref{lth3a}
make sense under the
less restrictive assumptions $f\in C(\ol{G}\times S\times I)$ and $g\in C(\Gamma_-)$,
respectively,
i.e. assuming that $f$ and $g$ are merely continuous (and $\Sigma\in C(\ol G\times S\times I)$),
but not necessarily continuously differentiable with respect to $x$ and $y$, respectively
.
Then these $\psi$s can be considered as generalized (classical) solutions
to the corresponding boundary value problems
in the sense that if, by convention,
we replace $\omega\cdot\nabla\psi$
by $\dif{s}\psi(x+s\omega,\omega,E)|_{s=0}$,
then \eqref{l19a}-\eqref{l20a} and \eqref{l19}-\eqref{l20a:1}
are satisfied
for all $(x,\omega,E)\in D$.
\end{remark}

%%%%%%%%%%%%%%%%%%%%%%%%%%%%%%%%%%%%%%%%%%%%%%%%%%%%%%%%%%%%%%%%%%%%%%%%
\section{Dissipativity of the Convection Operator}\label{diss1}
%%%%%%%%%%%%%%%%%%%%%%%%%%%%%%%%%%%%%%%%%%%%%%%%%%%%%%%%%%%%%%%%%%%%%%%%

%%%%%%%%%%%%%%%%%%%%%%%%%%%%%%%%%%%%%%%%%%%%%%%%%%%%%%%%%%%%%%%%%%%%%%%%
\subsection{On Dissipativity of Linear Operators in Banach Spaces}\label{diss1a}
%%%%%%%%%%%%%%%%%%%%%%%%%%%%%%%%%%%%%%%%%%%%%%%%%%%%%%%%%%%%%%%%%%%%%%%%

Let $X$ be a {\it real} Banach space and let $X^*$ be its dual space.
Suppose that $x\in X$. Denote by $J(x)=J_X(x)$ the subset of $X^*$ defined by
\[
J(x)=\{l\in X^*\ |\ \n{l}_{X^*}=\n{x}_{X}\ {\rm and}\
\la l,x\ra:=l(x)=\n{x}_{X}\n{l}_{X^*}\}.
\]
In the product space $X=X_1\times X_2\times X_3$ we use the norm
\[
\n{x}_{X_1\times X_2\times X_3}=\sum_{j=1}^3\n{x_j}_{X_j},
\quad x=(x_1,x_2,x_3)\in X.
\]
One has that $X^*=X_1^*\oplus X_2^*\oplus X_3^*$ in the sense that for any $l\in X^*$,
\[
l(x)=\sum_{j=1}^3\la l_j,x_j\ra,
\]
where $l_j:=l_{|X_j}\in X_j^*$,
and the corresponding norm is given by $\n{l}_{X^*}=\max_{1\leq j\leq 3}\n{l_j}_{X_j^*}$.
The structure of $J_X(x)$ can be obtained by applying iteratively the following lemma.

\begin{lemma}\label{dualfle}
Let $Y_1,Y_2$ be Banach spaces, $Y=Y_1\oplus Y_2$ their product
equipped with the $1$-norm $\n{y}_Y=\n{y_1}_{Y_1}+\n{y_2}_{Y_2}$, $y=(y_1,y_2)\in Y$ like above.
Then for every $y=(y_1,y_2)\in Y$ we have
\begin{align}\label{eq:prod_duality}
J_Y(y)=\begin{cases}
\hspace{0.35cm} J_{Y_1}\big(z_1(y)\big)\times J_{Y_2}\big(z_2(y)\big), & \textrm{if}\ y_1\neq 0\ \textrm{and}\ y_2\neq 0, \\
\ol{B}_{Y_1^*}(\n{y_2}_{Y_2})\times J_{Y_2}(y_2), & \textrm{if}\ y_1=0, \\
\hspace{0.95cm} J_{Y_1}(y_1)\times \ol{B}_{Y_2^*}(\n{y_1}_{Y_1}), & \textrm{if}\ y_2=0,
\end{cases}
\end{align}
where $\ol{B}_{Y_j^*}(r)$ denotes the closed ball of radius $r>0$ in $Y_j^*$,
\[
z_j(y):=\frac{\n{y}_Y}{\n{y_j}_{Y_j}}y_j,\quad \textrm{when}\ y_j\neq 0,\ j=1,2.
\]
\end{lemma}

\begin{proof}
Let $l\in J_Y(y)$.
Then by the definition of $J_Y$
\begin{align}\label{eq:duality:1}
& \n{l}_{Y^*}=\n{y}_{Y}=\n{y_1}_{Y_1}+\n{y_2}_{Y_2} \\
& (\n{y_1}_{Y_1}+\n{y_2}_{Y_2})\n{l}_{Y^*}= \n{y}_Y\n{l}_{Y^*} = l(y)=l_1(y_1)+l_2(y_2). \nonumber
\end{align}
Recalling that $\n{l}_{Y^*}=\max\{\n{l_1}_{Y_1^*}, \n{l_2}_{Y_2^*}\}$,
the second line above implies that
\[
(\n{y_1}_{Y_1}+\n{y_2}_{Y_2})\n{l}_{Y^*}
=&l_1(y_1)+l_2(y_2)\leq \n{l_1}_{Y_1^*}\n{y_1}_{Y_1}+\n{l_2}_{Y_2^*}\n{y_2}_{Y_2} \\
\leq & \n{l}_{Y^*}(\n{y_1}_{Y_1}+\n{y_2}_{Y_2}),
\]
i.e.
\[
(\n{y_1}_{Y_1}+\n{y_2}_{Y_2})\n{l}_{Y^*}=l_1(y_1)+l_2(y_2)=\n{l_1}_{Y_1^*}\n{y_1}_{Y_1}+\n{l_2}_{Y_2^*}\n{y_2}_{Y_2}.
\]
Taking into account the fact that $\n{l}_{Y^*}\geq \n{l_1}_{Y_1^*},\n{l_2}_{Y_2^*}$,
one can conclude from the above equality that
\begin{align}\label{eq:duality:2}
l_1(y_1)=&\n{y_1}_{Y_1}\n{l_1}_{Y_1^*}=\n{y_1}_{Y_1}\n{l}_{Y^*} \\
l_2(y_2)=&\n{y_2}_{Y_2}\n{l_2}_{Y_2^*}=\n{y_2}_{Y_2}\n{l}_{Y^*}. \nonumber
\end{align}

Assume first that $y_1=0$. If $y_2=0$ as well, we have $\n{l}_{Y^*}=0$
from the first line of \eqref{eq:duality:1}
and thus $\n{l_1}_{Y_1^*}=\n{l_2}_{Y_2^*}=0$,
which means that 
$(l_1,l_2)=(0,0)\in \ol{B}_{Y_1^* }(\n{y_2}_{Y_2})\times J_{Y_2}(y_2)$.
On the other hand, if $y_2\neq 0$, the second line of \eqref{eq:duality:2}
implies that $\n{l}_{Y^*}=\n{l_2}_{Y_2^*}$,
and hence from \eqref{eq:duality:1}, we have $\n{l_2}_{Y_2^*}=\n{y_2}_{Y_2}$.
Because $\n{l_1}_{Y_1^*}\leq \n{l}_{Y^*}=\n{y_2}_{Y_2}$, this shows that 
$(l_1,l_2)\in \ol{B}_{Y_1^* }(\n{y_2}_{Y_2})\times J_{Y_2}(y_2)$.

The case where $y_2=0$ and $y_1$ is arbitrary is handled similarly.

We may thus assume that both $y_1$ and $y_2$ are non-zero.
Using \eqref{eq:duality:2} and the definition of $z_j(y)$ as given above,
one has
\[
l_j(z_j(y))=\frac{\n{y}_Y}{\n{y_j}_{Y_j}}l_j(y_1)=\n{y}_Y\n{l_j}_{Y_j^*}=\n{z_j(y)}_{Y_j}\n{l_j}_{Y_j^*},\quad j=1,2.
\]
On the other hand, \eqref{eq:duality:2} implies that $\n{l}_{Y^*}=\n{l_1}_{Y_1^*}=\n{l_2}_{Y_2^*}$,
and hence by using \eqref{eq:duality:1},
\[
\n{z_j(y)}_{Y_j}=\n{y}_Y=\n{l}_{Y^*}=\n{l_j}_{Y_j^*}.
\]
By the definition of the duality set, then,
we conclude that $l_j\in J_{Y_j}(z_j(y))$, $j=1,2$.

We have thus shown that the set $J_Y(y)$
is a subset of the right hand side of \eqref{eq:prod_duality},
taking the appropriate cases into account.
The reverse inclusion is readily verified by checking, case by case, the validity
of both of the lines in \eqref{eq:duality:1}.
\end{proof}

In this section, we assume that $A:D(A)\subset X\to X$ is \emph{densily defined},
i.e. the domain of definition $D(A)$ of $A$ is dense in $X$.
In addition, instead of $A:D(A)\subset X\to X$ we usually write simply $A:X\to X$.

\begin{definition}
\begin{itemize}
\item[(i)] An (unbounded) linear operator $A:X\to X$ is said to be \emph{dissipative}, if for each $x\in D(A)$ there exists $l\in J(x)$ such that
\be\label{d1}
\la l,Ax\ra \leq 0.
\ee
The operator $A:X\to X$ is said to be \emph{accretive}, if  $-A$ is dissipative.

\item[(ii)] A dissipative operator $A:X\to X$ is \emph{$m$-dissipative}, if there exists $\lambda>0$ such that
\[
R(\lambda I-A)=X,
\]
where $R(\lambda I-A)$ is the range of $\lambda I-A$ and $I$ is the identity operator.
\end{itemize}
\end{definition}

One knows that if an operator $A:X\to X$ is  dissipative and if there exists $\lambda_0 >0$ such that
$R(\lambda_0 I-A)=X$ then $R(\lambda I-A)=X$  for every $\lambda >0$ (\cite{pazy83}, Section 1.4, \cite{engelnagel}, Section II.3.b).
On the other hand, the condition $R(\lambda I-A)=X$ is
equivalent to $\lambda\in\rho(A)$ (the resolvent set of $A$)
in the case when $A$ is dissipative, as follows from the theorem we present next.

\begin{theorem}\label{th:dissipative}
A linear operator $A:X\to X$ is dissipative if and only if for all $\lambda >0$ the estimate
\begin{align}\label{eq:dissipative}
\n{(\lambda I-A)x}\geq \lambda \n{x},\quad \forall x\in D(A),
\end{align}
holds.
\end{theorem}

\begin{proof}For the proof we refer to \cite{engelnagel}, Section II.3.b or \cite{pazy83}, Section 1.4.
\end{proof}

In particular, $m$-dissipative operator $A$ is closed since $\rho(A)\neq\emptyset$.
We also have the following (bounded) perturbation result for $m$-dissipative operators.

\begin{theorem}\label{perth}
Suppose that a  closed operator $A:X\to X$ is $m$-dissipative
and that $B:X\to X$ is a bounded dissipative operator. Then $A+B:X\to  X$ is $m$-dissipative.
\end{theorem}

\begin{proof}
See \cite{pazy83} (Chap. 1, Theorem 4.3 and Chap. 3, Corollary 3.3)
or \cite{engelnagel} (Chap. II, Theorem 3.15 and Chap. III, Theorem 2.7).

Here is a sketch of the proof.
Given $x\in D(A)$, since $B$ is dissipative and bounded ($D(B)=X$),
there exists $l\in J(x)$, such that $\la l,Bx\ra\leq 0$.
But because $A$ is $m$-dissipative,
one also has $\la l,Ax\ra\leq 0$ (\cite{pazy83}, Theorem 4.3 (b)), and hence
\[
\la l,(A+B)x\ra\leq 0,
\]
which shows that $A+B$ is dissipative.

It remains to show that $\lambda I-(A+B)$ is surjective for some $\lambda>0$.
Taking $\lambda >\n{B}$, the $m$-dissipativity of $A$ implies that $\lambda\in \rho(A)$ and
\[
\n{(\lambda I-A)^{-1}B}\leq \frac{\n{B}}{\lambda}<1,
\]
which shows that $I-(\lambda I-A)^{-1}B$ has a bounded inverse. On the other hand,
\[
\lambda I-(A+B)=(\lambda I-A)(I-(\lambda I-A)^{-1}B),
\]
which shows that $\lambda I-(A+B)$ has a bounded inverse. In particular, 
$\lambda I-(A+B)$ is surjective and the proof is complete.
\end{proof}

In the case where $X=L^1(G\times S\times I)$ we have $L^1(G\times S\times I)^*=
L^\infty(G\times S\times I)$ isomorphically (and isometrically) and for $l\in L^1(G\times S\times I)^*$ and $\psi\in L^1(G\times S\times I)$ one has (recall that we have everywhere real spaces)
\[
l(\psi)=\la w,\psi\ra=\int_{G\times S\times I} w\psi dx d\omega dE
\]
where $w\in L^1(G\times S\times I)^\infty$ is corresponding to
$l\in L^1(G\times S\times I)^*$ through the above mentioned isomorphism.

It is well known that
that for $\psi\in L^1(G\times S\times I)$
\[
J(\psi)=\{w\in L^\infty(G\times S\times I)\ |\ w=\n{\psi}_{L^1(G\times S\times I)}\psi^*\}
\]
(cf. \cite{dautraylionsv5} Chapter XVII, section 3.2, p. 344, or use Theorem 1.40 in \cite{rudin87})
where
\[
\psi^*(x,\omega,E)=\begin{cases}
1,\ &\psi(x,\omega,E)>0\\-1,\ &\psi(x,\omega,E)<0
\end{cases}
\]
and $\psi^*$ is a measurable function for which $|\psi^*(x,\omega,E)|\leq 1$ when $\psi(x,\omega,E)=0$.
As defined above a linear operator $A:L^1(G\times S\times I)\to L^1(G\times S\times I)$ is  dissipative, if for each $\psi\in D(A)$  there exists $w\in J(\psi)$ such that
\be\label{d1a}
\la w,A\psi\ra =\int_{G\times S\times I} wA\psi\ dxd\omega dE\leq 0.
\ee
Assume that $\psi\not =0$. We choose $\psi^*(x,\omega,E)=0$ when $\psi(x,\omega,E)=0$ . Then the condition (\ref{d1a}) (for that $w$) means that (here sign($\psi$) is the signum function)
\be
\int_{G\times S\times I} \n{\psi}_{L^1(G\times S\times I)}\psi^*A\psi dxd\omega dE 
=\n{\psi}_{L^1(G\times S\times I)}\int_{G\times S\times I}{\rm sign}(\psi) A\psi  dx d\omega dE
\leq 0
\ee
that is
\be\label{d2:1}
\int_{G\times S\times I}{\rm sign}(\psi) A\psi  dx d\omega dE\leq 0.
\ee

%%%%%%%%%%%%%%%%%%%%%%%%%%%%%%%%%%%%%%%%%%%%%%%%%%%%%%%%%%%%%%%%%%%%%%%%
\subsection{$m$-dissipativity of the Convection Operator}\label{diss1b}
%%%%%%%%%%%%%%%%%%%%%%%%%%%%%%%%%%%%%%%%%%%%%%%%%%%%%%%%%%%%%%%%%%%%%%%%

Let
\[
\tilde W^1_{-,0}(G\times S\times I)=\{\psi\in \tilde W^1(G\times S\times I)\ |\ \psi_{|\Gamma_-}=\gamma_-(\psi)=0\}.
\]
Furthermore, let $A:L^1(G\times S\times I)\to L^1(G\times S\times I)$ and
$A_0:L^1(G\times S\times I)\to L^1(G\times S\times I)$ be  linear operators defined by
\[
D( A)= W^1(G\times S\times I),\quad  A\psi=-\omega\cdot\nabla\psi.
\] 
and 
\[
D(A_0)=\tilde W^1_{-,0}(G\times S\times I),\quad  A_0\psi=-\omega\cdot\nabla\psi.
\]
Thus the domain of $A_0$ (so called {\it realization}) consists of those $\psi\in \tilde W^1(G\times S\times I)$ for which $\psi_{|\Gamma_-}=\gamma_-(\psi)=0$.

\begin{proposition}
The linear operator $A_0$ is closed and densely defined.
\end{proposition}

\begin{proof}
The domain $D(A_0)$ is dense in $L^1(G\times S\times I)$ since $C_0^1(G\times S\times I)$
is dense in $L^1(G\times S\times I)$.

That $A_0$ is closed can be seen as follows. Let $f,\ \psi\in L^1(G\times S\times I)$ and let $\{\psi_n\}\subset D(A_0)=\tilde W^1_{-,0}(G\times S\times I)$ be such that $\n{\psi_n-\psi}_{L^1(G\times S\times I)}\to 0$ and $\n{A_0\psi_n-f}_{L^1(G\times S\times I)}\to 0$ as $n\to\infty$.  Then  $\{\psi_n\}\subset\tilde W^1(G\times S\times I)$ is a Cauchy sequence in $ W^1(G\times S\times I)$ and hence there exists an element $\psi'\in  W^1(G\times S\times I)$ such that $\psi_n\to\psi'$ in
$W^1(G\times S\times I)$. As the latter space is continuously embedded in $L^1(G\times S\times I)$, we have $\psi_n\to\psi$ also in $L^1(G\times S\times I)$,
therefore $\psi=\psi'$ and so $\psi_n\to\psi$ in  $W^1(G\times S\times I)$. Because the trace mapping $\gamma_-:W^1(G\times S\times I)\to L^1_{\rm loc}(\Gamma_-,|\omega\cdot\nu|d\sigma d\omega dE)$  is continuous we get that $\gamma_-(\psi_n)\to \gamma_-(\psi)$ in  $L^1_{\rm loc}(\Gamma_-,|\omega\cdot\nu|d\sigma d\omega dE)$ and since $\gamma_-(\psi_n)=0$, also $\gamma_-(\psi)=0$. Hence $\psi\in \tilde W^1_{-,0}(G\times S\times I)=D(A_0)$ and $A_0\psi=-\omega\cdot\nabla\psi=\lim_{n\to\infty}-\omega\cdot\nabla\psi_n=f$, which shows that  $A_0$ is closed.
\end{proof}

\begin{lemma}\label{dle2}
Let $f\in C(\ol G\times S\times I)$ such that  ${\p f{x_j}}\in  C(\ol G\times S\times I)$ for $j=1,2,3$.
Then the (classical) solution (cf. Theorem \ref{lth1}) $\psi:G\times S\times I\to\R$ of the equation ($\lambda\in\R$)
\[
\omega\cdot \nabla\psi+\lambda\psi=f(x,\omega,E)
\quad\iff\quad (\lambda I-A)\psi=f
\]
defined by $\psi(x,\omega,E)=\int_0^{t(x,\omega)}e^{-\lambda t}f(x-t\omega,\omega,E)dt$
belongs to $\tilde W_{-,0}^1(G\times S\times I)=D(A_0)$. In addition for any $\lambda >0$
\be\label{l18}
\n{\psi}_{L^1(G\times S\times I)}\leq {1\over{\lambda}} \n{f}_{L^1(G\times S\times I)}=
{1\over{\lambda}} \n{(\lambda I-A_0)\psi}_{L^1(G\times S\times I)}.
\ee
\end{lemma}

\begin{proof}
Due to Remark \ref{lre2}  $\psi$ is measurable in $G\times S\times I$.
We show that $\psi\in W^1(G\times S\times I)$. Since   $\omega\cdot\nabla\psi+\lambda\psi=f\in C(\ol G\times S\times I)\subset L^1(G\times S\times I)$ it suffices to verify that $\psi\in L^1(G\times S\times I)$.
Denoting by $\ol f$ the extension by zero of $f$ on $\R^3\times S\times I$,
we have
\bea\label{wi}
\psi(x,\omega,E)=\int_0^\infty e^{-\lambda t}\ol f(x-t\omega,\omega,E)\chi_{[0,t(x,\omega)]}(t) dt
\eea
where $\chi_{[0,t(x,\omega)]}$ 
is the characteristic function of the interval $[0,t(x,\omega)]$ (note that the integrand of (\ref{wi}) is measurable).

Hence applying the change of variables $x-t\omega=z$ (in $x$-variable) we obtain
\bea\label{l19:1}
\n{\psi}_{L^1(G\times S\times I)}&=\int_{G\times S\times I}|\psi(x,\omega,E)| dx d\omega dE \nonumber\\
&\leq
\int_0^\infty e^{-\lambda t}\int_G\int_{S\times I} |\ol f(x-t\omega,\omega,E)|
\chi_{[0,t(x,\omega)]}(t) dx d\omega dE dt\nonumber\\
&=\int_0^\infty e^{-\lambda t}\int_{S\times I}\int_{(G-t\omega)\cap G} |\ol f(z,\omega,E)|
\chi_{[0,t(z+t\omega,\omega)]}(t) dz d\omega dE dt\nonumber\\
&\leq
\int_0^\infty e^{-\lambda t}\int_G\int_{S\times I} |f(z,\omega,E)| dz d\omega dE dt
={1\over{\lambda}} \n{f}_{L^1(G\times S\times I)}.
\eea
Hence $\psi\in L^1(G\times S\times I)$ and the estimate (\ref{l18}) holds.

By Theorem \ref{lth1} the inflow boundary condition $\psi(y,\omega,E)=0$ is true a.e. $(y,\omega,E)\in\Gamma_-$ (in the classical sense).
Hence the proof is complete.
\end{proof}

The following theorem is shown by different methods in \cite{dautraylionsv6}, section XXI.\S 2, Theorem 2 and Remark 3 (pp. 222-224). An alternative proof is also given in section XXI.\S 2, Prop. 5 (pp. 242-243) in \cite{dautraylionsv6}.
 
\begin{theorem}\label{dth1}
The operator $A_0:L^1(G\times S\times I)\to L^1(G\times S\times I)$ is $m$-dissipative.
\end{theorem}

Proof. A. We  show that $R(\lambda I-A_0)=L^1(G\times S\times I)$ for any $\lambda >0$. Let $f\in L^1(G\times S\times I)$. Then there exists a sequence $\{f_n\}\subset \mc{D}(\ol G\times S\times I)$ such that $\n{f_n-f}_{L^1(G\times S\times I)}\to 0$ when $n\to\infty$. By Lemma \ref{dle2}
there exists $\psi_n\in D(A_0)=\tilde W^1_{-,0}(G\times S\times I)$ such that $(\lambda I-A_0)\psi_n=f_n$ and
\be\label{d3}
\n{\psi_n-\psi_m}_{L^1(G\times S\times I)}\leq  {1\over \lambda}\n{(\lambda I-A_0)(\psi_n-\psi_m)}_{L^1(G\times S\times I)}= {1\over \lambda}\n{f_n-f_m}_{L^1(G\times S\times I)}
\ee
which implies that $\psi_n\to\psi$ in $L^1(G\times S\times I)$ for some $\psi\in L^1(G\times S\times I)$. Since also $(\lambda I-A_0)\psi_n=f_n\to f$ in $L^1(G\times S\times I)$ and since $\lambda I-A_0$ is closed we obtain that $\psi\in D(A_0)$ and $(\lambda I-A_0)\psi=f$. 

B. Since by Lemma \ref{dle2} again, for all $n\in\N$ and $\lambda>0$ the estimate
\[
\n{(\lambda I-A_0)\psi_n}\geq \lambda \n{\psi_n}
\]
is valid we see that 
\be\label{da0}
\n{(\lambda I-A_0)\psi}\geq \lambda \n{\psi}\ {\rm for\ all}\ \psi\in D(A_0).
\ee
Due to Theorem \ref{th:dissipative} $A_0$ is dissipative and hence by Part A of the proof $A_0$ is $m$-dissipative.
This completes the proof.

\begin{remark}
An alternative proof for the dissipativity of $A_0$ can be seen by applying the Green formula (see e.g. \cite{dautraylionsv6}, pp. 242-243) as follows.
One knows that $|\psi|\in  W^1(G\times S\times I)$ when $\psi\in  W^1(G\times S\times I)$ and
(cf. \cite{grigoryan09}, Sections 5.1--5.2)
\be\label{d2a}
\omega\cdot\nabla(|\psi|)={\rm sign}(\psi)\omega\cdot\nabla\psi
\ee
in $W^1(G\times S\times I)$.
Applying \eqref{eq:stokes} for $u=|\psi|$ then gives
\begin{multline*}
-\int_{G\times S\times I}{\rm sign}(\psi) A_0\psi  dxd\omega dE
=
\int_{G\times S\times I}\omega\cdot\nabla(|\psi|)\ dxd\omega dE \\
=\int_{\partial G\times S\times I}(\omega\cdot \nu)|\psi|\ d\sigma d\omega dE
=\int_{\Gamma_+}(\omega\cdot\nu)|\psi|\ d\sigma d\omega dE\geq 0,
\end{multline*}
since  $\int_{\Gamma_-}(\omega\cdot\nu(y))|\psi|\ d\sigma d\omega dE=0$ (recall that $\psi=0$ on $\Gamma_-$), $\nabla v=0$ and $\omega\cdot\nu(y)>0$ on $\Gamma_+$.
Hence \eqref{d2:1} holds, and so $A_0$ is dissipative for the reasons explained at the end of the corresponding section \ref{diss1a}.
\end{remark}

Lemma \ref{dle2} and the proof of the above theorem imply that for $\lambda>0$ the solution $\psi\in D(A_0)$ of the equation
\[
(\lambda I-A_0)\psi =f,\quad f\in L^1(G\times S\times I),
\]
is given by 
\be 
\psi=\lim_{n\to\infty}\psi_n=\lim_{n\to\infty}\left(
\int_0^{t(x,\omega)}f_n(x-t \omega,\omega,E)e^{-\lambda t}\ dt\right)
=\int_0^{t(x,\omega)}f(x-t \omega,\omega,E)e^{-\lambda t}\ dt,
\ee
almost everywhere on $G\times S\times I$.
Hence for $\lambda >0$ the resolvent  $(\lambda I-A_0)^{-1}:L^1(G\times S\times I)\to L^1(G\times S\times I)$ is given explicitly by
\be\label{resolA}
(\lambda I-A_0)^{-1}f=\int_0^{t(x,\omega)}f(x-t \omega,\omega,E)e^{-\lambda t}\ dt
\ee
and the resolvent satisfies the estimate
\[
\n{(\lambda I-A_0)^{-1}f}_{L^1(G\times S\times I)}\leq {1\over{\lambda}} \n{f}_{L^1(G\times S\times I)},\ \lambda >0.
\]

%%%%%%%%%%%%%%%%%%%%%%%%%%%%%%%%%%%%%%%%%%%%%%%%%%%%%%%%%%%%%%%%%%%%%%%%
\section{Coupled Boltzmann Transport Equation}\label{co}
%%%%%%%%%%%%%%%%%%%%%%%%%%%%%%%%%%%%%%%%%%%%%%%%%%%%%%%%%%%%%%%%%%%%%%%%

%%%%%%%%%%%%%%%%%%%%%%%%%%%%%%%%%%%%%%%%%%%%%%%%%%%%%%%%%%%%%%%%%%%%%%%%
\subsection{$m$-dissipativity of Cartesian Product Convection Operator}\label{cos1}
%%%%%%%%%%%%%%%%%%%%%%%%%%%%%%%%%%%%%%%%%%%%%%%%%%%%%%%%%%%%%%%%%%%%%%%%

As we mentioned above in Section \ref{diss1a}
in the Cartesian product space $L^1(G\times S\times I)^3$ we use the norm
\[
\n{\psi}_{L^1(G\times S\times I)^3}=\sum_{j=1}^3\n{\psi_j}_{L^1(G\times S\times I)},\
\psi=(\psi_1,\psi_2,\psi_3).
\]
and similarly in its subspaces $X^3\subset L^1(G\times S\times I)^3,\ X\subset L^1(G\times S\times I)$ we use the norms
\[
\n{\psi}_{X^3}=\sum_{j=1}^3\n{\psi_j}_{X},\
\psi=(\psi_1,\psi_2,\psi_3).
\]

Define linear operators ${\bf A}$ and ${\bf A}_0:L^1(G\times S\times I)^3\to L^1(G\times S\times I)^3$ by 
\[
&D({\bf A})= W^1(G\times S\times I)^3\nonumber\\
&{\bf A}\psi = (-\omega\cdot\nabla\psi_1,-\omega\cdot\nabla\psi_2,-\omega\cdot\nabla\psi_3).
\]
and
\[
D({\bf A}_0)= \tilde{W}^1_{-,0}(G\times S\times I)^3,\quad {\bf A}_0\psi={\bf A}\psi.
\]
We see that
\[
{\bf A}\psi=\qmatrix{
A&0&0\cr 0&A&0\cr 0&0&A\cr}
\qmatrix{\psi_1\cr\psi_2\cr\psi_3
}
\]
where $A:L^1(G\times S\times I)\to L^1(G\times S\times I)$ is above in Section \ref{diss1} defined operator and similarly for ${\bf A}_0$.

Since $A_0$ (whose domain is $D(A_0)=\tilde{W}^1_{-,0}(G\times S\times I)$) is a closed densily defined operator we see that ${\bf A}_0:L^1(G\times S\times I)^3\to L^1(G\times S\times I)^3$ is a closed densily defined operator.

\begin{theorem}\label{coth1}
The operator  ${\bf A}_0:L^1(G\times S\times I)^3\to L^1(G\times S\times I)^3$ is $m$-dissipative.
\end{theorem}

\begin{proof}
We find by (\ref{da0}) that for all $\psi\in D({\bf A}_0)$ and $\lambda>0$
\begin{multline*}
\n{(\lambda I-{\bf A}_0)\psi}_{L^1(G\times S\times I)^3}
=
\sum_{j=1}^3\n{(\lambda I-{ A}_0)\psi_j}_{L^1(G\times S\times I)} \\
\geq
\lambda \sum_{j=1}^3\n{\psi_j}_{L^1(G\times S\times I)}=\lambda\n{\psi}_{L^1(G\times S\times I)^3},
\end{multline*}
and then ${\bf A}_0$ is dissipative by Theorem \ref{th:dissipative}.

We verify that ${\bf A}_0$ is $m$-dissipative that is, in addition to dissipativity one has $R(\lambda I-{\bf A}_0)=L^1(G\times S\times I)^3$ for (any) $\lambda >0$. Let $f=(f_1,f_2,f_3)\in
L^1(G\times S\times I)^3$. Then by Theorem \ref{dth1} for any $j=1,2,3$ there exists $\psi_j\in D(A_0)=\tilde W^1_{-,0}(G\times S\times I) $ such that $(\lambda I-A_0)\psi_j=f_j$ and so
$R(\lambda I-{\bf A}_0)=L^1(G\times S\times I)^3$ for any $\lambda >0$. This completes the proof.
\end{proof}

%%%%%%%%%%%%%%%%%%%%%%%%%%%%%%%%%%%%%%%%%%%%%%%%%%%%%%%%%%%%%%%%%%%%%%%%
\subsection{Dissipativity of Scattering-Collision Operator}\label{coss2}
%%%%%%%%%%%%%%%%%%%%%%%%%%%%%%%%%%%%%%%%%%%%%%%%%%%%%%%%%%%%%%%%%%%%%%%%

Let $\Sigma_j:G\times S\times I\to\R,\ j=1,2,3$ be functions,
the so-called {\it total cross sections}, such that
\be\label{scateh}
\Sigma_j\in L^\infty(G\times S\times I),\quad \Sigma_j\geq 0\quad {\rm a.e.\ in}\ G\times S\times I,\ j=1,2,3.
\ee
Furthermore, let  $\sigma_{kj}:G\times S^2\times I^2\to\R,\ 1\leq k,j\leq 3$ be  measurable functions, the so-called {\it differential cross sections},  such that the {\it Schur conditions}
\begin{gather}
\sum_{k=1}^3\int_{S\times I}\sigma_{jk}(x,\omega,\omega',E,E')d\omega' dE'\leq C \quad {\rm a.e.}\ (x,\omega,E)\in
G\times S\times I,\nonumber\\
\sigma_{kj}\geq 0\quad {\rm a.e.}\ G\times S^2\times I^2,\ k,j=1,2,3,\label{colleh}
\end{gather}
and
\be\label{colleha}
&\sum_{k=1}^3\int_{S\times I}\sigma_{kj}(x,\omega',\omega,E',E)d\omega' dE'\leq C \quad {\rm a.e.}\ (x,\omega,E)\in
G\times S\times I,\ j=1,2,3
\ee
hold.
In the case $p=1$ we will only need the condition (\ref{colleh}).

Define the {\it scattering operator} $\Sigma_j$ and the {\it collision operator} $K_j$ corresponding to the particle $j$ for $j=1,2,3$ as follows
\be\label{scat}
(\Sigma_j\psi_j)(x,\omega,E)=\Sigma_j(x,\omega,E)\psi_j(x,\omega,E),\ \psi_j\in
 L^1(G\times S\times I)
\ee
and
\be\label{coll}
(K_j\psi)(x,\omega,E)=\sum_{k=1}^3\int_{S\times I}\sigma_{kj}(x,\omega',\omega,E',E)\psi_k(x,\omega',E')d\omega' dE' ,
\ee
where $\psi\in L^1(G\times S\times I)^3$.
Furthermore, we define for $\psi\in L^1(G\times S\times I)^3$
\be\label{}
\Sigma\psi=(\Sigma_1\psi_1,\Sigma_2\psi_2,\Sigma_3\psi_3)
\ee
and
\be\label{}
K\psi=(K_1\psi,K_2\psi,K_3\psi).
\ee
The operators $\Sigma$ and $K$ are linear and continuous,
as we formulate next.

\begin{theorem}\label{NoLabel}
The operators $\Sigma$ and $K$ are bounded linear maps  $L^1(G\times S\times I)^3\to L^1(G\times S\times I)^3$.
\end{theorem}

\begin{proof}
We see that
\bea
\n{\Sigma_j\psi_j}_{L^1(G\times S\times I)}
=&\int_{G\times S\times I}|\Sigma_j(x,\omega,E)\psi_j(x,\omega,E)| dx d\omega dE\nonumber \\
\leq& 
\n{\Sigma_j}_{L^\infty(G\times S\times I)}\n{\psi_j}_{L^1(G\times S\times I)}\nonumber
\eea
and then 
\be\label{scb1}
\n{\Sigma\psi}_{L^1(G\times S\times I)^3}
\leq \max_{1\leq j\leq 3}\n{\Sigma_j}_{L^\infty(G\times S\times I)}\n{\psi}_{L^1(G\times S\times I)^3}.
\ee
Furthermore,
\bea
\n{K_j\psi}_{L^1(G\times S\times I)}
=&\int_{G\times S\times I}\Big|\Big(\sum_{k=1}^3\int_{S\times I}\sigma_{kj}(x,\omega',\omega,E',E)\psi_k(x,\omega',E') d\omega' dE'\Big)\Big| dx d\omega dE\nonumber\\
\leq &
\int_G\Big(\int_{S\times I}\Big(\sum_{k=1}^3\int_{S\times I}\sigma_{kj}(x,\omega',\omega,E',E) d\omega dE\Big)|\psi_k(x,\omega',E')| d\omega' dE'\Big) dx,
\eea
and then by the assumption (\ref{colleh})
\bea\label{scb2}
\n{K\psi}_{L^1(G\times S\times I)^3}\leq
C\sum_{k=1}^3\int_{G}\int_{S\times I}|\psi_k(x,\omega',E')| d\omega' dE' dx=C\n{\psi}_{L^1(G\times S\times I)^3}.
\eea
The assertion follows from (\ref{scb1}) and (\ref{scb2}). 
\end{proof}

In order that the operator $-(\Sigma-K)=-\Sigma+K:L^1(G\times S\times I)^3\to L^1(G\times S\times I)^3$
would be dissipative we assume that the cross-sections satisfy the following condition: 

There exists $c\geq 0$ such that for every $j=1,2,3$ (cf. \cite{dautraylionsv6}, pp. 241 for one particle and \cite{tervo07}, \cite{bomanthesis} for coupled system)
\be\label{co2a}
\Sigma_j(x,\omega,E)-\sum_{k=1}^3\int_{S\times I}\sigma_{jk}(x,\omega,\omega',E,E') d\omega' dE'
\geq c\quad {\rm a.e.}\ (x,\omega,E)\in G\times S\times I.
\ee
and
\be\label{co2aa}
\Sigma_j(x,\omega,E)-\sum_{k=1}^3\int_{S\times I}\sigma_{kj}(x,\omega',\omega,E',E) d\omega' dE'
\geq c\quad {\rm a.e.}\ (x,\omega,E)\in G\times S\times I.
\ee
When considering $L^1$-solutions we need only the assumption (\ref{co2a}). 

We show next the following dissipativity type result for $-\Sigma+K$.

\begin{theorem}\label{dfsco}
Suppose that the assumptions (\ref{scateh}), (\ref{colleh}) and (\ref{co2a}) are valid for some constant $c\geq 0$.
Then the operator
$-\Sigma+K$ satisfies the following dissipativity condition: For all $\lambda >0$ and $\psi\in L^1(G\times S\times I)^3$ one has
\be\label{co2b}
\n{\big(\lambda I-(-\Sigma+K+cI)\big)\psi}_{L^1(G\times S\times I)^3}
\geq \lambda\n{\psi}_{L^1(G\times S\times I)^3}.
\ee
In other words, the operator $-\Sigma+K+cI:L^1(G\times S\times I)^3\to
L^1(G\times S\times I)^3$ is dissipative.
\end{theorem}

\begin{proof}
We have for any $\psi\in L^1(G\times S\times I)^3$ and $\lambda >0$,
\bea\label{53p1}
&\n{(\lambda I-(-\Sigma+K+cI))\psi}_{L^1(G\times S\times I)^3}
=\sum_{j=1}^3\n{(\lambda I-cI+\Sigma_j)\psi_j -K_j\psi}_{L^1(G\times S\times I)}\nonumber\\
=&
\sum_{j=1}^3\int_{G\times S\times I}\Big|(\lambda -c+\Sigma_j(x,\omega,E))\psi_j(x,\omega,E) -(K_j\psi)(x,\omega,E)\Big|dx d\omega dE\nonumber\\
\geq &
\sum_{j=1}^3\int_{G\times S\times I}\Big(|\lambda -c+\Sigma_j(x,\omega,E)|\ |\psi_j(x,\omega,E)| -|(K_j\psi)(x,\omega,E)|\Big)dx d\omega dE.
\eea
Furthermore, we have
\[
|(K_j\psi)(x,\omega,E)|\leq
 \sum_{k=1}^3 \int_{S\times I}\sigma_{kj}(x,\omega',\omega,E',E)|\psi_k(x,\omega',E')| d\omega' dE',
\]
and by the assumption (\ref{co2a}) 
\[
\lambda-c+\Sigma_j(x,\omega,E)\geq \lambda+
 \sum_{k=1}^3 \int_{S\times I}\sigma_{jk}(x,\omega,\omega',E,E') d\omega' dE'>0,
\]
which, when combined with (\ref{53p1}), give
\bea\label{53p2}
&\n{(\lambda I-(-\Sigma+K+cI))\psi}_{L^1(G\times S\times I)^3}\nonumber \\
\geq &
\sum_{j=1}^3\int_G\Big[\int_{S\times I}\Big(\lambda +
\sum_{k=1}^3 \int_{S\times I}\sigma_{jk}(x,\omega,\omega',E,E') d\omega' dE'\Big)
|\psi_j(x,\omega,E)| d\omega dE\nonumber\\
&-
\int_{S\times I}\Big(
\sum_{k=1}^3\int_{S\times I}
\sigma_{kj}(x,\omega',\omega,E',E) 
|\psi_k(x,\omega',E')| d\omega' dE'\Big) d\omega dE\Big] dx\nonumber\\
=&
\lambda \n{\psi}_{L^1(G\times S\times I)^3} \\
&+
\int_G\Big[\sum_{j=1}^3\sum_{k=1}^3\Big(\int_{S\times I}\int_{S\times I}
\sigma_{jk}(x,\omega,\omega',E,E') 
|\psi_j(x,\omega,E)| d\omega' dE' d\omega dE\nonumber\\
&-
\int_{S\times I}\int_{S\times I}
\sigma_{kj}(x,\omega',\omega,E',E) 
|\psi_k(x,\omega',E')| d\omega' dE'  d\omega dE\Big)\Big] dx. \nonumber
\eea
Writing,
\[
A_{jk}(x,\omega,E)&:=\int_{S\times I}\sigma_{jk}(x,\omega,\omega',E,E') d\omega' dE', \\
B_{jk}(x)&:=\int_{S\times I}A_{jk}(x,\omega,E)|\psi_j(x,\omega,E)| d\omega dE,
\]
we see that the last two terms on the right hand side of the above formula (\ref{53p2}) can be written as
\[
\int_G\sum_{j=1}^3\sum_{k=1}^3 (B_{jk}(x)-B_{kj}(x))dx=0,
\]
which allows us to conclude that
\[
\n{\big(\lambda I-(-\Sigma+K+cI)\big)\psi}_{L^1(G\times S\times I)^3}\geq\lambda \n{\psi}_{
L^1(G\times S\times I)^3}.
\]
This completes the proof.
\end{proof}

\begin{remark}\label{cor1}
Theorem \ref{dfsco} also implies (by substituting $\lambda+c$ for $\lambda$) that 
\[
\n{(\lambda I-(-\Sigma+K))\psi}_{L^1(G\times S\times I)^3}\geq(\lambda+c) \n{\psi}_{
L^1(G\times S\times I)^3} 
\]  
for all $\lambda>0$. In particular, $-\Sigma+K$ is dissipative.
\end{remark}

Recall that  the dual of $L^1(G\times S\times I)^3$ is
\[
(L^1(G\times S\times I)^3)^*=\bigoplus_{j=1}^3L^1(G\times S\times I)^*=\bigoplus_{j=1}^3 L^\infty (G\times S\times I)=L^\infty (G\times S\times I)^3,
\]
in the sense that for any $l\in (L^1(G\times S\times I)^3)^*$ there exists a unique $w=(w_1,w_2,w_3)\in
L^\infty (G\times S\times I)^3$ such that
\be\label{co1}
l(\psi)=\la w,\psi\ra,
\ee
where 
\[
\la w,\psi\ra=\sum_{j=1}^3\la w_j,\psi_j\ra=\sum_{j=1}^3\int_{G\times S\times I}w_j\psi_j\ dx d\omega dE,
\]
and, in the other direction, any $w\in L^\infty (G\times S\times I)^3$ defines by (\ref{co1}) a linear
form belonging to $(L^1(G\times S\times I)^3)^*$.
The norm in $L^\infty (G\times S\times I)^3$ is
$\n{w}_{L^\infty (G\times S\times I)^3}=\max_{1\leq j\leq 3}\n{w_j}_{L^\infty (G\times S\times I)}$.

The following corollary is a direct consequence
of Theorems \ref{th:dissipative} and \ref{dfsco}.

\begin{corollary}\label{ddd}
Under the assumptions of Theorem \ref{dfsco},
one has
\[
\forall \psi\in L^1(G\times S\times I)^3,
\quad\exists w\in J(\psi)\quad \textrm{s.t.}\quad 
\la w,(\Sigma-K)\psi\ra \geq c\n{\psi}^2_{L^1(G\times S\times I)^3}.
\]
\end{corollary}

In Corollary \ref{ddd} above, the structure of $J(\psi)$ is known by applying Lemma \ref{dualfle}.

%%%%%%%%%%%%%%%%%%%%%%%%%%%%%%%%%%%%%%%%%%%%%%%%%%%%%%%%%%%%%%%%%%%%%%%%
\subsection{On Existence and Uniqueness of Solutions in $L^1$-spaces for Coupled BTE-system}\label{coss3}
%%%%%%%%%%%%%%%%%%%%%%%%%%%%%%%%%%%%%%%%%%%%%%%%%%%%%%%%%%%%%%%%%%%%%%%%

At first
we consider the existence and uniqueness of solutions in the space $L^1(G\times S\times I)^3$ for the problem:
Given $f=(f_1,f_2,f_3)\in L^1(G\times S\times I)^3$, find $\psi=(\psi_1,\psi_2,\psi_3)\in \tilde{W}^1_{-,0}(G\times S\times I)^3$ such that
\bea\label{co3}
\omega\cdot\nabla\psi_j+\Sigma_j\psi -K_j\psi &=f_j(x,\omega,E),
\nonumber\\
{\psi_j}_{|\Gamma_-}&=0,
\eea
for $j=1,2,3$.

Using the notations introduced earlier, the problem (\ref{co3}) is equivalent to
\[
(-{\bf A}_0+\Sigma-K)\psi=f
\]
where $f=(f_1,f_2,f_3)$ and $\psi=(\psi_1,\psi_2,\psi_3)\in D({\bf A}_0)=\tilde{W}^1_{-,0}(G\times S\times I)^3$. The (unique) solvability of the problem (\ref{co3}) is of course equivalent to
the (unique) solvability of the problem
\[
({\bf A}_0-\Sigma+K)\psi=f.
\]

\begin{theorem}\label{coth2}
Suppose that the assumptions (\ref{scateh}), (\ref{colleh}) and (\ref{co2a})
are valid for $c>0$. Then for every $f\in L^1(G\times S\times I)^3$ the problem (\ref{co3})
has a unique solution $\psi\in \tilde{W}_{-,0}^1(G\times S\times I)^3$.
\end{theorem}

\begin{proof}
By Theorem \ref{coth1} the operator ${\bf A}_0:L^1(G\times S\times I)^3\to L^1(G\times S\times I)^3$  is m-dissipative and by Theorem \ref{dfsco} the operator $-(\Sigma-K)+c I:
L^1(G\times S\times I)^3\to L^1(G\times S\times I)^3$ is dissipative. Hence according to Theorem \ref{perth} the sum ${\bf A}_0-(\Sigma-K)+cI :L^1(G\times S\times I)^3\to L^1(G\times
S\times I)$ is $m$-dissipative. This implies, as $c>0$, that $R\big(cI-({\bf A}_0-(\Sigma-K)+cI)\big)=R(-{\bf A}_0+\Sigma-K)=L^1(G\times S\times I)^3$, and so the existence of solutions follows.

Because $c>0$ and since ${\bf A}_0-(\Sigma-K)+cI$ is dissipative,
we have by \eqref{eq:dissipative} of Theorem \ref{th:dissipative} that
\be\label{eq:uniquesol}
\n{(-{\bf A}_0+\Sigma-K)\psi}_{L^1(G\times S\times I)^3}\geq c\n{\psi}_{L^1(G\times S\times I)^3},
\ee
which implies the uniqueness of the solution.
This completes the proof.
\end{proof}

\begin{remark}
We note that the inequality \eqref{eq:uniquesol} implies that for all $f\in L^1(G\times S\times I)^3$
\bea\label{bestim}
\n{(-{\bf A}_0+\Sigma-K)^{-1}f}_{ L^1(G\times S\times I)^3}\leq {1\over c}\n{f}_{ L^1(G\times S\times I)^3},
\eea
or in other words, the solution of the problem (\ref{co3}) satisfies
\be\label{bestim1}
\n{\psi}_{L^1(G\times S\times I)^3}
\leq {1\over c}\n{f}_{L^1(G\times S\times I)^3}.
\ee
\end{remark}

For the consideration of inhomogeneous inflow boundary data we need some detailed information from the trace mapping $\gamma_-:\tilde W^1(G\times S\times I)\to T^1(\Gamma_-)$.

We have (cf. \cite{dautraylionsv6}, p. 252 and \cite{cessenat85}, \cite{choulli}, \cite{cipolatti06})

\begin{lemma}\label{lift1}
Every $g\in T^1(\Gamma_-)$ has an extension,
called the \emph{lift}, $\psi=Lg\in \tilde W^1(G\times S\times I)$ 
such that $\gamma_-(Lg)=(Lg)_{|\Gamma_-}=g$.
In addition, the linear lift operator $L: T^1(\Gamma_-)\to\tilde W^1(G\times S\times I)$ satisfies
\be\label{li0a}
\omega\cdot\nabla(Lg)=0
\ee
and
\bea\label{li0}
\n{Lg}_{L^1(G\times S\times I)}
=\n{Lg}_{W^1(G\times S\times I)}
%=\n{g t_+}_{T^1(\Gamma_-)}
\leq d\n{g}_{T^1(\Gamma_-)}\quad \forall g\in T^1(\Gamma_-),
\eea
where $d$ is the diameter of $G$.
\end{lemma}

\begin{proof}
Let $g\in T^1(\Gamma_-)$. Define $Lg:G\times S\times I\to\R$ by
\begin{align}\label{li1}
(Lg)(x,\omega,E)=g(x-t(x,\omega)\omega,\omega,E).
\end{align}
Using Theorem \ref{lth3a} and limiting techniques,
we have that $\omega\cdot\nabla (Lg)=0$ in $L^1(G\times S\times I)$. In addition, $\gamma_-(Lg)=g$ since $t(y,\omega)=0$ a.e. $(y,\omega,E)\in\Gamma_-$.  We have to show that $Lg\in
L^1(G\times S\times I)$ and that the estimate (\ref{li0}) holds. 

We apply the known change of variables (see e.g. \cite{choulli}, Prop. 2.1.). Define for $(y,\omega)\in\partial G\times S$,
\[
t_+(y,\omega)=\inf\{s>0\ |\ y+s\omega\not\in G\}.
\]
We find that $t_+(y,\omega)\leq d$ for $(y,\omega,E)\in\Gamma_-$ since $\n{\omega}=1$.
Assume for simplicity that $\partial G$ has parametrization
(which is almost global), say $h:V\to\partial G\setminus \Gamma_1 $ where $\Gamma_1$ has zero surface measure. Generally we have a finite number of parametrized patches that cover $\partial G$.
Applying for each fixed $\omega$ the change of variables (in $x$-variable) $x=h(v)+t\omega=:H(v,t)$,
we find that the Jacobian of $J_H$ of $H$ is
\[
J_H(v,t)=\omega\cdot (\partial_1h\times\partial_2h)(v)
=\omega\cdot\nu(h(v))
\n{(\partial_1h\times\partial_2h)(v)},
\]
and that $H(W)=G$, where
$W:=\{(v,t)\ |\ v\in V_-,\ 0<t<t_+(h(v),\omega)\}$ and
$V_-:=\{v\in V\ |\ \omega\cdot\nu(h(v))< 0\}$. Hence we get
\[
\n{Lg}_{ L^1(G\times S\times I)}
=&\int_{G\times S\times I}|(Lg)(x,\omega,E)| dx d\omega dE\\
=&
\int_{S\times I} \int_G |g(x-t(x,\omega)\omega,\omega,E)| dx d\omega dE\nonumber\\
=&
\int_{S\times I}\int_{V_-}\int_0^{t_+(h(v),\omega)}|g(h(v),\omega,E)||J_H(v,t)| dt dv d\omega dE\nonumber\\
=&
\int_{S\times I}\int_{V_-}\int_0^{t_+(h(v),\omega)}|g(h(v),\omega,E)||\omega\cdot\nu(h(v))|
\n{(\partial_1 h\times\partial_2 h)(v)} dt dv d\omega dE
\nonumber\\
=&
\int_{S\times I}\int_{V_-} |g(h(v),\omega,E)|t_+(h(v),\omega) |\omega\cdot\nu(h(v))|
\n{(\partial_1 h\times\partial_2 h)(v)} dv d\omega dE \\
=&\n{gt_+}_{T^1(\Gamma_-)},
\]
where in the third step we used that
$t(h(v)+t\omega,\omega)=t$,
while in the last step we
noticed that $\n{\partial_1h\times\partial_2h(v)}$ is the Jacobian $J_h$ of $h$,
and that $h(V_-)$ differs from $\Gamma_-$ only by a zero-measurable set
(in fact $h(V_-)=\Gamma_-\setminus\{(y,\omega,E)\in \Gamma_1\times S\times I\}$).

Since $t_+(y,\omega)\leq d$ for all $(y,\omega,E)\in\Gamma_-$,
we have furthermore
\[
\n{gt_+}_{T^1(\Gamma_-)}\leq d\n{g}_{T^1(\Gamma_-)}.
\]
This completes the proof.
\end{proof}

As mentioned in section \ref{ls1},
the spaces $T^1(\Gamma_-)$ and $T^1(\Gamma_+)$ can be identified in a natural way.
This is formulated in the following corollary.

\begin{corollary}\label{cor:gamma-pm}
For every $g\in T^1(\Gamma_-)$ we have $(Lg)|_{\Gamma_+}\in T^1(\Gamma_+)$ and
the map
\[
\Theta_-:T^1(\Gamma_-)\to T^1(\Gamma_+);\quad g\mapsto (Lg)_{|\Gamma_+}
\]
is an isometric isomorphism. In particular,
\[
\n{(Lg)_{|\Gamma_+}}_{T^1(\Gamma_+)}=\n{g}_{T^1(\Gamma_-)},\quad \forall g\in T^1(\Gamma_-).
\]
\end{corollary}

\begin{proof}
If applies \eqref{eq:stokes} with $u=L|g|=|Lg|$,
recalls from Lemma \ref{lift1} that $\omega\cdot\nabla_x (L|g|)=0$, for $g\in T^1(\Gamma_-)$,
and takes into account that the part $\Gamma_0$ of $\Gamma=\Gamma_0\cup\Gamma_-\cup\Gamma_+$ is zero-measurable,
one obtains
\[
0=&\int_{G\times S\times I}\omega\cdot \nabla_x(L|g|)\diff x\diff \omega\diff E
=\int_{\Gamma} L|g| (\omega\cdot \nu)\diff\sigma\diff\omega\diff E \\
=&-\int_{\Gamma_-} |g| |\omega\cdot \nu|\diff\sigma\diff\omega\diff E
+\int_{\Gamma_+} |(Lg)_{|\Gamma_+}| |\omega\cdot \nu|\diff\sigma\diff\omega\diff E \\
=&-\n{g}_{T^1(\Gamma_-)}+\n{(Lg)_{|\Gamma_+}}_{T^1(\Gamma_+)}.
\]
This shows that the map $\Theta_-$ is isometric.
By constructing the obvious inverse map of $\Theta_-$ shows
that $\Theta_-$ is surjective as well.
\end{proof}

\begin{remark}
Combining Lemma \ref{lift1} and Corollary \ref{cor:gamma-pm},
we have the following bound for the lift operator $L$
into the space $\tilde{W}^1(G\times S\times I)$:
\[
\n{Lg}_{\tilde W^1(G\times S\times I)}\leq (d+2)\n{g}_{T^1(\Gamma_-)}.
\]
\end{remark}

As an immediate corollary of the Lemma \ref{lift1}, we have:

\begin{lemma}\label{lift2}
Let $T>0$ and $k\in \N_0$. Then
for every $g\in C^k([0,T],T^1(\Gamma_-))$ there exists a lift $\psi=Lg\in C^k([0,T],\tilde W^1(G\times S\times I))$ 
such that $\gamma_-(Lg)=(Lg)_{|\Gamma_-}=g$. 
Moreover, 
\[
\omega\cdot\nabla(Lg)=0
\]
and
\[
\n{Lg}_{C^k([0,T],L^1(G\times S\times I))}
=\n{Lg}_{C^k([0,T],W^1(G\times S\times I))}
\leq d\n{g}_{C^k([0,T],T^1(\Gamma_-))}.
\]
\end{lemma}

\begin{proof}
Defining the lift $Lg$ by
\begin{align}\label{li1a}
(Lg)(x,\omega,E,t)=g(x-t(x,\omega)\omega,\omega,E,t),
\end{align}
we have $(Lg)(x,\omega,E,t)=L(g(t))(x,\omega,E)$,
with the latter $L$ the lift as defined in Lemma \ref{lift1}.
As $L$ of Lemma \ref{lift1} is linear and bounded, it follows from $g\in C^k([0,T],T^1(\Gamma_-))$
that $Lg\in C^k([0,T],\tilde{W}^1(G\times S\times I))$,
and for $j=1,\dots,k$,
\[
\n{\partial_t^i (Lg)}_{C([0,T],L^1(G\times S\times I))}
=\n{L(\partial_t^i g)}_{C([0,T],L^1(G\times S\times I))}
\leq d\n{\partial_t^i g}_{C([0,T],T^1(\Gamma_-))},
\]
from which the desired estimate.
\end{proof}

\begin{example}
If $G=B(0,r)\subset\R^3$ the lift $L$ of Lemma \ref{lift2}
can be seen, due to Example \ref{le1}, to be given by
\[
(Lg)(x,\omega,E,t)=g\left(x-\Big(x\cdot\omega+\sqrt{(x\cdot\omega)^2+r^2-\n{x}^2}\:\Big)\omega,\omega,E,t\right),
\]
for  $g\in C^k([0,T],T^1(\Gamma_-))$.
\end{example}

\begin{remark}
Using Lemma \ref{lift1} one can show that for any $1\leq p<\infty$
and every $g\in T^p(\Gamma_-)$ has an extension $\psi=Lg\in \tilde W^p(G\times S\times I)$ 
such that $\gamma_-(Lg)=(Lg)_{|\Gamma_-}=g$.
In addition, the linear lift operator $L: T^p(\Gamma_-)\to\tilde W^p(G\times S\times I)$ satisfies
\be\label{li0aa}
\omega\cdot\nabla(Lg)=0
\ee
and
\bea\label{li0aaa}
\n{Lg}_{L^p(G\times S\times I)}
=\n{Lg}_{W^p(G\times S\times I)}\leq d\n{g}_{T^p(\Gamma_-)}\ {\rm for\ all}\ g\in T^p(\Gamma_-).
\eea

Indeed, the proof of \eqref{li0aa} for any $1\leq p<\infty$
proceeds precisely as in the case $p=1$ (see the beginning of the proof of Lemma \ref{lift1}).
On the other hand, if $g\in T^p(\Gamma_-)$ then $g^p\in T^1(\Gamma_-)$
and as $L(g^p)=(Lg)^p$, with $Lg$ defined pointwise (a.e.) by \eqref{li1a},
and hence \eqref{li0a} immediately implies \eqref{li0aaa}.
respectively.

Similarly, Lemma \ref{lift2} admits a generalization to any $1\leq p<\infty$.
\end{remark}

For inhomogeneous inflow boundary data we get
\begin{theorem}\label{coth3}
Suppose that the assumptions (\ref{scateh}), (\ref{colleh}) and (\ref{co2a}) and  are valid with $c>0$.
Then for every $f\in L^1(G\times S\times I)^3$ and $g\in T^1(\Gamma_-)^3$ the problem
\bea\label{co3aa}
\omega\cdot\nabla\psi_j+\Sigma_j\psi_j -K_j\psi&=f_j(x,\omega,E)\ \nonumber\\
{\psi_j}_{|\Gamma_-}&=g_j,
\eea
where $j=1,2,3$, has a unique solution $\psi\in \tilde W^1(G\times S\times I)^3$.
\end{theorem}

\begin{proof}
 As usual we apply the lift of inflow boundary data. By Lemma \ref{lift1} there exists $\tilde{\psi}_j:=Lg_j
\in\tilde W^1(G\times S\times I)$ such that ${\tilde{\psi_j}}_{|\Gamma_-}=g_j$. Let $\tilde\psi=(\tilde\psi_1,\tilde\psi_2,\tilde\psi_3)$ and substitute in the problem (\ref{co3aa}) $u=\psi-\tilde\psi$ for $\psi$.
Then we get
\bea\label{co4}
& \omega\cdot\nabla u_j+\Sigma_j u_j -K_j u=f_j-\omega\cdot\nabla \tilde\psi_j-\Sigma_j\tilde\psi_j +K_j\tilde\psi=:\tilde{f_j}(x,\omega,E)\nonumber\\
& u_j{|\Gamma_-}={\psi_j}_{|\Gamma_-}-{\tilde{\psi_j}}_{|\Gamma_-}=g_j-g_j=0,
\eea
for $j=1,2,3$.
Since $\tilde f:=(\tilde f_1, \tilde f_2,\tilde f_3)\in L^1(G\times S\times I)^3$ we get by Theorem \ref{coth2} that the problem (\ref{co4}) has a unique solution $u\in \tilde W^1_{-,0}(G\times S\times I)^3$.
Then $\psi:=u+\tilde\psi\in \tilde W^1(G\times S\times I)^3$ is the required unique solution of (\ref{co3aa}) and so we obtain the assertion.
\end{proof}

\begin{corollary}\label{cdd}
Under the assumptions of Theorem \ref{coth3} the solution $\psi$ of the problem (\ref{co3aa}) satisfies, with some constants $C_1, \ C_2,\ C_3>0$, the estimates
\be\label{ess1}
\n{\psi}_{L^1(G\times S\times I)^3}\leq {1\over c}\n{f}_{L^1(G\times S\times I)^3}
+C_1\n{g}_{T^1(\Gamma_-)^3},
\ee
\be\label{ess2}
\n{\psi}_{W^1(G\times S\times I)^3}\leq C_2\Big(\n{f}_{L^1(G\times S\times I)^3}
+\n{g}_{T^1(\Gamma_-)^3}\Big)
\ee
and
\be\label{ess2a}
\n{\psi}_{\tilde W^1(G\times S\times I)^3}\leq C_3\Big(\n{f}_{L^1(G\times S\times I)^3}
+\n{g}_{T^1(\Gamma_-)^3}\Big)
\ee
\end{corollary}

\begin{proof}
By the proof of Theorem \ref{coth3} $\psi=u+Lg$ where $u\in \tilde W^1_{-,0}(G\times S\times I)^3$ satisfies
\[
(-{\bf A}_0+\Sigma-K)u=f-(-{\bf A}+\Sigma-K)(Lg).
\]
In addition, ${\bf A}(Lg)=0$ by Lemma \ref{lift1}. Furthermore, by (\ref{bestim1}) and by Lemma \ref{lift1} 
\[
&
\n{\psi}_{L^1(G\times S\times I)^3}=\n{u+Lg}_{L^1(G\times S\times I)^3} \\
\leq & {1\over c}\n{f-(\Sigma-K)(Lg)}_{L^1(G\times S\times I)^3}
+\n{Lg}_{L^1(G\times S\times I)^3} \\
\leq & {1\over c}\Big(\n{f}_{L^1(G\times S\times I)^3}+\n{\Sigma-K}d\n{g}_{T^1(\Gamma_-)^3}\Big)+d\n{g}_{T^1(\Gamma_-)^3}
\]
which implies (\ref{ess1}). 

The assertion (\ref{ess2}) follows from estimate (\ref{ess1}) since $\omega\cdot\nabla\psi=f-(\Sigma-K)\psi$ when $\psi$ is the solution of (\ref{co3aa}).
Finally, the last estimate (\ref{ess2a}) follows from Theorem \ref{th:trace:2},
which tells us that
\[
\n{\psi}_{T^1(\Gamma_+)}\leq \n{\psi}_{W^1(G\times S\times I)^3}+\n{\psi}_{T^1(\Gamma_-)},
\]
and from (\ref{ess2}). This completes to proof.
\end{proof}

The result of Corollary \ref{cdd} means that the solution $\psi$ depends continuously on the data $f,\ g$.

Finally we will consider the non-negativity of solutions.
Since ${\bf A}_0$ is $m$-dissipative it generates a contraction $C^0$-semigroup $T(t),\ t\geq 0$
(Lumer-Phillips Theorem, see e.g. \cite{dautraylionsv5}, p. 343, \cite{engelnagel}, pp. , \cite{pazy83}, pp. 14-15, \cite{goldstein}). For $f\in D({\bf A}_0)=\tilde W^1_{-,0}(G\times S\times I)^3$, the curve $\psi(t)=T(t)f$, $t>0$, is the unique solution of the problem (\cite{dautraylionsv5}, pp. 397-405, \cite{pazy83}, p. 100, \cite{goldstein})
\be\label{cp}
{\p \psi{t}}-{\bf A}_0\psi=0,\quad \psi(0)=f
\ee
where $\psi\in C^1([0,\infty[,L^1(G\times S\times I)^3)\cap C([0,\infty[,\tilde W^1_{-,0}(G\times S\times I)^3)$.

Denote
$\psi(x,\omega,E,t):=\psi(t)(x,\omega,E)$.
The problem (\ref{cp}) can be solved (as above in Section \ref{lss1}) by the Lagrange's method in the classical sense which we describe shortly in the sequel assuming that $f$ is sufficiently smooth, say $f\in C(\ol G\times S\times I)\cap D({\bf A}_0)$.
The equation (\ref{cp}) is uncoupled and for each $j$ it is of the form
\be\label{cpa}
{\p {\psi_j}{t}}+\sum_{k=1}^3\omega_k{\p {\psi_j}{x_k}}=0
\ee
and $\psi_j$ must satisfy an initial-boundary condition of the form
\be\label{cpib}
\psi_j(x,\omega,E,0)=&f_j(x,\omega,E),\quad && (x,\omega,E)\in G\times S\times I, &&\\
\psi_j(y,\omega,E,t)=&0,\quad && (y,\omega,E,t)\in \Gamma_-\times [0,\infty[. &&\nonumber
\ee

We solve the problem \eqref{cpa}--\eqref{cpib} for a fixed $j$ and we denote for simplicity  $\psi:=\psi_j$ and $f:=f_j$.
Furthermore, denote $(x,\omega,E,t)=(x_1,x_2,x_3,\omega_1,\omega_2,\omega_3,E,t)$. Then  the augmented system of ordinary differential equations (system of characteristics) is
\bea\label{aug}
T'(t)&= 1\nonumber\\
X_1'(s)&=\Omega_1,\ \ \ \Omega_1'(s)=0,\nonumber\\
X_2'(s)&=\Omega_2,\ \ \ \Omega_2'(s)=0,\nonumber\\
X_3'(s)&=\Omega_3,\ \ \ \Omega_3'(s)=0,\\
{\s E}'(s)&=0 \nonumber\\
\Psi'(s)&= 0\nonumber
\eea
We denote $X=(X_1,X_2,X_3),\ \Omega=(\Omega_1,\Omega_2,\Omega_3)$. We find that
\be \label{asol}
T(s)=s+C,\ \Omega(s)=C',\ X(s)=sC'+C'',\ {\s E}(s)=C''',\ \Psi(s)=C''''
\ee
where $C,\ C',\ C'',\ C''',\ C''''$ are constants.

Taking into account the condition (\ref{cpib}) we see that the solution of the augmented system must satisfy the initial condition of the form
\bea\label{augin}
(X(0),\Omega(0),{\s E}(0),T(0),\Psi(0)) &=
(h(v),\omega,E,t',0),\ t'>0\\
(X(0),\Omega(0),{\s E}(0),T(0),\Psi(0)) &=
 (x',\omega,E,t',f(x',\omega,E)),\ t'=0\nonumber
\eea
where $h=h(v)$ is as in Section \ref{lss1} the local parametrization of $\partial G$. Here $x$ (resp. $t$) is replaced by $x'$ (resp. $t'$) for notational reasons. Matching the initial condition (\ref{augin}) to the solution
(\ref{asol}) we get the solution $(X(s),\Omega(s),{\s E}(s),T(s),\Psi(s))$. By eliminating
$x',\ v,\ t',\ E,\ \omega$ from the system
\[
(X(s),\Omega(s),{\s E}(s),T(s),\Psi(s))&=(x,\omega,E,t,0)\ {\rm for}\ t'>0\\
(X(s),\Omega(s),{\s E}(s),T(s),\Psi(s))&=
 (x,\omega,E,t,f(x,\omega,E))\
{\rm for}\ t'=0
\]
and noting that (formally)
\[
\Psi(s)=H(-t')f(x',\omega,E)
\]
we get the solution $\psi$ as in Section \ref{lss1}. The result is
\be\label{sgAa}
\psi(x,\omega,E,t)
=f(x-t\omega,\omega,E)H(t(x,\omega)-t),\ f\in \tilde W_{-,0}^1(G\times S\times I)
\cap C(\ol G\times S\times I)
\ee
where $H$ is the Heaviside function.
Applying the limiting techniques (cf. the proof of Theorem \ref{dth1})
we get
\be\label{sgA1}
(T(t)f)(x,\omega,E)
=\psi(t)(x,\omega,E)=H(t(x,\omega)-t)f(x-t\omega,\omega,E),\ f\in \tilde W_{-,0}^1(G\times S\times I).
\ee
Since $\tilde W^1_{-,0}(G\times S\times I)$ is dense in
$L^1(G\times S\times I)$ the formula (\ref{sgA1}) is valid for any
$f\in L^1(G\times S\times I)$.

For the three particles the semigroup $T(t)$ is given by
\be\label{sgA}
&(T(t)f)(x,\omega,E)\nonumber\\
=&
H(t(x,\omega)-t)
\big(f_1(x-t\omega,\omega,E),
f_2(x-t\omega,\omega,E),
f_3(x-t\omega,\omega,E)\big),
\ee
where $f\in  L^1(G\times S\times I)^3$.
In literature (see e.g. \cite{dautraylionsv6}, pp. 222-224 ),
the formula (\ref{sgA}) is demonstrated for one particle system
by using different methods.

We have the following result on non-negativity of solutions.

\begin{theorem}\label{coth4}
Suppose that the assumptions of Theorem \ref{coth3} are valid and that moreover
\bea\label{snn}
&f_j(x,\omega,E)\geq 0\quad {\rm a.e.}\ (x,\omega,E)\in G\times S\times I\\
&g_j(y,\omega,E)\geq 0\quad {\rm a.e.}\ (y,\omega,E)\in \Gamma_-,\ {\rm for}\ j=1,2,3.\nonumber
\eea
Then the solution given in Theorem \ref{coth3} satisfies $\psi(x,\omega,E)\geq 0\ {\rm a.e.}\ (x,\omega,E)\in G\times S\times I$.
\end{theorem}

\begin{proof}
A. We put the problem (\ref{co3aa}) in the abstract form
\bea\label{co5}
-{\bf A}\psi+\Sigma\psi -K\psi&=f,\\
 \psi_{|\Gamma_-}&=g.\nonumber
\eea
Assume at first that $g=0$. Then the problem (\ref{co5}) is
$(-{\bf A}_0+\Sigma -K)\psi=f$. Let, as above, $T(t)$ be the $C^0$-semigroup generated by ${\bf A}_0$. Then
by (\ref{sgA})
\[
T(t)f\geq 0\ {\rm for}\ f\geq 0
\]
In addition, we immediately see that
\[
K\psi\geq 0\ {\rm for}\ \psi\geq 0,
\]
and that
\[
(\Sigma_j\psi)(x,\omega,E)=\Sigma_j(x,\omega,E)\psi_j(x,\omega,E)\geq 0\ {\rm for}\ \psi\geq 0,
\]
a.e. $(x,\omega,E)\in G\times S\times I$, as $\Sigma_j(x,\omega,E)\geq 0$ by assumption.
These imply that if $T_K$ and $T_{-\Sigma}$ are the semigroups generated by
the bounded operators $K$ and $-\Sigma$,
i.e. $T_K(t) \psi=\sum_{i=0}^\infty \frac{1}{i!}t^i K^i\psi$
and $T_{-\Sigma}(t)\psi=(e^{-t\Sigma_1}\psi_1,e^{-t\Sigma_2}\psi_2, e^{-t\Sigma_3}\psi_3)$,
we have $T_K\psi\geq 0$, $T_{-\Sigma}\psi\geq 0$ whenever $\psi\geq 0$.

Since by the proof of Theorem \ref{coth2} ${\bf A}_0-(\Sigma-K)+cI$ is $m$-dissipative, the operator ${\bf A}_0-\Sigma+K$ generates a contraction $C^0$-semigroup $G(t)$ for which in addition 
\[
\n{G(t)}\leq e^{-c't},\quad \forall t\geq 0,
\]
where $c'$ is a positive number which is less than or equal to $c>0$.
Note here that the $m$-dissipative operator ${\bf A}_0-(\Sigma-K)+c'I$ generates the semigroup $e^{c't}G(t)$. This is a consequence of the Lumer-Phillips Theorem
(\cite{dautraylionsv5}, p. 343, \cite{engelnagel}, Theorem II.3.15, p. 83 and \cite{goldstein}).
From Hille-Yosida Theorem (\cite{dautraylionsv5}, p. 321 and \cite{engelnagel}, Theorem II.3.5, p. 73) and from the resolvent formula (cf. \cite{engelnagel}, Theorem II.1.10, p. 55)
we obtain that
\be\label{co6}
\psi
=&(-{\bf A}_0+\Sigma-K)^{-1}f
=\big(c'I-({\bf A}_0-(\Sigma-K)+c'I)\big)^{-1} \nonumber \\
=&\int_0^\infty e^{-c't}\big(e^{c't}G(t)\big)f dt
=\int_0^\infty G(t)f dt.
\ee
By the Trotter's formula (\cite{engelnagel}, Theorem III.5.2, p. 220, or \cite{goldstein}, p.53, where the proof is given only for contraction semigroups) 
\[
G(t)f=\lim_{n\to\infty} \big(T(t/n)T_{-\Sigma}(t/n)T_K(t/n)\big)^nf
\]
which implies that $G(t)f\geq 0$ for $f\geq 0$ (cf. Section XXI-\S 2, Proposition 2, pp. 226-227 of \cite{dautraylionsv6}). Hence $\psi\geq 0$
and then the proof is complete in this special case. 

B. Suppose that more generally $g\in T^1(\Gamma_-)^3$ is such that $g\geq 0$. By Theorem \ref{coth3} the solution $u\in \tilde W^1(G\times S\times I)^3$ of the problem
\be\label{co7}
-{\bf A}u+\Sigma u=0,\quad u_{|\Gamma_-}=g,
\ee
exists.
We show that it is non-negative. Indeed, applying again limiting techniques we find that by (\ref{l21c}) the (distributional) solution is 
\begin{align}\label{co7a}
u(x,\omega,E)&=\Big(
e^{\int_0^{t(x,\omega)}\Sigma_1(x-s\omega,\omega,E)ds}\cdot g_1(x-t(x,\omega)\omega,\omega,E),\\
& e^{\int_0^{t(x,\omega)}\Sigma_2(x-s\omega,\omega,E)ds}\cdot g_2(x-t(x,\omega)\omega,\omega,E),\nonumber\\
& e^{\int_0^{t(x,\omega)}\Sigma_3(x-s\omega,\omega,E)ds}\cdot g_3(x-t(x,\omega)\omega,\omega,E)\Big),\nonumber
\end{align}
from which one immediately sees that $u\geq 0$ once $g\geq 0$.

Finally, let $w:=\psi-u$. Then we find that
\begin{align}\label{co8}
-{\bf A}w+\Sigma w-Kw &=f+({\bf A}u-\Sigma u+Ku)=f+Ku\geq 0 \\
 w_{|\Gamma_-} &=g-g=0.\nonumber
\end{align}
Hence by Part A. of the proof, we have $w\geq 0$ and therefore $\psi=w+u\geq 0$.
This completes the proof.
\end{proof}

\begin{remark}\label{decomp}

Consider the transport problem (\ref{co3aa}).
The solution $\psi\in \tilde W^1(G\times S\times I)^3$ can be decomposed as follows. Let $u\in \tilde W^1(G\times S\times I)^3$ be the solution of the problem 
\be\label{re1}
\omega\cdot\nabla u_j+\Sigma_j u_j=f_j,\quad {\rm on}\ G\times S\times  I,\ j=1,2,3,
\ee
together with the inflow boundary condition
\be\label{re2}
u_{|\Gamma_-}=g.
\ee
Furthermore, let $w\in \tilde W^1(G\times S\times I)^3$ be the solution of the problem
\be\label{re3}
\omega\cdot\nabla w_j+\Sigma_j w_j-K_jw=K_j u,\quad {\rm on}\ G\times S\times  I, \ j=1,2,3,
\ee
with the homogeneous inflow boundary values
\be\label{re4}
w_{|\Gamma_-}=0.
\ee 
 
Then we find that $\psi=u+w\in \tilde W^1(G\times S\times I)^3$ is the solution of (\ref{co3aa}).
This decomposition is corresponding to the evolution of primary particles ($u$) and of secondary particles ($w$) of the overall particle transport.
The decomposition $\psi=u+w$ may be useful e.g. in constructing numerical solutions. 
Note that {\it the system (\ref{re1})-(\ref{re2}) is uncoupled}.
By (\ref{l24}) we formally have an explicit solution for (\ref{re1})-(\ref{re2}) 
\bea\label{re6}
u_j(x,\omega,E)
=&\int_0^{t(x,\omega)}e^{\int_0^{t}-\Sigma_j(x-s\omega,\omega,E) ds}f_j(x-t\omega,\omega,E)dt
\nonumber\\
&+
e^{\int_0^{t(x,\omega)}-\Sigma_j(x-s\omega,\omega,E) ds}g_j(x-t(x,\omega)\omega,\omega,E).
\eea
\end{remark}

\begin{remark}\label{re:calc:1}
By the proof of Theorem \ref{coth3} the solution $\psi$ of the transport problem
\[
(-{\bf A}+\Sigma-K)\psi=f,\quad \psi_{|\Gamma_-}=g
\]
is the sum (recall that ${\bf A}(Lg)=0$)
\bea\label{calc1}
\psi=&u+Lg=(-{\bf A}_0+\Sigma-K)^{-1}(f-(-{\bf A}+\Sigma-K)(Lg))+Lg
\nonumber\\
=&(-{\bf A}_0+\Sigma-K)^{-1}(f-(\Sigma-K)(Lg))+Lg.
\eea

Since $Lg$ is known, the essential part from the computational point of view is to find $u$ that is the solution of the equation
\[
(-{\bf A}_0+\Sigma-K)u=f-(\Sigma-K)(Lg)=:\tilde f.
\]
We see, on the other hand, that this equation is equivalent to
\[
(-{\bf A}_0+\Sigma)u=Ku+\tilde f
\]
or to (notice that $(-{\bf A}_0+\Sigma)^{-1}$ exists)
\[
u=(-{\bf A}_0+\Sigma)^{-1}Ku+(-{\bf A}_0+\Sigma)^{-1}\tilde f.
\]
This can, furthermore, be written into the form
\be\label{calc2}
(I-T)u=\tilde{\tilde f}
\ee
where $T:=(-{\bf A}_0+\Sigma)^{-1}K$
is a bounded linear operator from $L^1(G\times S\times I)^3$ into itself,
and $\tilde{\tilde f}:=(-{\bf A}_0+\Sigma)^{-1}\tilde f$.

If it happened that $\n{T}<1$,
the solution $u$ would be obtained from the Neumann series
\be\label{ssol}
u=\sum_{k=0}^\infty T^k\tilde{\tilde f}
=
\sum_{k=0}^\infty ((-{\bf A}_0+\Sigma)^{-1}K)^k
((-{\bf A}_0+\Sigma)^{-1}(f-(\Sigma-K)(Lg))),
\ee
where by (\ref{l21cc}) the $j$-th component, $j=1,2,3$, of $(-{\bf A}_0+\Sigma)^{-1}h$ is for any $h\in L^1(G\times S\times I)$ given by
(in generalized sense; see (\ref{resolA}))
\begin{align}\label{nb0}
((-{\bf A}_0+\Sigma)^{-1}h)_j
=\int_0^{t(x,\omega)}e^{\int_0^t-\Sigma_j(x-s\omega,\omega,E)ds}h_j(x-t\omega,\omega,E)dt.
\end{align}

From the computational point of view,
this approach, or rather a discretized version of it, has the advantage that no explicit inversions of matrices are needed.
The condition $\n{T}<1$ is, however, restrictive. For $p=\infty$ a sufficient condition for having $\n{T}<1$ is that for $j=1,2,3$ (we omit all details here)
\[
\Sigma_j(x,\omega,E)\geq c>0
\]
and for some $0<\beta<1$
and a.e. on $G\times S\times I$,
\[ 
\beta \Sigma_j(x,\omega,E)\geq \sum_{k=1}^3\int_S\int_I\sigma_{jk}(x,\omega,\omega',E,E')d\omega' dE',
\]
which is stronger a condition to satisfy than (\ref{co2a}).
In addition, the data ($f,\ g$) must be in the corresponding $L^\infty$-spaces (cf. \cite{dautraylionsv6}, pp. 243-244, in the case of one species of particles).
We refer also to \cite{choulli}, Prop. 2.3, where a sufficient condition
to have $\n{T}<1$
is given in the case where $p=1$ and one species of particles is considered. 
\end{remark}

\begin{remark}\label{re:calc:2}
Another method to compute approximately the solution of the problem
\be\label{calceq1}
(-{\bf A}_0+\Sigma-K)u=f-(\Sigma-K)(Lg)=:\tilde f,
\ee
which avoids the explicit inversions of matrices, can be described as follows.
By the Trotter's formula,
the semigroup $G(t)$ generated by ${\bf A}_0-\Sigma+K$ is given by
(see the proof of Theorem \ref{coth4} above)
\be\label{calceq2}
G(t)f=\lim_{n\to\infty}(T(t/n)T_{-\Sigma}(t/n)T_K(t/n))^nf,\quad (t\geq 0)
\ee
and the limit is uniform on compact intervals $[0,T]$. We know that
\bea\label{calceq3}
&
T(t)f=H(t(x,\omega)-t)(f_1(x-t\omega,\omega,E),f_2(x-t\omega,\omega,E),f_3(x-t\omega,\omega,E))\nonumber\\
&
T_{-\Sigma}(t)f=e^{-t\Sigma(x,\omega,E)}f=(e^{-t\Sigma_1(x,\omega,E)}f_1,e^{-t\Sigma_2(x,\omega,E)}f_2,e^{-t\Sigma_3(x,\omega,E)}f_3)\nonumber\\
&
T_K(t)f=e^{tK}f=\sum_{k=0}^\infty {1\over{k!}}(tK)^kf\approx \sum_{k=0}^{N_0} {1\over{k!}}(tK)^kf .
\eea
In virtue of formula (\ref{co6})
\bea\label{calceq4}
\psi=&\int_0^\infty G(t)\tilde f dt\approx \int_0^T G(t)\tilde f dt
=\int_0^T\lim_{n\to\infty}(T(t/n)T_{-\Sigma}(t/n)T_K(t/n))^n\tilde f dt
\nonumber\\
\approx &
\int_0^T[T(t/n_0)T_{-\Sigma}(t/n_0)T_K(t/n_0)]^{n_0}\tilde f dt\nonumber\\
=&\int_0^T\Big[
T(t/n_0)
e^{-(t/n_0)\Sigma(x,\omega,E)}
\sum_{k=0}^{N_0} {1\over{k!}}((t/n_0)K)^k\Big]^{n_0}\tilde f dt
\eea
where $T$, $n_0$ and $N_0$ are large enough. Note that the result in (\ref{calceq4}) can be immediately computed since $T(t)$ is explicitly known.

This approach, unlike the one given in Remark \ref{re:calc:1},
does not require extra assumptions on cross-sections.
\end{remark}

%%%%%%%%%%%%%%%%%%%%%%%%%%%%%%%%%%%%%%%%%%%%%%%%%%%%%%%%%%%%%%%%%%%%%%%%
\section{On Time-dependent Solutions for the Coupled System}\label{tds1}
%%%%%%%%%%%%%%%%%%%%%%%%%%%%%%%%%%%%%%%%%%%%%%%%%%%%%%%%%%%%%%%%%%%%%%%%

In this section we do not need  the assumption (\ref{co2a})  since for time-dependent equations only the $C^0$-semigroup property is essential. The contraction property of semigroup is not needed.

We have $E=\frac{1}{2}m_j\n{v_j}^2$ where $m_j$ (resp. $\n{v_j}$) is the mass (resp. the speed) of the particle $j$. Hence $\n{v_j}=\sqrt{\frac{2E}{m_j}}$.
In the following we consider the problem (for $\n{v_j}\not=0$),
\[
& {1\over{\n{v_j}}}{\p {\psi_j}{t}} +\omega\cdot\nabla\psi_j+\Sigma_j\psi -K_j\psi=f_j(x,\omega,E,t),\quad
&& (x,\omega,E)\in G\times S\times I,\ t\in ]0,T] &\nonumber\\
& \psi(y,t)=g_j(y,t),\ && y\in\Gamma_-,\ t\in ]0,T] &\\
& \psi_j(x,\omega,E,0)=\psi_0(x,\omega,E),\ && (x,\omega,E)\in G\times S\times I & \nonumber
\]
where $j=1,2,3$, $T>0$ and $\psi_j=\psi_j(t)(x,\omega,E)=\psi_j(x,\omega,E,t)$ (we use this agreement without further mention).
Multiplying the above transport equation by $\n{v_j}$ we obtain the equation
\bea
& {\p {\psi_j}{t}} +v_j\cdot\nabla\psi_j+\tilde\Sigma_j\psi -\tilde K_j\psi=\tilde f_j,\quad && (x,\omega,E,t)\in G\times S\times I\times ]0,T[\label{td3a} &\\
& \psi(y,t)=g(y,t),&& y\in\Gamma_-,\ t\in ]0,T] &\\
& \psi(x,\omega,E,0)=\psi_0(x,\omega,E),&& (x,\omega,E)\in G\times S\times I &
\label{td3aa}
\eea
where $v_j$ is the velocity $\n{v_j}\omega=\sqrt{{2E}\over{m_j}}\omega$ of
the $j$-th particle, and where
$\tilde\Sigma_j=\sqrt{{2E}\over{m_j}}\Sigma_j$, $\tilde K$ is the collision operator corresponding to the cross-sections
$\tilde\sigma_{kj}=\sqrt{{2E}\over{m_j}}\sigma_{kj}$ and $\tilde f_j=\sqrt{{2E}\over{m_j}}f_j$.

We modify slightly the function spaces given in section \ref{ls1}.
Let
\[
{\s W}^1(G\times S\times I)=\{\psi\in L^1(G\times S\times I)\ |\ \sqrt{E}\omega\cdot\nabla\psi\in L^1(G\times S\times I)\}.
\]
Then by standard arguments ${\s W}^1(G\times S\times I)$
equipped with the norm
\[
\n{\psi}_{{\s W}^1(G\times S\times I)}=\n{\psi}_{L^1(G\times S\times I)}+
\n{\sqrt{E}\omega\cdot\nabla\psi}_{L^1(G\times S\times I)}.
\]
is a Banach space and $\mc{D}(\ol G\times S\times I)$ is dense subspace of it.

Furthermore we define
\[
{\s T}^1(\Gamma_-)=L^1(\Gamma_-,\sqrt{E}|\omega\cdot\nu|\ d\sigma d\omega dE)
\]
with the norm
\[
\n{h}_{{\s T}^1(\Gamma_-)}=\int_{\Gamma_-}|h(y,\omega,E)|\ \sqrt{E} |\omega\cdot\nu|\ d\sigma d\omega dE.
\]
The space ${\s T}^1(\Gamma_+)$ is defined similarly.
Again any element $\psi\in {\s W}^1(G\times S\times I)$ has well defined trace $\psi_{|\Gamma_-}$ in $L^1_{\rm loc}(\Gamma_-,\sqrt{E}|\omega\cdot\nu|\ d\sigma d\omega dE)$ 
and the trace mapping $\gamma_-:{\s W}^1(G\times S\times I)\to \ L^1_{\rm loc}(\Gamma_-,
\sqrt{E}|\omega\cdot\nu|\ d\sigma d\omega dE);\
\gamma_-(\psi)=\psi_{|\Gamma_-}$ is  continuous.
Similarly for the trace $\gamma_+$ on the outflow boundary $\Gamma_+$, and we can define the trace $\gamma(\psi)=\psi_{|\Gamma}$ on the whole $\Gamma$ as in section \ref{ls1}.

We denote by ${\s T}^1(\Gamma)$ the space of $L^1$-functions with respect to the measure
$\sqrt{E}|\omega\cdot\nu|\ d\sigma d\omega dE$, and equip it
with the norm 
\[
\n{h}_{{\s T}^1(\Gamma)}= \int_{\Gamma}|h(y,\omega,E)|\ \sqrt{E}|\omega\cdot\nu|\ d\sigma d\omega dE.
\]
Finally we define
\[
\tilde {\s W}^1(G\times S\times I)=\{\psi\in {\s W}^1(G\times S\times I)\ |\  \gamma(\psi)\in {\s T}^1(\Gamma) \},
\]
which is again a Banach space with respect to the norm
\[
\n{\psi}_{\tilde {\s W}^1(G\times S\times I)}=\n{\psi}_{{\s W}^1(G\times S\times I)}+ \n{\gamma(\psi)}_{{\s T}^1(\Gamma)},
\]
and denote its subspace of elements of zero trace on the inflow boundary $\Gamma_-$ by
\[
\tilde {\s W}^1_{-,0}(G\times S\times I)=\{\psi\in {\s W}^1(G\times S\times I)\ |\  \gamma_-(\psi)=0 \}.
\]

Define closed operators $\tilde{\bf A}, \tilde{\bf A}_0:L^1(G\times S\times I)^3\to L^1(G\times S\times I)^3$ by
\[
\tilde {\bf A}\psi
&:=(-v_1\cdot\nabla\psi_1,-v_2\cdot\nabla\psi_2,-v_3\cdot\nabla\psi_3),
\\
&=\left(-\sqrt{{{2E}\over{m_1}}}\omega\cdot\nabla\psi_1,
-\sqrt{{{2E}\over{m_2}}}\omega\cdot\nabla\psi_2,
-\sqrt{{{2E}\over{m_3}}}\omega\cdot\nabla\psi_3\right),
\quad \psi\in D(\tilde{\bf A}):={\s W}^1(G\times S\times I)^3
\]
and 
\[
\tilde{\bf A}_0\psi:=\tilde{\bf A}\psi,\quad \psi\in D(\tilde{\bf A}_0):=\tilde {\s W}^1_{-,0}(G\times S\times I)^3.
\] 
In addition, let $\tilde\Sigma\psi=(\tilde\Sigma_1\psi,\tilde\Sigma_2\psi,\tilde\Sigma_3\psi)$ and $\tilde K\psi=(\tilde K_1\psi,\tilde K_2\psi,\tilde K_3\psi)$. 

Assuming that (\ref{scateh}) and (\ref{colleh}) hold we see similarly as in section \ref{cos1} that the operators $\tilde\Sigma$ and $\tilde K$
are bounded operators $L^1(G\times S\times I)^3\to L^1(G\times S\times I)^3$. In addition, the operator $\tilde {\bf A}_0: L^1(G\times S\times I)^3\to L^1(G\times S\times I)^3$
is $m$-dissipative. In fact, the equation $(\lambda I-\tilde{\bf A}_0)\psi=\tilde f$
is nothing more than
\[
\sqrt{{{2E}\over {m_j}}}\omega\cdot \nabla\psi_j+\lambda\psi_j=\tilde f_j,
\quad {\psi_j}_{|\Gamma_-}=0,\ j=1,2,3,
\]
or equivalently for $E>0$
\[
&\omega\cdot \nabla\psi_j+\lambda\sqrt{{{m_j}\over {2E}}}\psi_j=
\sqrt{{{m_j}\over {2E}}}\tilde f_j\\
&{\psi_j}_{|\Gamma_-}=0,\ j=1,2,3
\]
whose solution for each $j$ is, by sections \ref{lss3} and \ref{diss1} (see \eqref{l21a}, \eqref{resolA}),
\[
\psi_j=\sqrt{{{m_j}\over {2E}}}\int_0^{t(x,\omega)}e^{-\lambda t\sqrt{m_j/(2E)}}
\tilde f_j(x-\omega t,\omega,E) dt.
\]
Similarly as in sections \ref{diss1b} and \ref{cos1} we see that
\[
\n{(\lambda I-\tilde{\bf A}_0)\psi}_{L^1(G\times S\times I)^3}
\geq \lambda \n{\psi}_{L^1(G\times S\times I)^3},\ \psi\in D(\tilde{\bf A}_0)
\]
and that $R(\lambda I-\tilde{\bf A}_0)=L^1(G\times S\times I)^3$. Hence $\tilde {\bf A}_0$ is $m$-dissipative.

Since $\tilde{\bf A}_0$ is $m$-dissipative and $\tilde\Sigma-\tilde K$ is bounded, the operator $ \tilde{\bf A}_0-\tilde\Sigma+\tilde K:L^1(G\times S\times I)^3\to
L^1(G\times S\times I)^3$ generates a $C^0$-semigroup, which we denote by $\tilde G(t)$
(\cite{dautraylionsv5}, p. 348, \cite{engelnagel}, Theorem III.1.3., pp. 158, \cite{goldstein} and \cite{pazy83}, pp. 76-77)
and which satisfies the estimate
\[
\n{\tilde G(t)}\leq e^{\n{\tilde\Sigma-\tilde K}t},\quad \forall t\geq 0.
\]

We get the following standard result from the theory of abstract Cauchy problems for $g=0$.

\begin{theorem}\label{tdth1}
Suppose that the assumptions (\ref{scateh}) and (\ref{colleh}) are valid.
Furthermore, suppose that  $\tilde f\in C^1([0,T],L^1(G\times S\times I)^3)$ and
$\psi_0\in D(\tilde {\bf A}_0) = \tilde {\s W}^1_{-,0}(G\times S\times I)^3$. Then the   problem (\ref{td3a})-(\ref{td3aa}) for $g=0$ has a unique solution
\be\label{td4}
\psi\in C^1([0,T],L^1(G\times S\times I)^3)\cap C([0,T],{\s W}^1(G\times S\times I)^3),
\ee
such that
\be\label{td5}
\psi(t)\in  D(\tilde {\bf A}_0) =\tilde {\s W}^1_{-,0}(G\times S\times I)^3\ {\rm for \ all}\ t\geq 0.
\ee
Moreover, this solution is given by
\be\label{td6}
\psi(t)=\tilde G(t)\psi_0+\int_0^t\tilde G(t-s)\tilde f(s) ds.
\ee
\end{theorem}

\begin{proof}
Theorem follows from the solution theory of abstract Cauchy problems. See e.g. \cite{dautraylionsv5}, pp. 397-400, \cite{engelnagel}, Corollary VI.7.6., pp. 439, \cite{goldstein}, \cite{pazy83}, pp. 105-108.
\end{proof}

The next theorem includes a non-zero inflow boundary data.

\begin{theorem}\label{tdth2}
Suppose that the assumptions (\ref{scateh}) and (\ref{colleh})  are valid. Furthermore, suppose that  $\tilde f\in C^1([0,T],L^1(G\times S\times I)^3)$,
$\psi_0\in \tilde {\s W}^1(G\times S\times I)^3$ and
$g\in C^2([0,T],{\s T}^1(\Gamma_-)^3)$ such that 
\be\label{eq:compat}
g(0)={\psi_0}_{|\Gamma_-}.
\ee
Then the   problem (\ref{td3a})-(\ref{td3aa})
has a unique solution
\be\label{td7}
\psi\in C^1([0,T],L^1(G\times S\times I)^3)\cap C([0,T],{\s W}^1(G\times S\times I)^3)
\ee
such that
\be\label{td8}
\psi(t)_{|\Gamma_-}=g(t),\ {\rm for \ all}\ t\geq 0.
\ee
\end{theorem}

The condition \eqref{eq:compat} is called a {\it compatibility condition}: one must have $g(y,\omega,E,0)=\psi_0(y,\omega,E)$ for a.e. $(y,\omega,E)\in\Gamma_-$.

\begin{proof} Similarly as in Lemma \ref{lift2}  
for each $j$ there exists a lift
\[
\tilde\psi_j=Lg_j\in C^2([0,T],\tilde {\s W}^1(G\times S\times I))
\]
such that $\tilde{\psi_j}_{|\Gamma_-\times [0,T]}=g_j$.

Define $\tilde\psi=(\tilde\psi_1,\tilde\psi_2,\tilde\psi_3)$ and substitute $u=\psi-\tilde\psi$ for $\psi$ in problem (\ref{td3a})-(\ref{td3aa}) to obtain
\bea\label{td10}
&{\p {u_j}{t}} +v_j\cdot\nabla u_j+\tilde\Sigma_ju -\tilde K_ju \\
&\hspace{1.1cm} =\tilde f_j-
{\p {\tilde\psi_j}{t}} -v_j\cdot\nabla \tilde\psi_j-\tilde\Sigma_j\tilde\psi +\tilde K_j\tilde\psi=:\ol f_j
\quad
&& {\rm on}\ G\times S\times I\times ]0,T] &\nonumber\\
&  u(y,t)=g(y,t)-g(y,t)=0&& {\rm on}\ \Gamma_-\times ]0,T] &\nonumber \\
& u(\cdot,0)=\psi_0-\tilde\psi(0)\in \tilde {\s W}^1_{-,0}(G\times S\times I)^3 && {\rm on}\ G\times S\times I. & \nonumber
\eea
Notice that in the last step we have by the compatibility condition 
$\tilde\psi(0)_{|\Gamma_-}=(Lg)(0)_{|\Gamma_-}=g(0)={\psi_0}_{|\Gamma_-}$
and therefore $u(\cdot,0)\in \tilde {\s W}^1_{-,0}(G\times S\times I)^3$ (here $Lg=(Lg_1,Lg_2,Lg_3)$).
In addition we find that $\ol f_j\in C^1(]0,T],L^1(G\times S\times I)^3)$
(we omit the details here).
By Theorem \ref{tdth1} the problem (\ref{td10}) has a unique solution
\[
u\in  C([0,T],L^1(G\times S\times I)^3)\cap C^1(]0,T],L^1(G\times S\times I)^3)
\cap C(]0,T],{\s W}^1(G\times S\times I)^3)
\]
such that $u(t)\in \tilde {\s W}^1_{-,0}(G\times S\times I)^3$ for $t\geq 0$. Then $\psi:=u+\tilde\psi$ is the required unique solution of
the problem (\ref{td3a})-(\ref{td3aa}) and the proof is complete.
\end{proof}

The non-negativity of solutions in this dynamical case follows as above for steady state solutions. Here we denote $f_j(x,\omega,E,t)=f_j(t)(x,\omega,E)$ and so on. We have

\begin{theorem}\label{tdth3}
Suppose that the assumptions of Theorem \ref{tdth2} are valid and that moreover
for $j=1,2,3$,
\bea\label{td11}
f_j(x,\omega,E,t)&\geq 0,\quad {\rm for}\ t\geq 0\ {\rm and\ a.e.}\ (x,\omega,E)\in G\times S\times I\\
g_j(y,\omega,E,t)&\geq 0,\quad {\rm for}\ t\geq 0\ {\rm and\ a.e.}\ (y,\omega,E)\in \Gamma_-, \\
\psi_0(x,E,\omega)&\geq 0,\quad {\rm for\ a.e.}\ (x,\omega,E)\in G\times S\times I.
\eea
Then the solution of the problem (\ref{td3a})-(\ref{td3aa}) given in Theorem \ref{tdth2} satisfies $\psi(x,\omega,E,t)\geq 0$ for $t\geq 0$ and a.e. $(x,\omega,E)\in G\times S\times I$.
\end{theorem}

\begin{proof}
The problem (\ref{td3a})-(\ref{td3aa}) has the form
\bea\label{td12}
& {\p {\psi}t}-\tilde{\bf A}\psi+\tilde\Sigma\psi -\tilde K\psi=\tilde f,\\
&  \psi_{|\Gamma_-\times ]0,T]}=g\\
&\psi(0)=\psi_0.
\eea

A. In the first step, we assume that $g=0$ (then by the assumptions of Theorem \ref{tdth2} $\psi_0\in D(\tilde {\bf A}_0)$). Let $\tilde T(t)$  be the $C^0$-semigroup generated by $\tilde{\bf A}_0$.
Then one has
\[
\tilde{T}(t)f\geq 0\quad \textrm{for all}\ \tilde f\in L^1(G\times S\times I)^3\ \textrm{such that}\  \tilde f\geq 0.
\]
Indeed, according to what was done in section \ref{coss3}, one has
\[
(T(t)\tilde f)(x,\omega,E)=((T_1(t)\tilde f_1)(x,\omega,E),(T_2(t)\tilde f_2)(x,\omega,E),(T_3(t)\tilde f_3)(x,\omega,E))
\]
where (see \eqref{sgA})
\[
(T_j(t)\tilde f_j)(x,\omega,E)=
H\Big(t(x,\omega)-t\sqrt{{{2E}\over{m_j}}}\Big)\tilde f_j\Big(x-\omega\sqrt{{{2E}\over{m_j}}}t,\omega,E\Big),\ j=1,2,3
\]
for $ \tilde f\in  L^1(G\times S\times I)^3$.

Moreover, we have
\[
\tilde K\psi\geq 0\ {\rm for}\  \psi\in L^1(G\times S\times I)^3; \psi\geq 0
\]
and $\tilde{\Sigma}\in L^\infty(G\times S\times I)^3$ such that $\tilde{\Sigma}_j \geq 0$.
Since $\tilde K:
L^1(G\times S\times I)^3\to L^1(G\times S\times I)^3$ is a bounded operator we obtain,
as earlier in the proof of Theorem \ref{coth4},
by Trotter's formula 
that $\tilde G(t)\tilde{f}\geq 0$ for $\tilde{f}\geq 0$
(cf. also Proposition 2 of \cite{dautraylionsv6}, pp. 226-227) , which implies that
\be\label{td14}
\psi=\tilde G(t)\psi_0+\int_0^t\tilde G(t-s)\tilde f(s) ds \geq 0.
\ee
Hence the assertion of the theorem is true for $g=0$.

B. Suppose more generally that $g\in C^2([0,T],{\s T}^1(\Gamma_-)^3)$
and $g\geq 0$. We see
(cf. \cite{dautraylionsv6}, pp. 231-232)
that the solution $u\in C^0([0,T],L^1(G\times S\times I)^3)\cap C^1(]0,T],L^1(G\times S\times I)^3) $ of the problem (which exists at least if $g(0)=0$)
\be\label{td15}
{\p {u}t}-\tilde{\bf A}u+\tilde\Sigma u=0,\quad u_{|\Gamma_-\times ]0,T]}=g,\quad u(0)=0
\ee
is non-negative.
Again the
problem (\ref{td15}) can be solved (as above in Section \ref{lss1}) by the Lagrange's method in the classical sense which we describe briefly in what follows, assuming that $g$ is sufficiently smooth, say $g\in C(\ol G\times S\times I\times [0,T])^3$, and that $g_j$ is zero in a neighbourhood of the surfaces
given by the equation $t=t(x,\omega)\sqrt{{{m_j}\over{2E}}}$ for $E>0$ in the phase space $\ol G\times S\times I\times [0,T]$.
For a general $g$ then, the {\it mild}
(generalized) solution is obtained by standard limiting processes (cf. the end of section \ref{diss1b}).

The equation (\ref{td15}) is uncoupled and for each $j$ it is of the form (we assume $E>0$)
\be\label{td12a}
\sqrt{{{m_j}\over{2E}}}{\p {u_j}{t}}+\sum_{k=1}^3\omega_k{\p {u_j}{x_k}}+\Sigma_j(x,\omega,E)u_j=0
\ee
and $u_j$ must satisfy an initial-boundary condition of the form
\be\label{td12b}
u_j(x,\omega,E,0)&=0&&\ {\rm for}\ (x,\omega,E)\in G\times S\times I, &\nonumber\\
u_j(y,\omega,E,t)&=g_j(y,\omega,E,t)&&\ {\rm for}\ (y,\omega,E,t)\in \Gamma_-\times ]0,T]. &
\ee

We solve the problem \eqref{td12a}-\eqref{td12b} for a fixed $j$ and we denote for simplicity  $u:=u_j$ and $m:=m_j,\ g:=g_j$.
The augmented system of ordinary differential equations is
\bea\label{auga}
T'(t)&=\sqrt{{{m}\over{2E}}}\nonumber\\
X_1'(s)&=\Omega_1, \quad \Omega_1'(s)=0, \nonumber\\
X_2'(s)&=\Omega_2, \quad \Omega_2'(s)=0, \nonumber \\
X_3'(s)&=\Omega_3, \quad \Omega_3'(s)=0, \nonumber \\
{\s E}'(s)&=0 \nonumber\\
U'(s)&= -\Sigma(X,\Omega,E)U
\eea
We find that
\bea \label{asola}
T(s)&=\sqrt{{{m}\over{2E}}}s+C_1,\ \Omega(s)=C_2,\ X(s)=C_2s+C_3,\ {\s E}(s)=C_4,\nonumber\\ \Psi(s)&=C_5
e^{\int_0^s-\Sigma(X(\tau),\Omega(\tau),E(\tau))d\tau}
\eea
where $C_1,\dots,C_5$ are some constants.

Taking into account the conditions in \eqref{td12b}, we see that the solution (the characteristics) of the augmented system must satisfy the initial condition of the form
\bea\label{augina}
(X(0),\Omega(0),{\s E}(0),T(0),U(0)) &=
(x',\omega,E,0),\ t'=0
\\
(X(0),\Omega(0),{\s E}(0),T(0),U(0)) &=
\big(h(v),\omega,E,t',g(h(v),\omega,E,t')\big),\ t'>0\nonumber
\eea
where $ h=h(v)$ is a local parametrization of $\partial G$. Also, $x$ (resp. $t$) was replaced by $x'$ (resp. $t'$) for notational reasons. Matching the initial condition (\ref{augina}) to the solution
(\ref{asola}) we get the solution $(X(s),\Omega(s),{\s E}(s),T(s),\Psi(s))$ and by eliminating
$x',\ v,\ t',\ \omega,\ E$ from the system
\[
(X(s),\Omega(s),{\s E}(s),T(s))=(x,\omega,E,t)
\]
and noting that for $t'>0$ (formally)
\[
\Psi(s)=H(t')g(x-s\omega,\omega,E,t')e^{\int_0^s-\Sigma(x-\tau\omega,\omega,E)d\tau},\quad \textrm{with} s=t(x,\omega),
\]
we get the solution $u$ as in section \ref{lss3}. The result is 
\bea\label{sgAab}
&u(x,\omega,E,t)\\
&=
H\left(t-\sqrt{{{m}\over{2E}}}t(x,\omega)\right)g\left(x-t(x,\omega)\omega,\omega,E,t-\sqrt{{{m}\over{2E}}}t(x,\omega)\right)e^{\int_0^{t(x,\omega)}-\Sigma(x-\tau \omega,\omega,E)d\tau}.\nonumber
\eea
where $H$ is again the Heaviside function.

For three particles system
the solution is (when $E>0$)
\[
u(x,\omega,E,t)=(u_1(x,\omega,E,t),u_2(x,\omega,E,t),u_3(x,\omega,E,t))
\]
where
\bea
&u_j(x,\omega,E,t)\\
&=
H\left(t-\sqrt{{{m_j}\over{2E}}}t(x,\omega)\right)g_j\left(x-t(x,\omega)\omega,\omega,E,t-\sqrt{{{m_j}\over{2E}}}t(x,\omega)\right)e^{\int_0^{t(x,\omega)}-\Sigma_j(x-\tau \omega,\omega,E)d\tau}.\nonumber
\eea

Let $w=\psi-u$. Then we find that
\bea\label{td16}
&{\p {w}t}-\tilde{\bf A}w+\tilde\Sigma w-\tilde Kw=\tilde f-\big({\p {u}t}-\tilde{\bf A}u+\tilde\Sigma u\big)+\tilde Ku=\tilde f+\tilde Ku\geq 0\\
& w_{|\Gamma_-\times [0,T]}=g-g=0\\
&w(0)=\psi(0)-u(0)=\psi_0\geq 0.
\eea
Notice that here ${\psi_0}_{|\Gamma_-}$ is not necessarily zero (that is, the compatibility condition is not necessarily true) and
that $\tilde f+\tilde Ku$ does not necessarily belong to the space
$C^1\big([0,T],L^1(G\times S\times I)^3\big)$.
The solution can be understood in the mild sense,
although we omit the treatment of these details here;
see e.g. \cite{engelnagel} Section VI.7.a, or \cite{pazy83}, p. 106.
Hence by part A. of the proof, $w\geq 0$
and hence $\psi=w+u\geq 0$ which completes the proof.
\end{proof}

%%%%%%%%%%%%%%%%%%%%%%%%%%%%%%%%%%%%%%%%%%%%%%%%%%%%%%%%%
\subsection{A Note on Regularity of Solutions}\label{fr}
%%%%%%%%%%%%%%%%%%%%%%%%%%%%%%%%%%%%%%%%%%%%%%%%%%%%%%%%%

Above  given existence results of solutions for the coupled system can be analogously obtained also for general $p\in [1,\infty[$. 
For $p>1$ one needs the both conditions (\ref{colleh}), (\ref{colleha}) and also  (\ref{co2a}), (\ref{co2aa}) of the cross-sections.
In \cite{tervo07} (see also \cite{bomanthesis}) we have shown the dissipativity of the scattering-collision operator and related existence results 
(for the stationary problem) of coupled system in the case $p=2$  by applying so-called 
{\it Lions-Lax-Milgram Theorem} (generalized Lax Milgram Theorem) (\cite{grisvard}, Lemma 4.4.4.1, p. 234). This approach offers for $p=2$ an alternative method. We also point out that the dimension $n$ of the Euclidean space $\R^n$ can be any $n\geq 1$ (i.e. not only $n=3$). We omit these generalizations in this paper.

Let $(m_1,m_2,m_3)\in \N_0^3$ be a multi-index and let $U\subset G\times S\times I^\circ$
be an open subset.
Define Sobolev spaces $H^{p,(m_1,m_2,m_3)}(U)$ (formally) by
\begin{multline*}
H^{p,(m_1,m_2,m_3)}(U)=\big\{\psi\in L^p(G\times S\times I)
\ |\ \partial_x^\alpha\partial_{\tilde{\omega}}^\beta\partial_E^\gamma\psi \in L^p(U),\\
\forall |\alpha|\leq m_1,\ |\beta|\leq m_2,\ |\gamma|\leq m_3\big\},
\end{multline*}
where the derivatives are taken in the distributional sense, and where $\tilde\omega=(\tilde\omega_1,\tilde\omega_2)$
is a local coordinate chart in $S$,
and $\partial_{\tilde\omega_j},\ j=1,2$ are the respective coordinate vector fields. The rigorous definition would involve, for example,
the use of multiple local charts of $S$ covering it,
along with an associated partition of unity,
or the use of covariant derivatives with respect to the Levi-Civita connection on $S$ (cf. \cite{hebey96} for this latter point of view).
We ignore, however, these minor technicalities here to stay brief.
Spaces $H^{p,(m_1,m_2,m_3)}(U)$ are mixed-norm
Sobolev spaces. They are Banach spaces when equipped with the respective norms
\bea
\n{\psi}_{H^{p,(m_1,m_2,m_3)}(U)}
=
\Big(\sum_{|\alpha|\leq m_1}\sum_{|\beta|\leq m_2}\sum_{|\gamma|\leq m_3}
\n{\partial_x^\alpha\partial_\omega^\beta\partial_E^\gamma\psi }^p_{L^p(U)})^{1/p}.
\eea 
Note that 
\be 
H^{p,(1,0,0)}(U)\cap H^{p,(0,1,0)}(U)\cap H^{p,(0,0,1)}(U)
=H^{p,1}(U)
\ee
and 
\bea
& 
H^{p,(2,0,0)}(U)\cap H^{p,(1,1,0)}\cap H^{p,(0,2,0)}(U)\cap H^{p,(1,0,1)}(U)\cap H^{p,(0,1,1)}(U)\cap H^{p,(0,0,2)}(U)\nonumber\\
&
=H^{p,2}(U),
\eea
where $H^{p,2}(U)$ is the usual (isotropic) Sobolev space.
The spaces  $H^{p,(m_1,m_2,m_3)}(\Gamma_-)$ can be defined in the similar fashion.
Finally one defines
\[
{W}^{p,(m_1,m_2,m_3)}(U)=\{\psi\in H^{p,(m_1,m_2,m_3)}(U)\ |\ \omega\cdot \nabla_x\psi\in H^{p,(m_1,m_2,m_3)}(U)\}
\]
These mixed-norm Sobolev spaces could be replaced by mixed-norm
Bessel potential spaces or Sobolev spaces with fractional index $(s_1,s_2,s_3)$
(so-called Sobolev-Slobodeckij spaces, cf. \cite{nezza11}). Then the multi-index $(m_1,m_2,m_3)$ is replaced by the tuple $(s_1,s_2,s_3)\in \R_+^3$.
Above $p\in [1,\infty[$  is the Lebesgue index and $s$ is the regularity index indicating how "smooth" the functions $f$ and $g$ are (the tuple $(s_1,s_2,s_3)$ refer to the orders of the distributional  derivatives of the functions of spaces under consideration).

A natural question one can, and should, pose is the following: What can be said about the regularity of a solution of the equation  
\[
(-{\bf A}+\Sigma-K)\psi=&f, \\
\psi_{|\Gamma_-}=&g,
\]
when the cross-sections $\Sigma_j,\ \sigma_{jk}$ and the data $f$ and $g$ are sufficiently regular, say  $f\in H^{p,(s_1,s_2,s_3)}(G\times S\times { I})^3$ and $g\in H^{p,(s_1',s_2',s_3')}(\Gamma_-)^3$  ?
Is it possible to conclude that $\psi\in  { W}^{p',(s_1'',s_2'',s_3'')}(U)^3$,
for some indexes $p',\ (s_1'',s_2'',s_3'')$ and some open subset $U$ of $G\times S\times I^\circ$ ?.

In addition, the same kind of questions  can be formulated for time-dependent problems. One possibility in both of these stationary and time-dependent problems for systematic study is to apply the extensive theory of pseudo-differential and especially singular integral operator theory (cf.  \cite{hsiao}).

The regularity results are important among others in the connection of numerical methods, for example in the case when one applies higher order spline approximations (to get more rapid convergence results).

For monokinetic (one-velocity), one particle transport equations some regularity results can been found in \cite{agoshkov} (Chapter 4, see also the introduction of the monograph for related literature). Regularity results therein are concerning for periodic solutions, solutions for so called plane-parallel problems and for problems in   three-dimensional domain $G$. Typically the increment of regularity is "small" (only of order $\leq 1$). The formulations are exhibited with the help of appropriate difference-differential norms (and the corresponding spaces), which are closely related to (or are the same as) the Bessel potential and/or Sobolev-Sobodetskij spaces.

The following example shows that in the case of transport problems, the regularity
of the solution
{\it does not generally} arise from the regularity of data and cross-sections in the sense that "the solution is more and more regular on the whole domain $G\times S\times I$ when the data and cross-sections are more and more regular".

\begin{example}\label{exreg}
 
Let $G=B(0,r)\subset \R^3$ and consider the problem (for one particle)
\[
\omega\cdot\nabla\psi+\psi&=1, \\[2mm]
\psi_{|\Gamma_-}&=0.
\]
By (\ref{l21a}) the solution of the problem is
\[
\psi=1-e^{-t(x,\omega)},
\]
where, by virtue of Example \ref{le1}, for $(x,\omega)\in G\times S$,
\[
t(x,\omega)=x\cdot \omega +\sqrt{(x\cdot \omega)^2+r^2-\n{x}_2^2},
\]
where for clarity we denote $\n{x}_2^2=x\cdot x$.
In the present example,
$\Sigma=1$, $K=0$, $f=1\in H^{(p,(\infty,\infty,\infty))}(G\times S\times I)$ and $g=0\in H^{(p,(\infty,\infty,\infty))}(\Gamma_-)$. 

We see that $\psi\in C^\infty(G\times S\times I)$ since $r^2-\n{x}^2>0$ for $x\in G$. Hence $\psi\in  H^{(p,(\infty,\infty,\infty))}(U)$ for any subset $U\subset G\times S\times I$ which is of the form $U=G'\times S\times I$ where $G'\subset G$ is open
such that $\ol{G'}\subset G$. In particular, $\psi\in  H^{(p,(\infty,\infty,\infty))}(U_\epsilon)$
for any $U_\epsilon:=G_\epsilon\times S\times I$ where $G_\epsilon :=B(0,r-\epsilon)=\{x\in G\ |\ d(x,\partial G)>\epsilon\}$. 

We shall show that $\psi\in H^{(p,(1,0,0))}(G\times S\times I)$ for any $1\leq p<3$, but
$\psi\not\in H^{(p,(1,0,0))}(G\times S\times I)$ when $p\geq 3$.

We will occasionally denote the (surface) measure on $S$ by $\mu_S$,
and recall that it is induced by the Lebesgue measure.
Moreover, $\mu_S$ is $\mathrm{SO}(3)$ invariant.

Let $p\geq 1$. Since $\psi$ is independent of $E$, and since $I$ is bounded, we can leave $E$ away and consider computations in spaces $H^{(p,(1,0,0))}(G\times S)$ only. We find that
\[
{\p {\psi}{x_j}}=e^{-t(x,\omega)}{\p t{x_j}}
=e^{-t(x,\omega)}\omega_j
+
e^{-t(x,\omega)}
{{(x\cdot\omega)\omega_j-x_j}\over{\big((x\cdot\omega)^2+r^2-\n{x}_2^2\big)^{1/2}}}
=:u_1+u_2.
\]
Since $e^{-2r}\leq e^{-t(x,\omega)}\leq 1$, and $|\omega_j|\leq 1$,
we observe that
\begin{align}\label{eq:example:exreg:subresult:1}
{\p {\psi}{x_j}}\in L^p(G\times S)
\quad\iff\quad
u_2\in L^p(G\times S)
\quad\iff\quad
I_{p,j}<\infty,
\end{align}
where for $j=1,2,3$, and $p\geq 1$,
\begin{align}\label{eq:I_p_j}
I_{p,j}
:=
\int_G\int_S{{|(x\cdot\omega)\omega_j-x_j|^p}\over{\big((x\cdot\omega)^2+r^2-\n{x}_2^2\big)^{p/2}}}d\omega dx.
\end{align}
Write also,
\[
I_{p}:=(I_{p,1}, I_{p,2}, I_{p,3}).
\]

By using spherical transformation $x=s\Omega$, $(s,\Omega)\in [0,r[\times S$,
for integration over $G=B(0,r)$,
we obtain ($p\geq 1$)
\begin{align}\label{eq:I_p_j:2}
I_{p,j}=&\int_G\int_S{\frac{|(x\cdot\omega)\omega_j-x_j|^p}{((x\cdot\omega)^2+r^2-\n{x}^2)^{p/2}}} \diff\omega \diff x
=\int_{S\times S} \int_0^r \frac{|s(\Omega\cdot\omega)\omega_j-s\Omega_j|^p}{(s^2(\Omega\cdot\omega)^2+r^2-s^2)^{p/2}} s^2\diff s\diff \Omega\diff \omega \nonumber\\
=&\int_{S\times S}|(\Omega\cdot\omega)\omega_j-\Omega_j|^p \int_0^r \frac{s^{p+2}}{(-s^2(1-(\Omega\cdot\omega)^2)+r^2)^{p/2}}\diff s\diff\Omega\diff \omega \nonumber\\
=&\int_{S\times S}|(\Omega\cdot\omega)\omega_j-\Omega_j|^p J_{p}\big(1-(\Omega\cdot \omega)^2\big) \diff\Omega\diff\omega,
\end{align}
(integrand is non-negative, so we were allowed to apply Fubini's theorem in the first step)
where
\begin{align}\label{eq:J_p}
J_{p}(c)=\int_0^r \frac{s^{p+2}}{(-cs^2+r^2)^{p/2}}\diff s,
\quad c \in [0,1].
\end{align}
Notice that if $c=1-(\Omega\cdot \omega)^2$, then $0\leq c\leq 1$.

Performing two consecutive changes of variables, first $t=r/s$ and then $v=t^2$,
the integral $J_{p}(c)$ can be brought into the form
\begin{align}\label{eq:J_p:2}
J_{p}(c)=&\int_0^r \frac{s^{p+2}}{(-cs^2+r^2)^{p/2}}\diff s
=\int_0^r \frac{s^{2}}{(-c+(r/s)^2)^{p/2}}\diff s
=r^{3}\int_1^{\infty} \frac{1}{(-c+t^2)^{p/2} t^{4}}\diff t \nonumber\\
=&\frac{r^{3}}{2}\int_1^{\infty} \frac{1}{(-c+v)^{p/2}v^{5/2}}\diff v.
\end{align}

If $0\leq c<1$ and $v\geq 2$, then
$v\geq -c+v>1$, and 
\[
\frac{1}{v^{(p+5)/2}}\leq \frac{1}{(-c+v)^{p/2}v^{5/2}}\leq \frac{1}{v^{5/2}},
\]
\begin{align}\label{eq:estimate:J_p}
\frac{r^3}{2}L_{p+5}\leq J_{p}(c)-\frac{r^{3}}{2}\int_1^2 \frac{1}{(-c+v)^{p/2}v^{5/2}}\diff v\leq \frac{r^3}{2}L_{5},
\end{align}
where
\[
L_t:=\int_2^{\infty} \frac{1}{v^{t/2}}\diff v<\infty,\quad \forall t>2.
\]

We have moreover,
\begin{align}\label{eq:example:exreg:estimate:anon}
\frac{1}{2^{5/2}}\int_1^2 \frac{1}{(-c+v)^{p/2}}\diff v
\leq
\int_1^2 \frac{1}{(-c+v)^{p/2}v^{5/2}}\diff v
\leq
\int_1^2 \frac{1}{(-c+v)^{p/2}}\diff v.
\end{align}
Clearly, 
\begin{align}\label{eq:example:exreg:estimate:anon:2}
\int_{S} \int_{S} |(\Omega\cdot\omega)\omega_j-\Omega_j|^{p} \diff\Omega\diff\omega
\leq 2^o \mu_S(S)^2<\infty,
\end{align}
and hence for all $p\geq 1$,
\begin{multline}\label{eq:example:exreg:subresult:3}
I_{p,j}<\infty
\quad\iff\quad
\int_S \int_{S} |(\Omega\cdot\omega)\omega_j-\Omega_j|^p \ol{J}_p\big(1-(\Omega\cdot\omega)^2\big)d\Omega d\omega<\infty,
\end{multline}
where
\begin{align}\label{eq:ol_J_p}
\ol{J}_p(c):={}&\int_1^2 \frac{1}{(-c+v)^{p/2}}\diff v.
\end{align}

We notice that when $0<c<1$,
\begin{align}
\ol{J}_p(c)
=
\begin{cases}
\displaystyle \frac{2}{p-2}\Big(\frac{1}{(1-c)^{p/2-1}}-\frac{1}{(2-c)^{p/2-1}}\Big),
\ &\textrm{if}\ p\neq 2 \\[5mm]
\displaystyle \ln\Big(\frac{-c+2}{-c+1}\Big),
\ &\textrm{if}\ p=2
\end{cases}
\end{align}

The rest of the analysis will be split into three parts (i), (ii) and (iii)
depending on values of $p\geq 0$.

\begin{itemize}
\item[(i)] {\bf Case $p<2$.}
Writing $\alpha=1-p/2>0$, and noticing that whenever $0\leq c\leq 1$, we have
\[
\ol{J}_p(c)=\frac{1}{\alpha}\big((2-c)^{\alpha}-(1-c)^{\alpha}\big)
\leq \frac{1}{\alpha}2^{\alpha},
\]
which implies, taking into account \eqref{eq:example:exreg:estimate:anon:2},
\eqref{eq:example:exreg:subresult:3},
and the fact that $0\leq c\leq 1$ if $c=1-(\Omega\cdot\omega)^2$, for $\Omega,\omega\in S$,
\begin{align}\label{eq:example:exreg:convergence:1}
I_{p,j}<\infty,\quad \forall 1\leq p<2, \quad j=1,2,3.
\end{align}
In particular, by \eqref{eq:example:exreg:subresult:1},
\begin{align}\label{eq:example:exreg:result:1}
\nabla \psi\in L^p(G\times S)^3,\quad&\mathrm{if}\ 1\leq p<2, 
\end{align}

\item[(ii)] {\bf Case $p>2$.}
Notice that for $0\leq c<1$,
we have
\begin{multline}\label{eq:estimate:ol_J_p}
0<\frac{2}{p-2}\Big(\frac{1}{(1-c)^{p/2-1}}-1\Big) \\
\leq
\ol{J}_p(c)
=
\frac{2}{p-2}\Big(\frac{1}{(1-c)^{p/2-1}}-\frac{1}{(2-c)^{p/2-1}}\Big) \\
\leq
\frac{2}{p-2}\frac{1}{(1-c)^{p/2-1}},
\end{multline}
which implies, due to \eqref{eq:example:exreg:estimate:anon:2}, \eqref{eq:example:exreg:subresult:3}, \eqref{eq:ol_J_p},
and the fact that $0\leq c<1$ when $c=1-(\Omega\cdot\omega)^2$,
for $\Omega,\omega\in S$ such that $\Omega\cdot\omega\neq 0$ (the set $\{\Omega\in S\ |\ \Omega\cdot\omega=0\}$ being $\mu_S$-zero measurable on $S$ for all $\omega\in S$),
\begin{multline*}
I_{p,j}<\infty
\quad\iff\quad
\int_S \int_{S} \frac{|(\Omega\cdot\omega)\omega_j-\Omega_j|^p}{|\Omega\cdot\omega|^{p-2}}d\Omega d\omega<\infty,
\end{multline*}
for $j=1,2,3$.

Letting $\n{\cdot}_p$ denote the $p$-norm on $\R^3$, we have by the above equivalence
(recall that $I_p=(I_{p,1}, I_{p,2}, I_{p,3}$)
\[
\nabla\psi\in L^p(G\times S)^3
\quad\iff\quad {}&
\n{I_{p}}_p<\infty \\
\quad\iff\quad {}&
\int_S \int_{S} \frac{\n{(\Omega\cdot\omega)\omega-\Omega}_p^p}{|\Omega\cdot\omega|^{p-2}}d\Omega d\omega<\infty,
\]
and since the norms $\n{\cdot}_p$ and $\n{\cdot}_2$ are equivalent,
\[
\n{I_{p}}_p<\infty
\quad\iff\quad
\int_S \int_{S} \frac{\n{(\Omega\cdot\omega)\omega-\Omega}_2^p}{|\Omega\cdot\omega|^{p-2}}d\Omega d\omega<\infty.
\]
Of course for $p\geq 1$, the condition
$\n{I_{p}}_p<\infty$
is equivalent to having
$I_{p,j}<\infty$ for all $j=1,2,3$.

Notice that $(\Omega\cdot\omega)\omega-\Omega$ has $2$-norm (Euclidean norm)
\[
\n{(\Omega\cdot\omega)\omega-\Omega}_2^2
=1-(\Omega\cdot\omega)^2,
\]
and hence we have arrived at the result that
\begin{align*}
\n{I_{p}}_p<\infty
\quad\iff\quad
\int_S \int_{S} \frac{|1-(\Omega\cdot\omega)^2|^{p/2}}{|\Omega\cdot\omega|^{p-2}}d\Omega d\omega<\infty.
\end{align*}

But, if one employs the $\mathrm{SO}(3)$-invariance of $\mu_S$ in the integral with respect to $\Omega$,
choosing $M_{\omega}\in \mathrm{SO}(3)$ for each $\omega$
such that $M_{\omega}\omega=e_3$, we have
\[
&{}\int_S \int_{S} \frac{|1-(\Omega\cdot\omega)^2|^{p/2}}{|\Omega\cdot\omega|^{p-2}}d\Omega d\omega
=
\int_S \int_{S} \frac{|1-((M_{\omega}^{-1}\Omega)\cdot\omega)^2|^{p/2}}{|(M_{\omega}^{-1}\Omega)\cdot\omega|^{p-2}}d\Omega d\omega \\
=&{}
\int_S \int_{S} \frac{|1-\Omega_3^2|^{p/2}}{|\Omega_3|^{p-2}}d\Omega d\omega
=
\mu_S(S)\int_{S} \frac{|1-\Omega_3^2|^{p/2}}{|\Omega_3|^{p-2}}d\Omega.
\]

Using spherical coordinates, the last integral can be written as
\begin{multline*}
\int_{S} \frac{|1-\Omega_3^2|^{p/2}}{|\Omega_3|^{p-2}}d\Omega
=2\pi\int_0^\pi \frac{|1-\cos^2(\theta)|^{p/2}}{|\cos(\theta)|^{p-2}}\sin(\theta)d\theta \\
=4\pi\int_{0}^{\pi/2}\frac{|1-\sin^2(\theta)|^{p/2}}{\sin^{p-2}(\theta)}\cos(\theta)d\theta,
\end{multline*}
while by a change of variables $y=\sin(\theta)$,
\[
\int_{0}^{\pi/2}\frac{|1-\sin^2(\theta)|^{p/2}}{\sin^{p-2}(\theta)}\cos(\theta)d\theta
=\int_0^1 \frac{1-y^2}{y^{p-2}}dy=\int_0^1 \Big(\frac{1}{y^{p-2}}-\frac{1}{y^{p-4}}\Big)dy.
\]
It is clear that the right hand side integral is finite if $p<3$,
and that it diverges to $\infty$ if $p\geq 3$,
that is we have
\begin{align}\label{eq:example:exreg:convergence:2}
\n{I_{p}}_p<\infty \quad&\mathrm{if}\ 2<p<3, \nonumber \\[2mm]
\n{I_{p}}_p=\infty \quad&\mathrm{if}\ p\geq 3.
\end{align}

This allows us to conclude the case $p>2$,
with the result that
\begin{align}\label{eq:example:exreg:result:2}
\nabla \psi\in L^p(G\times S)^3\quad&\mathrm{if}\ 2<p<3, \nonumber \\[2mm]
\nabla \psi\notin L^p(G\times S)^3\quad&\mathrm{if}\ p\geq 3.
\end{align}

\item[(iii)] {\bf Case $p=2$.}
In this case, for $0<c<1$,
\[
\ol{J}_2(c)=\ln\Big(\frac{-c+2}{-c+1}\Big),
\]
and therefore
\[
{}&\int_S \int_{S} |(\Omega\cdot\omega)\omega_j-\Omega_j|^2 \ol{J}_2\big(1-(\Omega\cdot\omega)^2\big)d\Omega d\omega \\
={}&\int_S \int_{S} |(\Omega\cdot\omega)\omega_j-\Omega_j|^2 \ln\Big(\frac{1+(\Omega\cdot\omega)^2}{(\Omega\cdot\omega)^2}\Big)d\Omega d\omega.
\]
Since
\begin{multline*}
0\leq \int_S \int_{S} |(\Omega\cdot\omega)\omega_j-\Omega_j|^2\ln (1+(\Omega\cdot\omega)^2)d\Omega d\omega
\leq 4(\ln 2)\mu_S(S)^2<\infty,
\end{multline*}
we have, by \eqref{eq:example:exreg:subresult:3},
the equivalence
\[
I_{2,j}<\infty
\quad\iff\quad
\int_S \int_{S} |(\Omega\cdot\omega)\omega_j-\Omega_j|^2\ln \Big(\frac{1}{(\Omega\cdot\omega)^2}\Big)d\Omega d\omega<\infty.
\]

From this, by summing over $j$ and recalling that
$\n{(\Omega\cdot\omega)\omega-\Omega}_2^2=1-(\Omega\cdot\omega)^2$,
we obtain
\begin{align*}
\nabla\psi\in L^2(G\times S)
\quad\iff\quad {}&
\n{I_2}_2<\infty \nonumber\\
\quad\iff\quad {}&
\int_S \int_{S} |1-(\Omega\cdot\omega)^2|\ln \Big(\frac{1}{(\Omega\cdot\omega)^2}\Big)d\Omega d\omega<\infty.
\end{align*}

Making use of $\mathrm{SO}(3)$ invariance of $\mu_S$
in the integral with respect to $\Omega$,
(choose $M_{\omega}\in\mathrm{SO}(3)$ such that $M_{\omega}\omega=e_3$),
and then moving to spherical coordinates,
the integral on right hand side of the above equivalence becomes
\[
{}&\int_S \int_{S} |1-(\Omega\cdot\omega)^2|\ln \Big(\frac{1}{(\Omega\cdot\omega)^2}\Big)d\Omega d\omega
=-\int_S \int_{S} |1-\Omega_3^2|\ln (\Omega_3^2)d\Omega d\omega \\
={}&
-4\pi\mu_S(S)\int_0^\pi |1-\cos^2(\theta)|\ln|\cos(\theta)|\sin(\theta)d\theta \\
={}&
-8\pi\mu_S(S)\int_0^{\pi/2} |1-\sin^2(\theta)|\ln(\sin(\theta))\cos(\theta)d\theta.
\]
The change of variables, $y=\sin(\theta)$ allows us to transform the previous integral into
\[
-\int_0^{\pi/2} |1-\sin^2(\theta)|\ln(\sin(\theta))\cos(\theta)d\theta
=-\int_0^1 (1-y^2)\ln(y)dy,
\]
which is clearly finite, and thus
\begin{align}\label{eq:example:exreg:convergence:3}
\n{I_2}_2<\infty.
\end{align}

We have thus reached the conclusion of the case $p=2$,
namely
\begin{align}\label{eq:example:exreg:result:3}
\nabla \psi\in L^2(G\times S)^3.
\end{align}
\end{itemize} 

To conclude this example,
we have by 
\eqref{eq:example:exreg:result:1},\eqref{eq:example:exreg:result:2} and \eqref{eq:example:exreg:result:3},
\[
&\psi\in H^{(p,(1,0,0))}(G\times S),\quad \textrm{if}\ 1\leq p<3 \\
&\psi\notin H^{(p,(1,0,0))}(G\times S),\quad \textrm{if}\ p\geq 3.
\]
\end{example}

\begin{example}\label{counterex}
In the previous Example \ref{exreg} we have shown that $\psi\in H^{2,(1,0,0)}(G\times S)$.
The goal of this example is to establish that $\psi\not\in  H^{(2,(2,0,0))}(G\times S)$.

Recall that we have
\[
t(x,\omega)={}&x\cdot \omega +\sqrt{(x\cdot \omega)^2+r^2-\n{x}_2^2} \\[2mm]
\psi={}&1-e^{-t(x,\omega)},
\]
and therefore
\[
{\p t{x_j}}={}&\omega_j+\frac{(x\cdot\omega)\omega_j-x_j}{\big((x\cdot\omega)^2+r^2-\n{x}_2^2\big)^{1/2}} \\
{\q t{x_j}}={}&\frac{\omega_j^2-1}{\big((x\cdot\omega)^2+r^2-\n{x}_2^2\big)^{1/2}}
-\frac{((x\cdot\omega)\omega_j-x_j)^2}{\big((x\cdot\omega)^2+r^2-\n{x}_2^2\big)^{3/2}}
\]
as well as
\[
{\p \psi{x_j}}={}&e^{-t(x,\omega)}{\p t{x_j}} \\[2mm]
{\q \psi{x_j}}={}&-e^{-t(x,\omega)}\Big({\p t{x_j}}\Big)^2+
e^{-t(x,\omega)}{\q t{x_j}}.
\]
Substituting the above expressions for ${\p t{x_j}}$
and ${\q t{x_j}}$ into the last formula, one obtains
\bea\label{ce1}
{\q \psi{x_j}}={}&-e^{-t(x,\omega)}
\Big(\omega_j+\frac{(x\cdot\omega)\omega_j-x_j}{\big((x\cdot\omega)^2+r^2-\n{x}_2^2\big)^{1/2}}\Big)^2\nonumber\\
{}&
+e^{-t(x,\omega)}\Big(\frac{\omega_j^2-1}{\big((x\cdot\omega)^2+r^2-\n{x}_2^2\big)^{1/2}}
-\frac{((x\cdot\omega)\omega_j-x_j)^2}{\big((x\cdot\omega)^2+r^2-\n{x}_2^2\big)^{3/2}}\Big)\nonumber\\
={}&
-\omega_j^2e^{-t(x,\omega)}-2\omega_j e^{-t(x,\omega)}
\frac{(x\cdot\omega)\omega_j-x_j}{\big((x\cdot\omega)^2+r^2-\n{x}_2^2\big)^{1/2}}\nonumber\\
&
-e^{-t(x,\omega)}{{((x\cdot\omega)\omega_j-x_j)^2}\over{(x\cdot\omega)^2+r^2-\n{x}_2^2}}
+
(\omega_j^2-1)e^{-t(x,\omega)}{1\over{\big((x\cdot\omega)^2+r^2-\n{x}_2^2\big)^{1/2}}}\nonumber\\
&
-e^{-t(x,\omega)}{{((x\cdot\omega)\omega_j-x_j)^2}\over{\big((x\cdot\omega)^2+r^2-\n{x}_2^2\big)^{3/2}}}\nonumber\\
=:{}& u_{1,j}+u_{2,j}+u_{3,j}+u_{4,j}+u_{5,j}.
\eea

Making use of the fact that $e^{-2r}\leq e^{-t(x,\omega)}\leq 1$ and $|\omega_j|\leq 1$ for $(x,\omega)\in G\times S$,
we find that
$u_{1}:=(u_{1,1},u_{1,2},u_{1,3})\in L^2(G\times S)^3$.
By the result \eqref{eq:example:exreg:convergence:3} in Example \ref{exreg}, we have
$u_2:=(u_{2,1},u_{2,2},u_{2,3})\in L^2(G\times S)^3$.

Calculating in a similar fashion as in Example \ref{exreg}, we find that the integral
\[
I_j':= 
\int_G\int_S \frac{1}{(x\cdot\omega)^2+r^2-\n{x}_2^2}d\omega dx
\]
is finite for $j=1,2,3$,
and hence
$u_4:=(u_{4,1}, u_{4,2}, u_{4,3})\in L^2(G\times S)^3$.

Since, again $e^{-2r}\leq e^{-t(x,\omega)}\leq 1$,
we see that the convergence of
$\n{{\q \psi{x_j}}}_{L^2(G\times S)}$ for all $j=1,2,3$,
is equivalent to the convergence of
\begin{multline*}
\int_{G\times S} |e^{t(x,\omega)}(u_{3,j}+u_{5,j})|^2dx d\omega \\
=
\int_{G\times S}\Big|{{((x\cdot\omega)\omega_j-x_j)^2}\over{(x\cdot\omega)^2+r^2-\n{x}_2^2}}
+
{{((x\cdot\omega)\omega_j-x_j)^2}\over{\big((x\cdot\omega)^2+r^2-\n{x}_2^2\big)^{3/2}}}\Big|^2 dx d\omega\\
\geq 
\int_{G\times S}
\frac{((x\cdot\omega)\omega_j-x_j)^4}{\big((x\cdot\omega)^2+r^2-\n{x}_2^2\big)^{2}} dx d\omega,
\end{multline*}
for all $j=1,2,3$.
However, choosing $p=4$ in \eqref{eq:I_p_j},
the result \eqref{eq:example:exreg:convergence:2}
implies that the sum of integrals over $j=1,2,3$ on the right hand side
diverges to $\infty$,
whence we may conclude that
\[
\sum_{j=1,2,3} \n{{\q \psi{x_j}}}^2_{L^2(G\times S)}=\infty,
\]
and therefore
\[
\psi\not\in H^{2,(2,0,0)}(G\times S).
\]
\end{example}

\begin{example}\label{counterex:2}
Let $s=1+\kappa,\ 0<\kappa<1$. Recall that the norm 
in the fractional Sobolev-Slobodevskij spaces $H^{2,(s,0,0)}(G\times S)$ is given by
\be
\n{\psi}_{H^{2,(s,0,0)}(G\times S)}^2
=
\n{\psi}_{H^{2,(1,0,0)}(G\times S)}^2
+ \n{\psi}_{H^{2,(\kappa,0,0)}(G\times S)}^2,
\ee
where
(see \cite{nezza11}, p. 5, Eq. (2.2) with $n=3$)
\be 
\n{\psi}_{H^{2,(\kappa,0,0)}(G\times S)}^2
:=\sum_{j=1}^3
\int_{S}\int_G\int_G{{|{\p \psi{x_j}}(x,\omega)-{\p \psi{x_j}}(y,\omega)|^2}\over{|x-y|^{3+2\kappa}}} dx dy d\omega.
\ee
We have an estimate
\bea\label{E1}
&
\n{{\p \psi{x_j}}}^2_{H^{2,(\kappa,0,0)}(G\times S)}
=
\int_{S}\int_G\int_G{{|{\p \psi{x_j}}(x,\omega)-{\p \psi{x_j}}(y,\omega)|^2}\over{|x-y|^{3+2\kappa}}} dx dy d\omega
\nonumber\\
\leq {}&
2\int_I\int_{S}\int_G\int_G{{|e^{-t(x,\omega)}{\p t{x_j}}(x,\omega)
-e^{-t(y,\omega)}{\p t{x_j}}(x,\omega)
|^2}\over{|x-y|^{3+2\kappa}}} dx dy d\omega
\nonumber\\
{}& +
2\int_I\int_{S}\int_G\int_G{{|e^{-t(y,\omega)}{\p t{x_j}}(x,\omega)
-e^{-t(y,\omega)}{\p t{x_j}}(y,\omega)
|^2}\over{|x-y|^{3+2\kappa}}} dx dy d\omega
=:2I_1+2I_2.
\eea
We conjecture that using the techniques of Example \ref{exreg} one might be able to show that the integral $I_1$ is converging for $s=3/2$ but the integral $I_2$ is diverging for $s=3/2$.
This will be studied in detail in a future work.
\end{example}

We find that the transport operator $T:=\omega\cdot\nabla_x+\Sigma-K$ is the sum of the first order partial differential operator and a partial integral operator.
Let 
\be
P(x,\omega,E,D)\psi:=\omega\cdot\nabla_x\psi
+\Sigma(x,\omega,E)\psi.
\ee
Then the equation $T\psi=f$ equivalent to
\be\label{i3-e}
P(x,\omega,E,D)\psi-K\psi =f.
\ee
In the case where $K=0$ the transport problem   is a boundary value problem for the first order partial differential equation
\be\label{i4-e}
P(x,\omega,E,D)\psi =f,\quad 
{\psi}_{|\Gamma_-}={ g}.
\ee
The existence and regularity results of this reduced problem mirror those of the complete problem.

The problem \eqref{i4-e} is not hyperbolic.
We have imposed the assumptions which imply that the operator $P(x,\omega,E,D)$ is {\it formally dissipative}.
Literature contains numerous contributions for the existence and regularity results
for general first order partial differential initial boundary value problems which are formally dissipative,
beginning from \cite{lax}, \cite{friedrich58}  and \cite{phillips66}.
More recent results can be found e.g. in  \cite{rauch74}, \cite{rauch85}, \cite{hormander85} (Chapter XXIII)), \cite{rauch94}, \cite{nishitani96}, \cite{nishitani98}, \cite{secchi}, \cite{takayama02}. 
In fact, the transport operator $P(x,\omega,E,D)$ in itself is not problematic.
The difficulties arise from the inflow boundary condition
${\psi}_{|\Gamma_-}={ g}$ which must be imposed when $G\neq\R^n$. 
This can be briefly explained as follows. 
Define the {\it boundary matrix}
\[
A_{\nu}(z)=\omega\cdot\nu(y), \quad z=(y,\omega,E)\in\Gamma
=(\partial G)\times S\times I.
\] 
The above mentioned references are partially valid only for
problems where the dimension of the  kernel ${\rm Ker}(A_\nu)$ is constant on $\Gamma$
(the so-called {\it constant multiplicity}). Hence we are not directly able to apply them since 
\[
{\rm Ker}(A_\nu)=\begin{cases} 0,\ &z\in \Gamma_{\pm}\\
\R,\ &z\in\Gamma_0  \end{cases}
\]
that is, the {\it transport problem is not of constant multiplicity}.
Some results for the problems with {\it variable multiplicity}   are also treated in the above references
(\cite{nishitani96}, \cite{nishitani98}, \cite{secchi}, \cite{takayama02})
but they require 
additional assumptions concerning the "transition with a non-zero derivative"
over the smooth $(n-2)$-dimensional manifold $\Gamma_0$.

It is known that for the general first order partial differential (initial) boundary value problems,
the mentioned transition assumption is needed even to guarantee
the unique existence of solutions, that is, to guarantee that the problem is well-posed.
The proofs of existence results are often based on the equivalence of weak and strong solutions 
(obtained by Friedrich's mollifier smoothing), however this equivalence does not hold in general.
As for transport problems, such as the one considered in this paper,
the transition assumption is not necessarily required for
well-posedness as we have verified in this paper (see also e.g. \cite{dautraylionsv6}, \cite{agoshkov} and \cite{tervo16-up}).

The Sobolev regularity of solutions in the context of general first order PDE-systems
have been treated in some of the above references as well.  
In the case when the boundary condition is of constant multiplicity, the so-called co-normal
(or tangential) Sobolev regularity can be achieved quite generally (\cite{rauch85}). 
The co-normal regularity results can not be generalized for problems with variable multiplicity.
Nevertheless, in these cases some (co-normal) regularity results in
{\it weighted Sobolev spaces} can be found (\cite{nishitani96}, \cite{nishitani98}, \cite{secchi}, \cite{takayama02}).
We shall study these issues more thoroughly in a future work.

%%%%%%%%%%%%%%%%%%%%%%%%%%%%%%%%%%%%%%%%%%%%%%%%%%%%%%%%%%%%%%%%%%%%%%%%
\section{Related (Optimal) Control Problem and Radiation Treatment Planning}\label{srs}
%%%%%%%%%%%%%%%%%%%%%%%%%%%%%%%%%%%%%%%%%%%%%%%%%%%%%%%%%%%%%%%%%%%%%%%%

%%%%%%%%%%%%%%%%%%%%%%%%%%%%%%%%%%%%%%%%%%%%%%%%%%%%%%%%%%%%%%%%%%%%%%%%
\subsection{Control Problem}\label{srss}
%%%%%%%%%%%%%%%%%%%%%%%%%%%%%%%%%%%%%%%%%%%%%%%%%%%%%%%%%%%%%%%%%%%%%%%%

In the following we let $p$ be in the interval $[1,\infty[$ (cf. Remark \ref{fr}). The most reasonable (and practical) choices are $p=2$ and $p=1$.
From the control theoretic point of view a relevant output mapping for the stationary problem in radiation therapy is the dose
(distribution)
\bea\label{sr1}
D(x):=(D\psi)(x)=\sum_{j=1}^3\int_{S\times I}\kappa_j (x,E)\psi_j(x,\omega,E) d\omega dE
\eea
where $\kappa_j\in L^\infty(G\times I),\ \kappa_j\geq 0$, are the so-called \emph{energy-deposition cross sections} (\cite{bomanthesis}, \cite{lorence97}).
We find that $D:L^p(G\times S\times I)^3\to L^p(G)$ is a bounded linear operator and
\be\label{Dbound}
\n{D\psi}_{L^p(G)}\leq m(G\times S\times I)^{1/p'}\Big(\max_{1\leq j\leq 3}\n{\kappa_j}_{L^\infty(G\times I)}\Big)
\n{\psi}_{L^p(G\times S\times I)^3}
\ee
where $m(G\times S\times I)$ is the measure of $G\times S\times I$ and ${1\over p}+{1\over{p'}}=1$,
with the convention that $m(G\times S\times I)^{1/p'}=1$, if $p=1$.

In the case of time-dependent problem the dose (distribution) is
\be\label{sr1a}
D(x,t):=D(\psi(t))(x)=\sum_{j=1}^3\int_{S\times I}\kappa_j (x,E)
\psi_j(x,\omega,E,t) d\omega dE
\ee
where $D(\cdot,\cdot)\in C([0,T], L^p(G))$ (or only in $L^1([0,T], L^p(G))$).
$D$ is an operator $C([0,T], L^p(G\times S\times I)^3)\to C([0,T], L^p(G))$,
and the total dose in time interval $[0,T]$ is given by
\[
D(x)=(D\psi)(x)=\int_0^TD(x,t)dt.
\]

%%%%%%%%%%%%%%%%%%%%%%%%%%%%%%%%%%%%%%%%%%%%%%%%%%%%%%%%%%%%%%%%%%%%%%%%
\subsubsection{Time-dependent Control System}
%%%%%%%%%%%%%%%%%%%%%%%%%%%%%%%%%%%%%%%%%%%%%%%%%%%%%%%%%%%%%%%%%%%%%%%%

The time-dependent BTE system in its abstract form for $\psi_0=0$ is 
\bea\label{ar2}
{\p {\psi}t}-\tilde{\bf A}\psi+\tilde\Sigma\psi -\tilde K\psi&=\tilde f,\\
 \psi_{|\Gamma_-\times ]0,T]}&=g
\eea
where we assume that $g(0)=0$ (the compatibility condition).
Using the lift $Lg:=(Lg_1,Lg_2,Lg_3)\in  C^2([0,T], {\s W}^p(G\times S\times I)^3)$
as obtained analogously to Lemma \ref{lift2}, 
and denoting $u=\psi-Lg$ the equation (\ref{ar2}) becomes (note that $u(t)\in \tilde {\s W}^p_{-,0}(G\times S\times I)=D(\tilde {\bf A}_0)$)
\begin{align}\label{ar3}
{\p {u}t}
&=({\tilde{\bf A}_0}-\tilde\Sigma  +\tilde K)u
-\Big({{\partial}\over{\partial t}}-\tilde{\bf A}+\tilde\Sigma
-\tilde{K}\Big)Lg+\tilde f \\
&:={\s A}_0u+{\s A}Lg -{\p {(Lg)}{t}}+\tilde f
\end{align}
where ${\s A}:=\tilde{\bf A}-\tilde\Sigma  +\tilde K$ and ${\s A}_0$ is its restriction to $D({\s A}_0):=D(\tilde {\bf A}_0)=\tilde {\s W}^p_{-,0}(G\times S\times I)^3$. 
Hence we have a relevant control system
\bea\label{ar4}
{\p {u}t}&={\s A}_0u+{\s A}Lg -{\p {(Lg)}t}+\tilde f,\\
u(0)&=0\\
y&=D(\psi(\cdot))=D(u(\cdot))+D(Lg(\cdot))
\eea
where we have ${\p {(Lg)}{t}}=L{\p {g}{t}}$ (see the proof of Lemma \ref{lift2}).

Let $T>0$. Define
\[
H_T(\tilde f,g):=u(T)=\int_0^T{\tilde G}(t-s)\Big({\s A}Lg -{\p {(Lg)}{t}}+\tilde f\Big)(s)ds
\] 
where ${\tilde G}(t)$ is (as above) the semigroup generated by ${\s A}_0$.
In external therapy we have $\tilde f=0$ and in internal (brachy) therapy $g=0$.

The important and relevant problems are the following ones. Let $T>0$ and let
${\s C}:=\{g\in C^2\big([0,T],{\s T}^p(\Gamma_-)^3\big)\ |\ g(0)=0,\ g\geq 0\}$
and ${\s C}'=\{\tilde f\in C^1\big([0,T], L^p(G\times S\times I)^3\big)\ |\ \tilde f\geq 0\}$. How to characterize the sets
\[
{\s S}_T:=\{H_T(0,g)=\psi(T)=u(T)+(Lg)(T)\ |\ g\in {\s C}\}\subset L^p(G\times S\times I)^3
\]
and
\[
{\s S}'_T:=\{H_T(\tilde f,0)=\psi(T)\ |\ \tilde f\in {\s C}'\}\subset L^p(G\times S\times I)^3
\]
that is,
\emph{what  are the possible states $\psi (T)$ that can be produced using the controls chosen from ${\s C}$ and ${\s C}'$,
respectively, during the time $T$}?
Similarly for the doses: What can be said about the sets
\[
{\s D}_T:=\{y=D(\psi(T))\ |\ g\in {\s C}\}\subset L^p(G).
\]
and
\[
{\s D}'_T:=\{y=D(\psi(T))\ |\ \tilde f\in {\s C}'\}\subset L^p(G).
\]
that is,
\emph{which are the dose distributions $D$ that one can produce using the controls chosen from ${\s C}$ and ${\s C}'$,
respectively, during the time $T$}?
Here ${\s C}$ may be replaced with a larger space such as with $\{g\in H^1(]0,T[,{\s T}^p(\Gamma_-)^3\ |\ g(0)=0,\ g\geq 0\}$.
Similarly, ${\s C}'$ can be replaced e.g. with the set $\{\tilde f\in H^1(]0,T[,L^p(G\times S\times I)^3)\ |\ \tilde f\geq 0\}$.

Similarly for the total doses we can impose the problem: What can be said about the sets
\[
{\s D}:=\{y=D\psi|\ g\in {\s C}\}\subset L^p(G)
\]
and 
\[
{\s D}':=\{y=D\psi|\ \tilde f\in {\s C}'\}\subset L^p(G)
\]
that is, {\it which are the total dose distributions that one can produce using controls chosen from ${\s C}$ and ${\s  C}'$ , respectively, during the time $T$} ?

It has been shown that in a certain simplified case for one particle there exists $T>0$ such that ${\s S}_T=L^p(G\times S\times I)$ for $p=2$ (\cite{acosta}), which means that the control system is {\it exactly controllable}. 
For a general background on infinite dimensional control systems and related concepts, we refer to \cite{curtainzwart}, \cite{tucsnak09}.
Note that in the above control problem, the time derivative of the control $g$  appears in the system 
which makes the problem more nonstandard.

\begin{remark}
The control system (\ref{ar4}) can be written in the form
\[
\dot{u}=\mc{A}_0u+\mc{B}v+\mc{C}\dot{v}
\]
where $v:=(\tilde f,g)$ (the control variable) and
\[
\mc{A}_0=&\tilde{\bf A}_0-\tilde\Sigma+\tilde K :L^p(G\times S\times I)^3 \to L^p(G\times S\times I)^3 \\
\mc{B}=&{\s A}\circ L\circ\pr_2+\pr_1 \\
&\quad :L^1(G\times S\times I)^3\times C^1([0,T],T^1(\Gamma_-))\to C([0,T],L^1(G\times S\times I)^3)\\
\mc{C}=&-L\circ \pr_2 \\
&\quad :L^1(G\times S\times I)^3\times C^1([0,T],T^1(\Gamma_-))\to C^1([0,T],\tilde{W}^1(G\times S\times I)).
\]
Here $\pr_1,\ \pr_2$ are projections onto the first and the second factors of the cartesian product space $L^1(G\times S\times I)^3\times C^1([0,T],T^1(\Gamma_-))$.
\end{remark}

%%%%%%%%%%%%%%%%%%%%%%%%%%%%%%%%%%%%%%%%%%%%%%%%%%%%%%%%%%%%%%%%%%%%%%%%
\subsubsection{Stationary Control Problem}
%%%%%%%%%%%%%%%%%%%%%%%%%%%%%%%%%%%%%%%%%%%%%%%%%%%%%%%%%%%%%%%%%%%%%%%%

The corresponding control problems can also be stated for the stationary problem which we sketch as follows. The forward problem is
\bea\label{ar22}
-{\bf A}\psi+\Sigma\psi - K\psi &= f,\nonumber\\
\psi_{|\Gamma_-} &=g,
\eea
where $f\in L^p(G\times S\times I)^3,\ g\in T^p(\Gamma_-)$. 
Using the  lift $Lg:=(Lg_1,Lg_2,Lg_3)\in \tilde {W}^p(G\times S\times I)^3$
and denoting $u=\psi-Lg$ the system (\ref{ar22}) becomes
\bea\label{ar32}
(-{\bf A}_0+\Sigma  - K)u=f-(-{\bf A}+\Sigma  - K)Lg
\eea
where $u\in \tilde W_{-,0}^p(G\times S\times I)^3=D({\bf A}_0)$. Hence we have (under the assumptions of Theorem \ref{coth2})
\bea\label{ar42}
u=&(-{\bf A}_0+\Sigma  - K)^{-1}\big(f-(-{\bf A}+\Sigma  - K)Lg\big)\nonumber\\
y=&D\psi=Du+D(Lg)=D\Big[(-{\bf A}_0+\Sigma  - K)^{-1}\big(f-(-{\bf A}+\Sigma  - K)Lg\big)\Big]+
D(Lg). 
\eea

Again, the relevant problems in this stationary case are the following ones. How to characterize the sets (external therapy)
\[
{\s S}:=\{\psi=u+Lg=-(-{\bf A}_0+\Sigma  - K)^{-1}(-{\bf A}+\Sigma  - K)Lg+Lg\ |\ g\in T^p(\Gamma_-)^3,\ g\geq 0\}
\]
and (internal therapy)
\[
{\s S}':=\{\psi=(-{\bf A}_0+\Sigma  - K)^{-1} f\ |\ f\in L^p(G\times S\times I)^3,\ f\geq 0\}.
\]
For the dose distributions, it is important to
describe the structures of the sets
\[
{\s  D}:=\{D\psi\ |\ \psi\in {\s S}\}
\]
and
\[
{\s  D}':=\{D\psi\ |\ \psi\in {\s S}'\}.
\]

In addition to the above problems, one of the main challenges in radiation therapy is to develop methods on how the inflow flux $g$ and/or the internal source $f$ can be computed when an element of ${\s S}$ (resp. ${\s S}_T$), and especially when an element of ${\s D}$ (resp. ${\s D}_T$), is know and the same is concerning for the sets ${\s S}',\ {\s S}_T'\ {\s D}',\ {\s D}_T'$ . This is known as the \emph{Inverse Planning Problem} (see e.g. \cite{bomanthesis}, \cite{shepard99}) which we, from the mathematical point of view, shall describe briefly below. In some simple cases the inverse problem can probably be solved analytically (cf. \cite{mokhtarkharroubi}, Chapter 11), but in real situations only \emph{optimal inflow fluxes or internal sources} can be found. The analytical solutions are valuable (even though in simplified cases) because they may greatly help the actual optimization procedure (e.g. in seeking the initial point
for the global optimization) and give insight on which kind of states or dose distributions are reasonable and possible to generate.

In some cases  for $p=2$ the related stationary \emph{optimal control} problem has, in theory, an
explicit  solution. This is based on the optimal control theory for equations governed by closed densely defined coercive operators in Hilbert spaces and the convexity of the objective function and admissible sets (see section \ref{cis} and \cite{frank08}, \cite{tervo07}). 
The explicit solution obtained in this way
 can be used as an initial solution for the chosen optimization algorithm but it is not generally a ready treatment plan.
The same observations remain valid for the time-dependent case. We emphasize that time-dependent models are still not (at least extensively) applied in radiation therapy.
For a more extensive background of optimal control problems governed by partial differential/boundary value operators we refer to \cite{lions71}.

%%%%%%%%%%%%%%%%%%%%%%%%%%%%%%%%%%%%%%%%%%%%%%%%%%%%%%%%%%%%%%%%%%%%%%%%
\subsection{Radiation Treatment Planning}\label{RTP}
%%%%%%%%%%%%%%%%%%%%%%%%%%%%%%%%%%%%%%%%%%%%%%%%%%%%%%%%%%%%%%%%%%%%%%%%

We consider here only the treatment planning based on stationary Boltzmann transport equation.
Not all the configurations are achievable as dose distributions,
in the sense that  in general ${\s D}\neq L^p(G)$,
and therefore hence one can only hope to seek
dose distributions which are as optimal as possible
with respect to the given  configuration.

%%%%%%%%%%%%%%%%%%%%%%%%%%%%%%%%%%%%%%%%%%%%%%%%%%%%%%%%%%%%%%%%%%%%%%%%
\subsubsection{Background}\label{RTP1}
%%%%%%%%%%%%%%%%%%%%%%%%%%%%%%%%%%%%%%%%%%%%%%%%%%%%%%%%%%%%%%%%%%%%%%%%

As mentioned in the introduction, radiation therapy aims to generate dose distributions in such a way that the desired dose conforms to the target volume, while the healthy tissue and especially the so-called critical organs achieve as low dose as possible. 
Dose can be delivered externally (external therapy) or internally (internal therapy or brachytherapy).
The determination of (optimal) incoming external particle fluxes through the patches of patient surface or internal sources located inside the patient tissue is the basic task in treatment planning known as the {\it inverse treatment planning. }

Recall that the patient domain $G\subset\R^3$ consists of the tumor volume ${\bf T}$,
the critical organ region ${\bf C}$ and the normal tissue region ${\bf N}$,
as a mutually disjoint union $G={\bf T}\cup {\bf C}\cup {\bf N}$. We assume that the sets
${\bf T},\ {\bf C},\ {\bf N}$ are Lebesgue measurable.
The tumor volume, that is the target, includes the tumor and some safety margin around it.
Critical organs and normal tissue are build up of healthy tissue,
and must be conserved during the treatment as well as possible.

In the sequel we only deal with the inverse treatment planning problem in the context
of the stationary Boltzmann transport equation (i.e. we omit time dependency here which could be treated analogously after some modifications).
The dose is computed from the generated particle flux $\psi\in L^p(G\times S\times I)^3$
(as mentioned above) by
\[
D(x)=(D\psi)(x)=\sum_{j=1}^3\int_{S\times I}\kappa_j(x,E)\psi_j(x,\omega,E)d\omega dE
\] 
where $\psi=(\psi_1,\psi_2,\psi_3)\in L^p(G\times S\times I)^3\subset L^1(G\times S\times I)^3$ satisfies
the Boltzmann transport equation,
\bea\label{bte}
\omega\cdot\nabla\psi_j+\Sigma\psi_j-K\psi=&f_j \\
{\psi_j}_{|\Gamma_-}=&g_j, \nonumber
\eea
for $j=1,\ 2,\ 3$, or, more shortly,
\bea\label{bte1}
(-{\bf A}_0+\Sigma-K)\psi&=f, \\
\psi_{|\Gamma_-}&=g. \nonumber
\eea
In practice $g$ is non-zero only on a finite number of patches on patient's surface.
Let $\psi=\psi(f,g)$ be the unique solution of (\ref{bte1}) that is, as discussed in section \ref{coss3} (under the relevant assumptions stated there)
\bea\label{sol}
\psi=\psi(f,g)= (-{\bf A}_0+\Sigma-K)^{-1}\big(f-(-{\bf A}+\Sigma-K)Lg\big)+Lg.
\eea
The generated dose is then
\[
D=D(x)=\big(D(\psi(f,g))\big)(x),\quad x\in G.
\]
We denote ${\s D}(f,g):=D(\psi(f,g))$. Then we have 

\begin{lemma}\label{dc}
The dose operator ${\s D}:L^p(G\times S\times I)^3\times T^p(\Gamma_-)^3\to L^p(G)$ is linear and bounded, i.e.
\[
\n{{\s D}(f,g)}_{L^p(G)}\leq C\ (\n{f}_{L^p(G\times S\times I)^3}+\n{g}_{ T^p(\Gamma_-)^3}).
\]
\end{lemma}

\begin{proof}
The following proof is complete only for $p=1$ and $p=2$ because (\ref{co2b}) (and its consequence (\ref{bestim})) has been shown only for these cases (see Theorems \ref{dfsco} and \ref{coth2} above, and \cite{tervo07}).
By (\ref{ar42}) 
\[
{\s D}(f,g)=
D\Big[(-{\bf A}_0+\Sigma  - K)^{-1}(f-(-{\bf A}+\Sigma  - K)Lg)\Big]+D(Lg).
\]
Hence ${\s D}$ is a linear operator. We show that it is bounded.
Writing $C_0:=m(G\times S\times I)^{1/p'}\Big(\max_{1\leq j\leq 3}\n{\kappa_j}_{L^\infty(G\times I)}\Big)$, we have by (\ref{Dbound}),
\[
& \n{D\psi}_{L^p(G)}\leq C_0
\n{(-{\bf A}_0+\Sigma  - K)^{-1}(f-(-{\bf A}+\Sigma  - K)Lg)+Lg
}_{L^p(G\times S\times I)^3}, \\
& \n{D(Lg)}_{L^p(G)}\leq C_0\n{Lg}_{L^p(G\times S\times I)^3}.
\]
Then by (\ref{bestim}), (\ref{li0}), Lemma \ref{lift1} and since ${\bf A}(Lg)=0$
(as shown in the proof of the lemma)
\bea
&\n{D\psi}_{L^p(G)}\leq
C_0\Big(\n{(-{\bf A}_0+\Sigma  - K)^{-1}f}_{L^p(G\times S\times I)^3}\nonumber\\
&+\n{
(-{\bf A}_0+\Sigma  - K)^{-1}(-{\bf A}+\Sigma  - K)Lg}_{L^p(G\times S\times I)^3}+
\n{Lg}_{L^p(G\times S\times I)^3}\Big)\nonumber\\
&\leq
C_0\Big({1\over c}\n{f}_{L^p(G\times S\times I)^3}
+{1\over c}\n{(-{\bf A}+\Sigma  - K)Lg}_{L^p(G\times S\times I)^3}+
\n{Lg}_{L^p(G\times S\times I)^3}\Big)\nonumber\\
&\leq
C_0\Big({1\over c}\n{f}_{L^p(G\times S\times I)^3}
+{d\over c}\n{\Sigma-K}\n{g}_{T^p(\Gamma_-)^3}+
d\n{g}_{T^p(\Gamma_-)^3}\Big),
\eea
which implies the boundedness of ${\s D}$.
\end{proof}

Commonly used {\it physical criteria}
are the following ones. We demand that
\begin{align}
&D(x)=D_0,\ x\in {\bf T},\label{tc}\\
&D(x)\leq D_C,\ x\in {\bf C},\label{cc}\\  
&D(x)\leq D_N,\ x\in {\bf N},\label{nc}
\end{align}
where $D_0$ is the prescribed (usually uniform) dose in target ${\bf T}$ and where $D_C$ and $D_N$ are the allowed upper bounds in the critical organ ${\bf C}$ and normal tissue ${\bf N}$ regions, respectively. Instead of (\ref{tc}) one may ask for more flexibly (when considering the so-called feasible solutions) that only
\bea
d_T\leq D(x)\leq D_T,\ x\in {\bf T},\label{tc1}
\eea
where $D_T$ and $d_T$ are upper and lower bounds for dose in target.

In addition to the above requirements in modern planning, one imposes so-called {\it dose volume constraints} for the dose distribution, especially for the critical organ region (but also for some other tissue region's similar dose volume constraints may be considered). Dose volume constraint demands that the dose $D(x)$ cannot be greater than some prescribed dose level, say $d_C$, in a volume fraction $v_C$
of ${\bf C}$ which is greater
than some given fraction $v_C$. This can be expressed as follows
\bea\label{dv}
{{\mc{L}^3(\{x\in {\bf C}|\ D(x)\geq d_C\})}\over{\mc{L}^3({\bf C})}}\leq v_C
\eea
where $\mc{L}^3$ is the 3-dimensional Lebesgue measure.
Clearly the dose volume constraint is equivalent to
\be\label{dva}
{1\over{\mc{L}^3({\bf C})}}\int_{\bf C}H(D(x)-d_C) dx\leq v_C
\ee
where $H$ is the Heaviside function. Note that the integral in (\ref{dva}) exists.
In practice $H$ can be replaced here with a smooth (or continuous)
function $H_\epsilon$ which approximates it to some reasonable level of accuracy.

%%%%%%%%%%%%%%%%%%%%%%%%%%%%%%%%%%%%%%%%%%%%%%%%%%%%%%%%%%%%%%%%%%%%%%%%
\subsubsection{Object Function and the Optimization Problem}\label{objectf}
%%%%%%%%%%%%%%%%%%%%%%%%%%%%%%%%%%%%%%%%%%%%%%%%%%%%%%%%%%%%%%%%%%%%%%%%
   
Our aim is that the above requirements \eqref{tc}, \eqref{cc}, \eqref{nc}, \eqref{dv} for the dose distribution are valid as well as possible.
For that purpose, we define the object (cost) function
\bea\label{of}
J(f,g)=c_{\bf T} J_{\bf T}(f,g)+c_{\bf C} J_{\bf C}(f,g)+c_{\bf N} J_{\bf N}(f,g)+c_{\rm DV} J_{\rm DV}(f,g),
\eea 
where
\[
J_{\bf T}(f,g)&=\n{D_0-{\s D}(f,g)}_{L^p({\bf T})}^p,\\
J_{\bf C}(f,g)&=\n{(D_C-{\s D}(f,g))_-}_{L^p({\bf C})}^p,\\
J_{\bf N}(f,g)&=\n{(D_N-{\s D}(f,g))_-}_{L^p({\bf N})}^p,\\
J_{\rm DV}(f,g)&=
\Big(\big(v_C-{1\over{\mc{L}^3({\bf C})}}\int_{\bf C}H({\s D}(f,g)(x)-d_C) dx\big)_-\Big)^p,
\]
and where $c_{\bf T}, c_{\bf C}, c_{\bf N}, c_{\rm DV}$ are non-negative weights
with which one controls the different 
priorities in the optimization. Here $a_-$ denotes the negative part $\frac{1}{2}(|a|-a)$ of $a\in\R$.
Also, notice that $J_{\rm DV}(f,g)\leq 1$.

As mentioned above in external radiotherapy $f=0$ and in internal radiotherapy $g=0$
(from the mathematical point of view both can, of course, be non-zero). Thus the corresponding object functions in practice are
\[
J_{\rm ex}(g):=&J(0,g)\ \ {\rm (external\ radiotherapy})\ {\rm and}\\
J_{\rm in}(f):=&J(f,0)\ \ {\rm (internal\ radiotherapy)}.
\]
The {\it admissible sets}
for the optimal control problems  are respectively
\[
U_{\rm ad}=\{g\in T^p(\Gamma_-)^3\ |\ g\geq 0\}\ ({\rm external\ radiotherapy})
\]
and
\[
U_{\rm ad}'=\{f\in L^p(G\times S\times I))^3\ |\ f\geq 0\}\ ({\rm internal\ radiotherapy}).
\]  
They both are convex sets (cones) of the ambient spaces. If (in practical optimization) one wants to take the whole ambient space as an admissible set that is, $U_{\rm ad}=T^p(\Gamma_-)^3$ or
respectively $U_{\rm ad}'=L^p(G\times S\times I)^3$
one must add to the object function the penalty term
\[
+c_{\rm ad} J_{\rm ad}(g),
\ {\rm where},\ J_{\rm ad}(g)=\n{g_-}^p_{T^p(\Gamma_-)^3}
\ {(\rm external\ radiotherapy)}
\]
and
\[
+c_{\rm ad'} J_{\rm ad'}(f),
\ {\rm where}\ J_{\rm ad'}(f)=\n{f_-}^p_{L^p(G\times S\times I)^3}
\ {\rm (internal\ radiotherapy)}.
\]
These take care of the non-negativity of the incoming flux or source, respectively.
In theory as well as in practice, it is also reasonable to add a stabilizing cost terms
correspondingly
\be\label{ofsc1}
+c_{\rm sc} J_{\rm sc}(g)\ {\rm where}\ J_{\rm sc}(g)=\n{\psi(0,g)}^p_{L^p(G\times S\times I)^3},
\ee
or
\be\label{ofsc2}
+c_{\rm sc'} J_{\rm sc'}(f)\ {\rm where}\ J_{\rm sc'}(f)=\n{\psi(f,0)}^p_{L^p(G\times S\times I)^3}.
\ee

As a conclusion, we have for the external therapy the object function
\bea\label{fof}
J_{\rm ex}(g)=c_{\bf T} J_{\bf T}(0,g)+c_{\bf C} J_{\bf C}(0,g)+c_{\bf N}J_{\bf N}(0,g)
+c_{\rm DV} J_{\rm DV}(0,g)+c_{\rm sc}J_{\rm sc}(g),
\eea
when $U_{\rm ad}=\{g\in T^p(\Gamma_-)^3|\ g\geq 0\}$ or
\bea\label{fofa}
J_{\rm ex}(g)=c_{\bf T}J_{\bf T}(0,g)+c_{\bf C}J_{\bf C}(0,g)+c_{\bf N}J_{\bf N}(0,g)
+c_{\rm DV}J_{\rm DV}(0,g)+c_{\rm ad} J_{\rm ad}(g)+c_{\rm sc}J_{\rm sc}(g).
\eea
when $U_{\rm ad}=T^p(\Gamma_-)^3$.
The object function $J_{\rm in}(f)$ for the internal therapy is formulated analogously.
In practice one may have $p=1$, which is from the physical point of view
a reasonable choice,
or $p=2$, which gives mathematically a very convenient setting because of Hilbert space
structure of the underlying spaces.

With these concepts the overall optimal control (boundary value problem) can be stated as
: Find the {\it global minimum}
\bea\label{taske}
\min\{J_{\rm ex}(g)\ |\ g\in U_{\rm ad}\}\quad {(\rm external\ radiotherapy)},
\eea
or
\bea\label{taski}
\min\{J_{\rm in}(f)\ |\ f\in U_{\rm ad}\}\quad {\rm (internal\ radiotherapy)},
\eea
where $U_{\rm ad}$ or $U'_{\rm ad}$, respectively, is chosen in the way explained above.

Suppose that $X$ is a vector space and that $F:U\to \R$ is a function defined on a convex set $U\subset X$. We say that $F$ is {\it convex} if for any choice of $x,y\in U$ it holds that
\[
F(tx+(1-t)y)\leq tF(x)+(1-t)F(y),\quad 0\leq t\leq 1.
\]
$F$ is called strictly convex if for $x,y\in U, \ x\neq y$
\[
F(tx+(1-t)y)<tF(x)+(1-t)F(y),\quad 0<t<1.
\]

Let $F_-:\R\to \R$ be the negative part function $F_-(x)={1\over 2}(|x|-x)$. We find that it is  non-differentiable (at $x=0$). Hence in general case (one can show) the object functions $J_{\rm ex}$ 
and $J_{\rm in}$ are also  non-differentiable in the (interior of the) corresponding admissible sets.  However, we have for $J_{\rm ex}$ and $J_{\rm in}$ the following.

\begin{theorem}\label{pof}
\begin{itemize}
\item[(i)] The terms
\[
J_{\bf T}(0,g),\ J_{\bf C}(0,g),\ J_{\bf N}(0,g),\ J_{\rm ad}(g)\ {\rm and}\ J_{\rm sc}(g)
\]
of the object function $J_{\rm ex}: T^p(\Gamma_-)^3\to \R$ are convex,
and they are locally (resp. globally) Lipschitz continuous if $p\in ]1,\infty[$ (resp. $p=1$).
In addition, the term
\[
J_{\rm DV}(0,g)
\]
is Lipschitz continuous, if the Heaviside function $H$ in its definition is
replaced by a Lipschitz continuous approximation $H_\epsilon$ (see (\ref{dca}).

\item[(ii)] When $p=2$, the terms
\[
J_{\bf T}(0,g),\ J_{\rm sc}(g)
\]
of the object function $J_{\rm ex}$ are differentiable on $T^2(\Gamma_-)^3$.
\end{itemize}

Analogous results hold for the terms of the object function $J_{\rm in}$.
\end{theorem}

\begin{proof}
(i) We first show the stated convexity properties.

Recall that ${\s D}:L^p(G\times S\times I)^3\times T^p(\Gamma_-)^3\to L^p(G)$ is linear and bounded, and hence the mapping $g\to D_0-{\s D}(0,g)$ is affine.
Moreover, it is a basic fact that the map $\n{\cdot}^p_{L^p}:L^p({\bf T})\to\R$; $u\mapsto \n{u}_{L^p({\bf T})}^p$
is strictly convex if $p\in ]1,\infty[$ and convex if $p=1$
(the point being that $\R\to\R$; $x\mapsto |x|^p$ is strictly convex, or convex,
in the respective two situations).
Therefore, as the composition of these two maps $J_{\bf T}(0,g)$ is clearly
convex if $p\in [1,\infty[$.

To see that $J_{\bf C}(0,g)$, $J_{\bf N}(0,g)$, $J_{\rm ad}(g)$
are convex,
it is enough to observe that $g\to D_0-{\s D}(0,g)$ is affine,
the negative part function $x\mapsto x_-=\frac{1}{2}(|x|-x)$ is convex,
the map $x\mapsto x^p$ for $p\geq 1$ is increasing for $x\geq 0$
and that the integral $\int L^p(X)\to\R$; $u\mapsto \int_X u$ is linear, where $X$ is one of the sets ${\bf C}$, ${\bf N}$ or $\Gamma_-$.

Finally, $J_{\rm sc}(g)$ is  convex if $p\in [1,\infty[$,
as the mapping $g\mapsto \psi(0,g)$ is linear
and $u\mapsto \n{u}_{L^p(G\times S\times I)}$ is convex,
in the corresponding cases.

We then move to showing the claims concerning the Lipschitz continuities of the terms of $J_{\rm ex}$.

That $J_{\bf T}(0,g)$, $J_{\bf C}(0,g)$, $J_{\bf N}(0,g)$, $J_{\rm ad}(g)$ and $J_{\rm sc}(g)$
are Lipschitz continuous (locally or globally), can be seen as follows.

By the proof of Lemma \ref{dc} the operator 
$T^p(\Gamma_-)^3\to L^p(G\times S\times I)^3$; $g\mapsto \psi(0,g)$,
whose value can be written as (see \eqref{sol})
\[
\psi(0,g)=-(-{\bf A}_0+\Sigma  - K)^{-1}(-{\bf A}+\Sigma  - K)Lg)+Lg
\]
is linear and bounded, hence globally Lipschitz.
Similarly, as the map $T^p(\Gamma_-)^3\to L^p(G)$; ${\s D}g={\s D}(0,g)$
is linear and bounded, for any $D\in\R$ the affine map $g\mapsto D-{\s D}(0,g)$
is globally Lipschitz.

The negative part map $\R\to\R$; $x\mapsto \frac{1}{2}(|x|-x)$ is globally Lipschitz,
as is the norm map $u\mapsto \n{u}_{L^p(X)}$ for any $p\geq 1$, where $X$ here is one of the sets ${\bf T}$, ${\bf C}$, ${\bf N}$ or $\Gamma_-$.
Finally $x\mapsto x^p$, $x\geq 0$, is locally Lipschitz if $p>1$ and globally Lipschitz if $p=1$.
Therefore, because $J_{\bf T}(0,g)$, $J_{\bf C}(0,g)$, $J_{\bf N}(0,g)$, $J_{\rm ad}(g)$
are appropriate composition maps of the above ones,
we see that they all are locally Lipschitz when $p\in ]1,\infty[$, and globally Lipschitz when $p=1$.

Finally, noticing that $|x^p-y^p|\leq p|x-y|$ for all $x,y\in [0,1]$,
and that
\[
0\leq \Big(v_C-{1\over{\mc{L}^3({\bf C})}}\int_{\bf C}H_\epsilon\big({\s D}(0,g)(x)-d_C\big) dx\Big)_-\leq 1
\]
the claimed Lipschitz continuity of $J_{\rm DV}(0,g)$ can established.

(ii) For $p=2$, as $g\mapsto \psi(0,g), \mc{D}(0,g)$ are continuous (hence smooth) linear maps,
and the squared norms $\n{\cdot}^2_{L^2({\bf T})}$, $\n{\cdot}_{L^2(G\times S\times I)^3}^2$
used in the definition of terms 
  $J_{\bf T}(0,g)$ and $J_{\rm sc}(g)$, respectively, are smooth,
we see that these latter two maps are smooth.

Finally, we content ourselves with noticing that the asserted properties for the object function $J_{\rm in}$ are  proven similarly as above, and so the proof is complete.
\end{proof}

\begin{remark}
A. The term $g\to J_{\rm DV}(0,g)$ is not convex in $U_{\rm ad}$. This is due to the fact that the Heaviside function is not convex (or concave). 
Similarly the term $f\to J_{\rm DV}(f,0)$ is not convex in $U_{\rm ad}'$.

B. If we used the term $\n{D_0-{\s D}(f,g)}_{L^p({\bf T})}$ instead of $\n{D_0-{\s D}(f,g)}_{L^p({\bf T})}^p$
and similarly for the other terms of the objective function(s),
we would get in part (i) of the preceding theorem \ref{pof}
a (globally) Lipschitz continuous terms of the object function when $H$ is replaced with $H_\epsilon$.
\end{remark}

%%%%%%%%%%%%%%%%%%%%%%%%%%%%%%%%%%%%%%%%%%%%%%%%%%%%%%%%%%%%%%%%%%%%%%%%
\subsubsection{Computation of Initial Solutions}\label{cis}
%%%%%%%%%%%%%%%%%%%%%%%%%%%%%%%%%%%%%%%%%%%%%%%%%%%%%%%%%%%%%%%%%%%%%%%%

One possibility to help the optimization process is to compute the initial solution for actual (global) optimization method as accurately and rapidly as possible.
We suggest the following approach for $p=2$, and consider its details only
in the case of the external radiotherapy.
The computations for internal radiotherapy (formulated below) are analogous
and thus omitted.

The spaces $W^2(G\times S\times I)$, $T^2(\Gamma)$ and $\tilde{W}^2(G\times S\times I)$ have Hilbert space structures, the inner products being respectively the following ones
\[
\la\psi,v\ra_{W^2(G\times S\times I)}=\la\psi,v\ra_{L^2(G\times S\times I)} +
\la\omega\cdot\nabla\psi,\omega\cdot \nabla v\ra_{L^2(G\times S\times I)}
\]
\[
\la h,g\ra_{T^2(\Gamma)}=\la h,g\ra_{L^2(\Gamma,|\omega\cdot\nu|d\sigma d\omega dE)}
\]
and
\[
\la\psi,v\ra_{\tilde W^2(G\times S\times I)}=\la\psi,v\ra_{W^2(G\times S\times I)} +
\la \gamma(\psi),\gamma(v)\ra_{T^2(\Gamma)}.
\]
In the product space $W^2(G\times S\times I)^3$ we use, as before, the inner product
\[
\la\psi,v\ra_{W^2(G\times S\times I)^3}=\sum_{j=1}^3\la\psi_j,v_j\ra_{W^2(G\times S\times I)},
\]
where $\psi=(\psi_1,\psi_2,\psi_3)$, $v=(v_1,v_2,v_3)\in {W^2(G\times S\times I)^3}$,
and similarly for other product spaces.

As is standard when $p=2$, the formulas related to the
optimal control system in the context of the transport equation
are written below by using the relevant \emph{variational equations} (based on the Green formula (\ref{green})).
For example, the finite element method (FEM) schemes can be naturally implemented applying this formulation.

Let $f\in L^2(G\times S\times I)^3$ and $g\in T^2(\Gamma_-)^3$. Then
$\psi\in \tilde W^2(G\times S\times I)^3$ is a solution of the problem
\be\label{va0}
(-{\bf A}+\Sigma-K)\psi&=f,\\
\psi_{|\Gamma_-}&=g \nonumber
\ee
if and only if
\be\label{va1}
B(\psi,v)=F(v),\quad \forall v\in \tilde W^2(G\times S\times I).
\ee
where $B(\cdot,\cdot)$ is a bilinear form given by 
\be\label{va2}
B(\psi,v)=&-\la\psi,\omega\cdot\nabla v\ra_{L^2(G\times S\times I)^3}+
\la\psi,(\Sigma^*-K^*)v\ra_{L^2(G\times S\times I)} \\
& +\sum_{j=1}^3\int_{\partial G\times S\times I}(\omega\cdot \nu)_+\psi_jv_j d\sigma d\omega dE, \quad \textrm{for}\ \psi, v\in \tilde W^2(G\times S\times I)^3,\nonumber
\ee
and where $F$
is a linear form
\be\label{va3}
F(v)=\la f,v\ra_{L^2(G\times S\times I)^3}
+\sum_{j=1}^3\int_{\partial G\times S\times I} (\omega\cdot \nu)_-g_j v_j d\sigma d\omega dE,\ v\in \tilde W^2(G\times S\times I)^3
\ee
where $\Sigma^* = \Sigma$ and $K^*\psi^*=(K_1^*\psi^*,K_2^*\psi^*,K_3^*\psi^*)$ with
\[
(K_j^* \psi^*)(x,\omega,E)
=\sum_{k=1}^3\int_{S\times I}\sigma_{jk}(x,\omega,\omega',E,E')\psi^*_k(x,\omega',E')d\omega' dE',\quad j=1,2,3.
\]
As above $(\omega\cdot \nu)_{\pm}$ are the positive and negative parts of $\omega\cdot \nu$,
respectively.
It is to be pointed out that the bilinear form $B$ is not symmetric.

The existence result of the solution to the problem (\ref{va0})
in the variational form (when $p=2$) corresponding to the Theorem \ref{coth3}
(where $p=1$) is the following.

\begin{theorem}\label{th:LM}
Under the assumptions \eqref{scateh}, \eqref{colleh},  \eqref{colleha}, \eqref{co2a}, \eqref{co2aa},
for all $f\in L^2(G\times S\times I)^3$ and all $g\in T^2(\Gamma_-)^3$,
there exists an unique $\psi\in \tilde W^2(G\times S\times I)^3$ such that
(\ref{va1}) is holds.
\end{theorem}

\begin{proof}
The proof is an application of standard Hilbert space methods (e.g. using the Lions-Lax-Milgram theorem), and hence omitted here (see \cite{tervo07}).

Alternatively, one can adapt the proofs of Theorems \ref{coth2} and \ref{coth3}
to the case $p=2$, using the additional assumption \eqref{colleha},  \eqref{co2aa} made above.
\end{proof}

We formulate the corresponding {\it adjoint problem} in the variational form.
Define a 
bilinear form $B^*(\cdot,\cdot)$ by 
\be\label{va5}
B^*(\psi^*,v)=&\la\psi^*,\omega\cdot\nabla v\ra_{L^2(G\times S\times I)^3}
+\la\psi^*,(\Sigma-K)v\ra_{L^2(G\times S\times I)^3} \\
&+\sum_{j=1}^3\int_{\partial G\times S\times I} (\omega\cdot \nu)_-\psi_jv_j d\sigma d\omega dE\ {\rm for}\ \psi^*,\ v\in \tilde{W}^2(G\times S\times I)^3,
\nonumber
\ee
and, for a given $f^*\in L^2(G\times S\times I)^3$ and $g^*\in T^2(\Gamma_+)^3$,
a linear form by
\be\label{va6}
F^*(v)=\la f^*,v\ra_{L^2(G\times S\times I)^3}
+\sum_{j=1}^3\int_{\partial G\times S\times I}(\omega\cdot \nu)_+g_j^*v_j d\sigma d\omega dE,\ v\in \tilde{W}^2(G\times S\times I)^3.
\ee

For the adjoint problem, we have existence and uniqueness result
similar to the above theorem.

\begin{theorem}\label{th:LM-ad}
Assuming that \eqref{scateh}, \eqref{colleh}, \eqref{colleha}, \eqref{co2a}, \eqref{co2aa} hold,
then for every $f^*\in L^2(G\times S\times I)^3$ and $g^*\in T^2(\Gamma_+)^3$,
there exists a unique $\psi^*\in \tilde W^2(G\times S\times I)^3$ such that
\be\label{va4}
B^*(\psi^*,v)=F^*(v),\quad \forall v\in \tilde W^2(G\times S\times I).
\ee
\end{theorem}

\begin{proof}
We omit the proof, which is essentially based on Lions-Lax-Milgram Theorem (as the proof of Theorem \ref{th:LM}); see \cite{grisvard}, Lemma 4.4.4.1, p. 234.
\end{proof}

The equation \eqref{va4} is the variational form of the adjoint problem
\be\label{va7}
({\bf A}^*+\Sigma^*-K^*)\psi^*=&f^*,\\
\psi^*_{|\Gamma_+}=&g^*. \nonumber
\ee
From the operator theoretical point of view, the above existence result for adjoint problem is based on the fact that for a densely defined closed operator $A:X\to X$
in Hilbert space $X$ whose range $R(A^*)$ is closed in $X$, one has $R(A^*)=N(A)^{\perp}$ and $N(A^*)=R(A)^{\perp}$
(cf. \cite{tucsnak09}, Section 2.8).
In (\ref{va7}) 
\[
{\bf A}^*\psi^*=(\omega\cdot\nabla\psi_1^*,\omega\cdot\nabla\psi_2^*,\omega\cdot\nabla\psi_3^*,),\ \psi^*\in D({\bf A}^*):=W^2(G\times S\times I)^3
\]
where $\Sigma^*=\Sigma$ and $K^*$ is as above.
By the Green's formula (\ref{green}) we see that 
\be\label{va8}
B(\psi,\psi^*)=B^*(\psi^*,\psi),\quad {\rm for\ all}\ \psi,\psi^*\in \tilde W^2(G\times S\times I)
\ee

Recall that the flux to dose operator $D:L^2(G\times S\times I)^3\to L^2(G)$ is
\[
(D\psi)(x,\omega,E)=\sum_{j=1}^3\int_{S\times I} \kappa_j(x,E)\psi_j(x,\omega,E) d\omega dE,
\]
and that $D$ is bounded. We find that its adjoint operator $D^*:L^2(G)\to L^2(G\times S\times I)^3$ is
\be\label{v8a}
D^*d=(\kappa_1,\kappa_2,\kappa_3)d, \quad {\rm for}\ d\in L^2(G).
\ee

In external radiotherapy we have $\psi=\psi(0,g)$ and ${\s D}(0,g)=D(\psi(0,g))$.
We shall denote $\psi(g)=\psi(0,g)$ and ${\s D}(g)={\s D}(0,g)$.
Suppose that $d_{\bf T}\in L^2({\bf T}),\ d_{\bf C}\in L^2({\bf C}),\ d_{\bf N}\in L^2({\bf N})$ are
some given dose distributions (for example, they may be constants). We define an object function
\be\label{va9}
J(g)=&c_{\bf T}\n{ d_{\bf T}-{\s D}(g)}_{L^2({\bf T})}^2+c_{\bf C}\n{ d_{\bf C}-{\s D}(g)}_{L^2({\bf C})}^2 \\
&+c_{\bf N}\n{d_{\bf N}-{\s D}(g)}_{L^2({\bf N})}^2+c\n{g}_{T^2(\Gamma_-)^3}^2, \nonumber
\ee
with strictly positive constants $c_{\bf T}, c_{\bf C}, c_{\bf N}, c>0$.
In practice, only one type of particles are inflowing simultaneously (usually photons or electrons) but we can formulate a more general result which includes this
more realistic situation. We denote (here the source $f=0$)
\[
F(v)=(Fg)(v)
:=\sum_{j=1}^3\int_{\partial G\times S\times I}(\omega\cdot \nu)_-g_jv_j d\sigma d\omega dE
=\la g,\gamma_-(v)\ra_{T^2(\Gamma_-)^3}
\]
and, as before,
\[
U_{\rm ad}=\{g\in T^2(\Gamma_-)^3\ |\ g\geq 0\}.
\] 

In what follows, we shall write $b_+=\max\{0,b\}$
for the positive part of $b\in\R$, and $a_+=((a_1)_+,(a_2)_+,(a_3)_+)$
when $a=(a_1,a_2,a_3)\in\R^3$.
We have the following optimality result.

\begin{theorem}\label{thinit}
Suppose that the assumptions \eqref{scateh}, \eqref{colleh}, \eqref{co2a} and \eqref{co2aa} are satisfied. Then the minimum $\min_{g\in U_{\rm ad}}{J}(g)$ exists
and is realized at the point $g=\ol g$ given by
\be\label{v9a}
\ol g=N(\gamma_-(\psi^*)):={1\over c}(\gamma_-(\psi^*))_+,
\ee
where the pair $(\psi,\psi^*)\in \tilde{W}^2(G\times S\times I)^3\times \tilde{W}^2(G\times S\times I)^3$ is the unique solution of the coupled non-linear system of variational equations
\bea\label{v10}
B^*(\psi^*,v)&+2 c_{\bf T}\la D\psi,Dv\ra_{ L^2({\bf T})}+2c_{\bf C}\la D\psi,Dv\ra_{ L^2({\bf C})}+2c_{\bf N}\la D\psi,Dv\ra_{ L^2({\bf N})}\nonumber\\
&=2c_{\bf T}\la d_{\bf T},Dv\ra_{ L^2({\bf T})}+2c_{\bf C}\la d_{\bf C},Dv\ra_{ L^2({\bf C})}+2c_{\bf N}\la d_{\bf N},Dv\ra_{ L^2({\bf N})} \\
B(\psi,v)&=\la N(\gamma_-(\psi^*)),\gamma_-(v)\ra_{T^2(\Gamma_-)^3},
\nonumber
\eea
for all  $v\in \tilde W^2(G\times S\times I)^3$.
\end{theorem}

\begin{proof}
By Lemma \ref{dc}  the dose operator ${\s D}:T^2(\Gamma_-)^3\to L^2(G)$ is a bounded linear operator. Hence the object function $J:T^2(\Gamma_-)^3\to \R$ is differentiable and strictly convex (see e.g. the proof of Theorem \ref{pof}; note that the last term makes it strictly convex).
Furthermore, $J$ is bounded from below (in fact, it is non-negative),
$U_{\rm ad}\subset T^2(\Gamma_-)^3$ is convex and
$\lim_{\n{g}_{T^2(\Gamma_-)^3}\to\infty,\ g\in U_{\rm ad}}J(g)=\infty$.
Therefore, there exists a unique minimum $\ol{g}$ of $J$ in $U_{\rm ad}$
and the necessary and sufficient
conditions for the minimality of $\ol{g}$ are (see e.g. \cite{lions71})
\bea\label{optim}
&J'(\ol g)(\tilde{w}-\ol g)\geq 0,\hspace{1.2cm} \forall \tilde{w}\in U_{\rm ad},
\\
&B(\psi(\ol g),v)=(F\ol g)(v),\quad \forall v\in \tilde W^2(G\times S\times I)^3.
\eea

Let
\[
(e_{\bf T} h)(x)=\begin{cases}
h(x),\ &x\in {\bf T}\cr 0,\ &x\in G\setminus {\bf T}
\end{cases}
\]
for a function $h\in L^2({\bf T})$
(the extension by zero onto $G$) and similarly we define extensions by zero for functions   $h\in L^2({\bf C})$ and $h\in L^2({\bf N})$ onto $G$, respectively.
The differential of $J'(g)$ is, recalling that ${\s D}g=D(\psi(g))=:D\psi(g)$,
\bea\label{v10a:1}
J'(g)w=&
2c_{\bf T}\la D\psi(g)-d_{\bf T} ,D\psi(w)\ra_{L_2({\bf T})}
+ 2c_{\bf C}\la D\psi(g)-d_{\bf C} ,D\psi(w)\ra_{L_2({\bf C})}\nonumber\\
&+ 2c_{\bf N}\la D\psi(g)-d_{\bf N} ,D\psi(w)\ra_{L_2({\bf N})}
+2c\la g,w\ra_{T^2(\Gamma_-)^3}\nonumber\\
=&2c_{\bf T}\la e_{\bf T}(D\psi(g)-d_{\bf T}) ,D\psi(w)\ra_{L_2(G)}
+ 2c_{\bf C}\la e_{\bf C}(D\psi(g)-d_{\bf C}) ,D\psi(w)\ra_{L_2(G)}\nonumber\\
&+ 2c_{\bf N}\la e_{\bf N}(D\psi(g)-d_{\bf N}) ,D\psi(w)\ra_{L_2(G)}
+2c\la g,w\ra_{T^2(\Gamma_-)^3}\nonumber\\
=&2c_{\bf T}\la D^*e_{\bf T}(D\psi(g)-d_{\bf T}) ,\psi(w)\ra_{L_2(G)^3}
+ 2c_{\bf C}\la D^*e_{\bf C}(D\psi(g)-d_{\bf C}) ,\psi(w)\ra_{L_2(G)^3}\nonumber\\
&+ 2c_{\bf N}\la D^*e_{\bf N}(D\psi(g)-d_{\bf N}) ,\psi(w)\ra_{L_2(G)^3}
+2c\la g,w\ra_{T^2(\Gamma_-)^3}.
\eea

Denoting
\[
f^*:=c_{\bf T} D^*e_{\bf T}(D\psi(g)-d_{\bf T}) 
+ c_{\bf C} D^*e_{\bf C}(D\psi(g)-d_{\bf C})
+ c_{\bf N}D^*e_{\bf N}(D\psi(g)-d_{\bf N} ),
\]
we have $f^*\in L^2(G\times S\times I)^3$, and so by Theorem \ref{th:LM-ad} there
exists a unique $\psi^*\in \tilde W^2(G\times S\times I)^3$ such that
\be\label{v12}
B(v,\psi^*)=B^*(\psi^*,v)=-\la f^*,v\ra_{L^2(G\times S\times I)^3},\quad \forall v\in 
\tilde W^2(G\times S\times I)^3.
\ee
Moreover, by definition $\psi(w)\in \tilde W^2(G\times S\times I)^3$,
for any $w\in T^2(\Gamma_-)^3$,
as the unique solution (by  Theorem \ref{th:LM}) of the problem \eqref{va1} (with $g=w$ and $f=0$), satisfies
\be\label{v13}
B(\psi(w),v)=(Fw)(v),\quad \forall v\in \tilde{W}^2(G\times S\times I)^3.
\ee
Hence we have
\bea\label{v14}
2\la f^*,\psi(w)\ra_{L^2(G\times S\times I)^3}
&=-2B(\psi(w),\psi^*)=-2(Fw)(\psi^*) \\
&=-2\sum_{j=1}^3\int_{\partial G\times S\times I} (\omega\cdot \nu)_-w_j\psi^*_j d\sigma d\omega dE
\nonumber \\
&=\la -2\gamma_-(\psi^*), w\ra_{T^2(\Gamma_-)^3},\nonumber
\eea
where $\gamma_-:W^2(G\times S\times I)\to L^2_{\rm loc}(\Gamma_-,|\omega\cdot \nu|d\sigma d\omega dE)$
as introduced in Section \ref{ls1}, but now in the context of $L^p$-spaces with $p=2$,
instead of $p=1$, which was the case there.
Combining the previous expressions \eqref{v10a:1} and \eqref{v14} thus leads to
\be\label{v14a}
J'(g)w=&2\la f^*,\psi(w)\ra_{L^2(G\times S\times I)^3}+2c\la g,w\ra_{T^2(\Gamma_-)^3} \\
=&\la -2\gamma_-(\psi^*)+2cg,w\ra_{T^2(\Gamma_-)^3}.
\ee

Choosing for $\tilde{w}$ in the condition (\ref{optim})
subsequently $w+\ol{g}$ and $0$, we see that
\bea
&J'(\ol g)w\geq 0,\quad \forall w\in U_{\rm ad} \label{v15:1}\\
&J'(\ol g)\ol g=0.\label{v15:2}
\eea
Hence, by (\ref{v14a}) and (\ref{v15:1}) one has
\be\label{v16}
\la-\gamma_-(\psi^*)+c\ol g,w\ra_{T^2(\Gamma_-)^3}\geq 0, \quad \forall w\in U_{\rm ad},
\ee
and so for each component $j=1,2,3$,
\be\label{v17}
-\gamma_-(\psi_j^*)+c\ol g_j\geq 0\ {\rm a.e.\ in}\ \Gamma_-.
\ee
On the other hand, due to (\ref{v15:2}) we have
\be\label{v17a}
\la-\gamma_-(\psi^*)+c\ol g,\ol g\ra_{T^2(\Gamma_-)^3}=0,
\ee
and so by (\ref{v17}) for each $j=1,2,3$,
\[
\ol g_j(-\gamma_-(\psi_j^*)+c\ol g_j)=0\ {\rm a.e.\ in}\ \Gamma_-.
\]
From this, using again (\ref{v17}) and the fact that $\ol{g}\geq 0$,
one concludes that
\[
\ol g_j={1\over{c}}\max\{0,\gamma_-(\psi_j^*)\},\quad j=1,2,3,
\]
which is the claim (\ref{v9a}).
Finally, substituting this into the equations (with $\psi=\psi(\ol{g})$),
\[
B(\psi,v)&=(F\ol g)(v),\\
B^*(\psi^*,v)&=-2\la f^*,v\ra_{L^2(G\times S\times I)^3},
\]
and noticing that
\[
\la f^*,v\ra_{L^2(G\times S\times I)^3}
=&c_{\bf T} \la D\psi(g)-d_{\bf T},Dv\ra_{L^2({\bf T})}
+ c_{\bf C}\la D\psi(g)-d_{\bf C},Dv\ra_{L^2({\bf C})} \\
&+ c_{\bf N}\la D\psi(g)-d_{\bf N},Dv\ra_{L^2({\bf N})},
\]
we get the system of equations (\ref{v10}) for the pair $(\psi, \psi^*)$. This completes the proof.
\end{proof}

\begin{remark}
If we choose the whole space $\tilde{U}_{\rm ad}=T^2(\Gamma_-)^3$ as an admissible set
instead of $U_{\rm ad}$ (i.e. if non-negativity of admissible controls was not imposed),
we would find by considerations similar to those in the above proof
that the following variation of Theorem \ref{thinit} holds:

Under the assumptions (\ref{scateh}), (\ref{colleh}), (\ref{co2a}) and (\ref{co2aa}),
the minimum $\min_{g\in \tilde{U}_{\rm ad}}{J}(g)$ exists and is realized at $g=\ol{g}$ given by
\be\label{v9aa}
\ol{g}={1\over{c}}\gamma_-(\psi^*)=:N'(\gamma_-(\psi^*))
\ee
where the pair $(\psi,\psi^*)\in \tilde{W}^2(G\times S\times I)^3\times \tilde{W}^2(G\times S\times I)^3$ is the solution of the coupled linear system of variational equations
\bea\label{v10a}
B^*(\psi^*,v)&+
2c_{\bf T}\la D\psi,Dv\ra_{ L^2({\bf T})}+2c_{\bf C}\la D\psi,Dv\ra_{ L^2({\bf C})}+2c_{\bf N}\la D\psi,Dv\ra_{ L^2({\bf N})}\nonumber\\
&=2c_{\bf T}\la d_{\bf T},Dv\ra_{ L^2({\bf T})}+2c_{\bf C}\la d_{\bf C},Dv\ra_{ L^2({\bf C})}+2c_{\bf N}\la d_{\bf N},Dv\ra_{ L^2({\bf N})} \\
B(\psi,v)&=\la N'(\gamma_-(\psi^*)), \gamma_-(v)\ra_{T^2(\Gamma_-)^3},
\nonumber
\eea
which holds for all  $v\in \tilde W^2(G\times S\times I)^3$.

By using this technique the initial solution for the full optimization
problem of finding the minimum of $J_{\rm ex}$ on $U_{\rm ad}$
would be taken to be
$\frac{1}{c}(\gamma_-(\psi_1^*))_+$. 
We point out that the equations (\ref{v10a}) are linear,
since $\psi^*\mapsto N'(\gamma_-(\psi^*))$ is linear,
and therefore no iteration scheme is necessarily required in solving them.
Presumably, however, the solution of
non-linear optimization problem given in Theorem \ref{thinit}
should give a more accurate initial solution $N(\gamma_-(\psi^*))$
for the full optimization problem $\min_{g\in U_{\rm ad}} J_{\rm ex}(g)$,
but this question will not be explored any further in this paper.
Similar observation is concerning the internal therapy optimization described below.
 
\end{remark}

In internal therapy $g=0$ and so $\psi=\psi(f):=\psi(f,0)$ and ${\s D}(f)=D(\psi(f))$. The object function for the initial solution may be
\be\label{va18}
J(f)=&c_{\bf T}\n{ d_{\bf T}-{\s D}(f)}_{L^2({+bf T})}^2+c_{\bf C}\n{ d_{\bf C}-{\s D}(f)}_{L^2({\bf C})}^2 \\
&+c_{\bf N}\n{d_{\bf N}-{\s D}(f)}_{L^2({\bf N})}^2+c\n{f}_{L^2(G\times S\times I)^3}^2. \nonumber
\ee 
and
\[
U'_{\rm ad}=\{f\in L^2(G\times S\times I)^3\ |\ f\geq 0\}.
\]
By arguments analogous to those used to prove Theorem \ref{thinit}
lead to the next result.

\begin{theorem}\label{thinit:internal}
Assume that \eqref{scateh}, \eqref{colleh}, \eqref{co2a}, \eqref{co2aa} hold.
Then the minimum $\min_{f\in U'_{\rm ad}}{J}(f)$ exists at the point $f=\ol{f}\in U'_{\rm ad}$
where
\be\label{v9a:3}
\ol{f}=\frac{1}{c}(\psi^*)_+=:N(\psi^*),
\ee
and the pair $(\psi,\psi^*)\in \tilde{W}^2_{-,0}(G\times S\times I)^3\times \tilde{W}^2_{+,0}(G\times S\times I)^3$
is the solution of the coupled non-linear system of variational equations
\bea\label{v10:1}
B^*(\psi^*,v)&+
2c_{\bf T}\la D\psi,Dv\ra_{ L^2({\bf T})}+2c_{\bf C}\la D\psi,Dv\ra_{ L^2({\bf C})}
+2c_{\bf N}\la D\psi,Dv\ra_{ L^2({\bf N})}\nonumber\\
&=2c_{\bf T}\la d_{\bf T},Dv\ra_{L^2({\bf T})}+2c_{\bf C}\la d_{\bf C},Dv\ra_{L^2({\bf C})}
+2c_{\bf N}\la d_{\bf N},Dv\ra_{L^2({\bf N})} \\
B(\psi,v)&=\la N(\psi^*),v\ra_{ L^2(G\times S\times I)^3} \nonumber,
\eea
for all  $v\in \tilde W^2(G\times S\times I)^3$.
\end{theorem}

\begin{proof}
We shall content ourselves here with sketching briefly the part of proof leading to \eqref{v9a:3},
as this will be referred to in the remark that follows.
Computations similar to those leading to \eqref{v16} in the proof of Theorem \ref{thinit},
would give in in the current context,
\bea\label{eq:optineq:internal}
\la -\psi^* + c \ol{f},w\ra_{L^2(G\times S\times I)^3} \geq 0,\quad \forall w\in U'_{\rm ad},
\eea
hence
\[
-\psi^* + c \ol{f}\geq 0\ \textrm{a.e. in}\ G\times S\times I,
\]
and those leading to \eqref{v17a} would give
\[
\la -\psi^* + c \ol{f},\ol{f}\ra_{L^2(G\times S\times I)^3}=0.
\]
Since $\ol{f}\geq 0$ as $\ol{f}\in U'_{\rm ad}$, 
we thus get
\bea\label{eq:opteq:internal}
\ol{f}(-\psi^* + c \ol{f})=0\ \textrm{a.e. in}\ G\times S\times I,
\eea
from which \eqref{v9a:3} easily follows.
\end{proof}

As we mentioned  these solutions can be utilized as the initial solution for the (global) optimization but they are not ready solutions for the treatment planning.

\begin{remark}\label{oads}
Here we discuss some other choices of admissible sets.
In \cite{frank08} (see also \cite{frank10}) one considers monoenergetic model for one species of particles. The existence and analogous optimal control formulas as above have been shown for the internal therapy when $f=f(x)$, that is when $f$ is independent of
the direction $\omega$ and energy $E$. This corresponds to the situation where one chooses
for the admissible control  the set (see below)
\[
\tilde{\tilde{U}}'_{\rm ad}=\{f\in L^2(G)\ |\ f\geq 0\}.
\] 
The practical availability for delivery (of internal therapy) is nowadays typically of this kind.
Moreover, in the referred paper they considered a term of the objective function
of the type $c\n{f-f_0}_{L^2(G)}^2$ instead of $c\n{f}_{L^2(G)}^2$,
where $f_0\in L^2(G)$ is a known source distribution. 

Assume that $ f_0=0$ (the generalization for $f_0\neq 0$ is straightforward).
Supposing, moreover, that admissible controls  are independent of $E$,
i.e. $f=f(x,\omega)$,
\[
\tilde{U}'_{\rm ad}=\{f\in L^2(G\times S)\ |\ f\geq 0\},
\] 
one gets the following necessary condition for optimal control $f=\ol{f}$:
\be\label{v19}
\ol f
=\frac{1}{c|I|}\Big(\int_I\psi^* dE\Big)_+
=:\tilde N(\psi^*) 
\ee
and $(\psi,\psi^*)\in \tilde{W}^2_{-,0}(G\times S\times I)^3\times \tilde{W}^2_{+,0}(G\times S\times I)^3$
is the solution of the coupled non-linear system of equations
\bea\label{v10:2}
B^*(\psi^*,v)
&+
c_{\bf T}\la D\psi,Dv\ra_{ L^2({\bf T})}+c_{\bf C}\la D\psi,Dv\ra_{ L^2({\bf C})}+c_{\bf N}\la D\psi,Dv\ra_{ L^2({\bf N})}\nonumber\\
&=c_{\bf T}\la d_{\bf T},Dv\ra_{ L^2({\bf T})}+c_{\bf C}\la d_{\bf C},Dv\ra_{ L^2({\bf C})}+c_{\bf N}\la d_{\bf N},Dv\ra_{ L^2({\bf N})} \\
B(\psi,v)&=\la\tilde N(\psi^*),v\ra_{L^2(G\times S\times I)^3},
\nonumber
\eea
for all $v\in \tilde W^2(G\times S\times I)^3$. Above, $|I|$ denotes the length of the interval $I$.

That the optimal solution $\ol{f}$ indeed has the above form \eqref{v19}
can be seen from the proof of Theorem \ref{thinit:internal},
where \eqref{eq:optineq:internal} now holds for all $w\in \tilde{U}'_{\rm ad}$,
which gives $-\int_I \psi^* dE+c|I|\ol{f}\geq 0$,
and eventually the corresponding version of equation \eqref{eq:opteq:internal} would be
\[
\ol{f}(-\int_I \psi^* dE+c|I|\ol{f})=0,
\]
which leads directly to \eqref{v19}.

Finally, assuming that admissible controls $f$ are independent of both $E$ and $\omega$,
i.e. $f=f(x)$,
\[
\tilde{\tilde{U}}'_{\rm ad}=\{f\in L^2(G)\ |\ f\geq 0\},
\]
one gets the following a necessary condition for optimal control $f=\ol{f}$:
\be\label{v19:1}
\ol f
=\frac{1}{4\pi c|I|}\Big(\int_{S\times I}\psi^* dE d\omega\Big)_+
=:\tilde{\tilde N}(\psi^*) 
\ee

where $(\psi,\psi^*)\in \tilde{W}^2_{-,0}(G\times S\times I)^3\times \tilde{W}^2_{+,0}(G\times S\times I)^3$
is the solution of the coupled non-linear system of variational equations
\bea\label{v20}
B^*(\psi^*,v)
&+c_{\bf T}\la D\psi,Dv\ra_{ L^2({\bf T})}+c_{\bf C}\la D\psi,Dv\ra_{ L^2({\bf C})}+c_{\bf N}\la D\psi,Dv\ra_{ L^2({\bf N})}\nonumber\\
&=c_{\bf T}\la d_{\bf T},Dv\ra_{ L^2({\bf T})}+c_{\bf C}\la d_{\bf C},Dv\ra_{ L^2({\bf C})}+c_{\bf N}\la d_{\bf N},Dv\ra_{ L^2({\bf N})}, \\
B(\psi,v)&=\la\tilde{\tilde N}(\psi^*),v\ra_{L^2(G\times S\times I)^3} ,
\nonumber
\eea
for all  $v\in \tilde W^2(G\times S\times I)^3$.
The arguments leading to \eqref{v19:1} are easy adaptations to the set $\tilde{\tilde{U}}'_{\rm ad}$ of
the steps between \eqref{eq:optineq:internal}-\eqref{eq:opteq:internal} in the proof of Theorem \ref{thinit:internal},
precisely as discussed above when justifying \eqref{v19}.
\end{remark}

A similar necessary formula can be obtained in the case of external therapy  which can be seen from (\ref{v17a}). When $g$ is independent of energy $E$ we have
\bea\label{v21}
\ol{g}
=\frac{1}{c|I|}\Big(\int_I\gamma_-(\psi^*) dE\Big)_+
=:\tilde N(\gamma_-(\psi^*)) 
\eea
when $\psi^*\in $ is the solution of the coupled nonlinear system of variational equations
\bea\label{v22}
&B^*(\psi^*,v)+
c_{\bf T}\la D\psi,Dv\ra_{ L^2({\bf T})}+c_{\bf C}\la D\psi,Dv\ra_{ L^2({\bf C})}+c_{\bf N}\la D\psi,Dv\ra_{ L^2({\bf N})}\nonumber\\
&=c_{\bf T}\la d_{\bf T},Dv\ra_{ L^2({\bf T})}+c_{\bf C}\la d_{\bf C},Dv\ra_{ L^2({\bf C})}+c_{\bf N}\la d_{\bf N},Dv\ra_{ L^2({\bf N})}\nonumber\\
&B(\psi,v)=(F(\tilde N(\gamma_-(\psi^*)))(v) 
\eea
for all  $v\in \tilde W^2(G\times S\times I)^3$.
In this case we have
\[
U_{\rm ad}=\{g\in L^2(\Gamma_-'),|\omega\cdot \nu|d\sigma d\omega)^3|\ g\geq 0\}
\] 
where $\Gamma_-'=\{(y,\omega)\in\partial G\times S|\ \omega\cdot\nu(y)<0\}$.
The independence of $g$  from angular $\omega$ is not reasonable in external therapy.

Finally, we notice that all the above formulas of optimal solutions can be applied in the case where only one species of particles is incoming or it is as a source in tissue
(simply choose $g=(g_1,0,0)$ and so on). 

\begin{remark}
Existence and formulas of optimal control for convex differentiable object functions on convex domains exists also for time-dependent (infinite dimensional) control systems (see e.g. \cite{tanabe}, Chapter 7).
\end{remark}

\begin{remark}

In the case where the actual object function $J_{\rm ex}$
in \eqref{fof} or \eqref{fofa} is differentiable at a (local) optimal point $\ol g\in U_{\rm ad}$ a necessary condition is that
\bea\label{ncond}
J'_{\rm ex}(\ol g)(g-\ol g)\geq 0\ {\rm for\ all}\ g\in U_{\rm ad}\cap B(\ol{g},r),
\eea
where $B(\ol{g},r)$ is the open ball of radius $r>0$ around $\ol{g}$ in $T^2(\Gamma_-)^3$.

On the other hand, if the admissible set is taken to be the whole space $\tilde{U}_{\rm ad}=T^1(\Gamma_-)^3$, the condition (\ref{ncond}) reduces to
\bea\label{ncond1}
J'_{\rm ex}(\ol g)=0.
\eea

For globally convex object function the condition (\ref{ncond1}) is both necessary and sufficient for the global optimal control point $\ol g\in U_{\rm ad}$ (when $J_{\rm ex}'(\ol g)$ exists).
Similar facts are true for the object function $J_{\rm in}$.
Recall that if a convex function $U\to\R$ has a minimum, it is necessarily global minimum even if the function is not differentiable. This implies especially that when the dose volume constraint is not included in the object function, the (local) gradient based optimization methods can be applied "on the sets where gradient exists".
For extensive literature of needed optimization and numerical analysis and techniques we refer to the recent monograph \cite{allaire07}.
\end{remark}

%%%%%%%%%%%%%%%%%%%%%%%%%%%%%%%%%%%%%%%%%%%%%%%%%%%%%%%%%%%%%%%%%%%%%%%%
\subsubsection{Proposed Optimization Strategy}\label{optims}
%%%%%%%%%%%%%%%%%%%%%%%%%%%%%%%%%%%%%%%%%%%%%%%%%%%%%%%%%%%%%%%%%%%%%%%%

The final optimization, that is, the inverse radiation treatment planning,
could be realized in the following three phases:

{\bf 1.} Compute the initial solution by Theorem \ref{thinit} or by its modification given in Remark \ref{oads}.
This step consists of carrying out \emph{convex differentiable optimization}.
However, it is not sufficient by itself because the optimal plan (control)
obtained for the (partial) object function \eqref{va9}
may produce unwanted dose to the critical organ/normal tissue. 

{\bf 2.} Compute the optimal plan for object function (\ref{fof}) \emph{without} the dose volume constraint (i.e. $c_{\rm CV}=0$),
using as the initial guess the solution obtained in step 1.
This step involves carrying out \emph{Lipschitz continuous convex optimization}.

{\bf 3.} Compute the optimal plan for the object function (\ref{fof}) \emph{with} the dose volume constraint (i.e. $c_{\rm CV}>0$),
using as the initial guess the solution acquired in step 2.
In this step one needs to perform \emph{non-convex optimization},
and needs, therefore, a global optimization scheme.

One may optionally add between the steps 2. and 3. an intermediate
optimization phase where the object function for the dose volume
constraint is replaced by a Lipschitz continuous term
\[
J_{{\rm DV},\epsilon}(0,g)=\Big(\big(v_C-\frac{1}{\mc{L}^3({\bf C})}\int_{\bf C}H_\epsilon({\s D}(0,g)(x)-d_C) dx\big)_-\Big)^p,
\]
where $H_\epsilon$ is a continuous approximation of the Heaviside function $H$,
for example
\be\label{dca}
H_\epsilon(x)=\begin{cases}
0,\ &x\leq 0\cr
{1\over\epsilon}x,\ &0\leq x\leq \epsilon\cr
1,\ &x\geq \epsilon\cr
\end{cases}.
\ee
Alternatively, this modified (smooth) dose volume term could be used in step 3.
When $p=2$ the further modified term
\[
\tilde{J}_{{\rm DV},\epsilon}(0,g)=\Big(v_C-\frac{1}{\mc{L}^3({\bf C})}\int_{\bf C}H_\epsilon({\s D}(0,g)(x)-d_C) dx\Big)^2
\]
turns out to be differentiable and (globally) Lipschitz continuous, which might be used to facilitate the optimization.
It should be pointed out, however, that the inherent \emph{non-convexity} of $J_{\rm DV}$
cannot be removed.

%%%%%%%%%%%%%%%%%%%%%%%%%%%%%%%%%%%%%%%%%%%%%%%%%%%%%%%%%%%%%%%%%%%%%%%%
\subsubsection{Some Remarks on the Discrete Problem and Modelling}
%%%%%%%%%%%%%%%%%%%%%%%%%%%%%%%%%%%%%%%%%%%%%%%%%%%%%%%%%%%%%%%%%%%%%%%%

\newcounter{c:ss:last}
\setcounter{c:ss:last}{0}
\newcommand{\csslast}{\printcbf{c:ss:last}}

In practical radiation treatment planning the coupled Boltzmann transport equation must be discretized. Commonly used methods for discretization are finite element method (FEM),
or collocation method in the spatial variable $x$ and in the energy variable $E$ and spherical harmonics in the angle variable $\omega$
(cf. e.g. \cite{ackroyd}, \cite{barnard06}, \cite{bomanthesis}).
We do not consider these issues here but list some of the challenges that the resulting discrete problems involve.
We shall denote below the set of $n\times m$ matrices by $\mathrm{M}(n\times m)$.

%%%%%%%%%%
\csslast %
Applying appropriate discretization methods
the finite dimensional approximation of the transport equation (\ref{bte1})
is of the form 
\bea\label{dp}
A\alpha=B\beta
\eea
where $A\in\mathrm{M}(3N\times 3N)$ and $B\in \mathrm{M}(3N\times 3M)$.
The approximative components of the solution $\psi$ of \eqref{bte1} are
\[
\psi_j\approx\tilde{\psi}_j=\tilde{\psi}_j(f,g):=\sum_{k=1}^N\alpha_{jk}\phi_k(x,\omega,E),=:{\Psi}_j(x,\omega,E)\alpha\quad j=1,2,3,
\]
where $\{\phi_k\ |\ k=1,...,N\}$ is a basis
for a chosen finite dimensional subspace $W_N$ of $\tilde{W}^2(G\times S\times I)^3$ (in the case $p=2$),
and the coefficients $\alpha_{jk}$ are obtained from $\alpha=A^{-1}B\beta$
(we omit here the detailed arrangements of matrices, for details see e.g. \cite{bomanthesis}).  Above ${\Psi}_j(x,\omega,E)\in \mathrm{M}(1\times N)$ is a matrix, which is computed with the help of basis functions.

For example,  each element $\phi_k$ of the above basis
might be taken to be finite linear combinations
of (tensor) products of the form $\varphi_{p}(x)\Omega_{q}(\omega){\s E}_{r}(E)$,
with $\varphi_{p}\in H^1(G)$ (the standard Sobolev space for $p=2$), $\Omega_q\in L^2(S)$ and ${\s E}_r\in L^2(I)$.

The column vector $\beta\in\R^{M}$ 
contains (is calculated from) the discretized known input data (internal sources and/or 
incoming flux). For example, in the case of external radiotherapy we put
\be
g_j\approx \sum_{k=1}^M\beta_{jk}\eta_k(y,\omega,E)
\ee
where $\eta_k$ is an appropriate basis of $M$-dimensional subspace of $T^2(\Gamma_-)$.
We see that 
\be 
g_j\approx {\s G}_j(y,\omega,E)\beta
\ee
where ${\s G}_j(y,\omega,E)\in \mathrm{M}(1\times M)$ (computed with the help of basis functions $\eta_k$).
Note that using the above matrices
\be 
\psi\approx {\Psi}(x,\omega,E)\alpha
\ee
and
\be 
g\approx {\s G}(y,\omega,E)\beta
\ee
for some matrices
${\Psi}(x,\omega,E)\in \mathrm{M}(1\times (3N))$  and ${\s G}(y,\omega,E)\in \mathrm{M}(1\times (3M))$ (obtained with the help of matrices  ${\Psi}_j(x,\omega,E)$ and ${\s G}_j(y,\omega,E)$, respectively).

In the case where FEM scheme is applied, the matrices $A$ and $B$ can be 
computed in a standard way from the variational form of the transport equations. The 
conditions (\ref{co2a}) and (\ref{co2aa}) guarantee the {\it convergence} of the FEM scheme by the well-known Cea's estimate (for $p=2$) since they imply the boundedness and coercitivity of the bilinear form $B(.,.)$ given in section \ref{cis} in appropriate Hilbert spaces.

By the above the dose operator ${ D}$ is approximated by
\[
D(x)\approx\tilde { D}(f,g)(x):=\tilde D(x)= 
\sum_{j=1}^3\sum_{k=1}^N\int_S\int_I\alpha_{jk}\kappa_j(x,E)\phi_k(x,\omega,E)d\omega dE,
=:{\s D}(x)\alpha
\]
for some ${\s D}(x)\in \mathrm{M}(1\times (3N)$,
while the terms for the approximative object function,
e.g. in the case of external radiotherapy, are given (for general $p>1$) by
\[
J_{\bf T}(0,g)&\approx\n{D_0-{\s D}(\cdot)\alpha}_{L^p({\bf T})}^p=:J_{\bf T}(\alpha)\\
J_{\bf C}(0,g)&\approx\n{\big(D_C-{\s D}(\cdot)\alpha\big)_-}_{L^p({\bf C})}^p=:J_{\bf C}(\alpha)\\
J_{\bf N}(0,g)&\approx\n{\big(D_N-{\s D}(\cdot)\alpha\big)_-}_{L^p({\bf N})}^p=:J_{\bf N}(\alpha)\\
J_{\rm DV}(0,g)&\approx\Big(\big(v_C-\frac{1}{\mc{L}^3({\bf C})}\int_{\bf C}H\big(({\s D}(\cdot))(x)-d_C\big) dx\big)_-\Big)^p=:J_{DC}(\alpha)\\
J_{\rm sc}(0,g)&\approx\n{\Psi(\cdot,\cdot,\cdot)\alpha)}^p_{L^p(G\times S\times I))^3}=:J_{\rm sc}(\alpha)
\]
and so the approximation of the whole object function is 
\be \label{diskob}
J=J(\alpha)=J_{\bf T}(\alpha)+J_{\bf C}(\alpha)+J_{\bf N}(\alpha)+J_{DV}(\alpha)+
J_{\rm sc}(\alpha).
\ee
Substituting $\alpha=A^{-1}B\beta$ to (\ref{diskob}) we get the object function with the help of control variables $\beta$.

The dimensionality of the discretized problem \eqref{dp} is typically very large
in the number $N$ of unknowns $\alpha_{jk}$,
although it can be reduced by techniques like the adaptation of the grid.
This is one of the main drawbacks of the method
because to form the inverse $A^{-1}$ 
one must calculate the inverse of the very large dimensional matrix,
even if matrices involved in FEM, as is well known, are sparse.
Iterative algorithms must be applied in solving the equation $A\alpha=B\beta$.
One can partially avoid this problem by applying the so-called parametrization, described below (cf. \cite{bomanthesis}, \cite{tervo08}),
but then another difficulty arises in constructing the parametrization operator (based e.g. on the Singular Value Decomposition (SVD)). 

The initial solution for the discrete problem can be calculated as follows.
Denote
\be 
D(\psi,v):=2c_{\bf T}\la D\psi,Dv\ra_{L^2({\bf T})}+2c_{\bf C}\la D\psi,Dv\ra_{L^2({\bf C})}+
2c_{\bf N}\la D\psi,Dv\ra_{L^2({\bf N})}
\ee
and 
\be 
d(v):=
2c_{\bf T}\la d_{\bf T},Dv\ra_{L^2({\bf T})}+2c_{\bf C}\la d_{\bf C},Dv\ra_{L^2({\bf C})}+
2c_{\bf N}\la d_{\bf N},Dv\ra_{L^2({\bf N})}.
\ee
Then the variational equations \eqref{v9a}-\eqref{v10} are
\bea\label{sve}
& B^*(\psi^*,v)+D(\psi,v)=d(v)\nonumber\\
& B(\psi,v)={1\over c}\la (\gamma_-(\psi^*))_+,\gamma_-(v)\ra_{L^2(\Gamma_-)^3},\ v\in \tilde W^2(G\times S\times I)^3.
\eea

The discrete approximation of the system (\ref{sve}) is of the form
\bea\label{sve1}
& A^*\xi+{\bf D}\alpha={\bf d}\nonumber\\
& A\alpha ={\bf g}(\xi)
\eea
where ${\bf D}\in \mathrm{M}(3N\times 3N)$, ${\bf d}\in \mathrm{M}(3N\times 1)$, ${\bf g}$ is a piecewise linear (non)function and
\be \label{sve2}
\psi_j^*\approx \tilde\psi_j^*:=\sum_{k=1}^N\xi_{jk}\phi_k.
\ee
The optimal control is approximately
\be 
\ol g={1\over c}(\gamma_-(\psi^*))_+\approx\ {1\over c}(\gamma_-(\tilde\psi^*))_+
\ee
where $\tilde\psi^*:=(\tilde\psi_1^*,\tilde\psi_2^*,\tilde\psi_3^*)$ is obtained from (\ref{sve2}) with the help of $\xi$.

%%%%%%%%%%
\csslast %
The  term \emph{parametrization} above means the following concept. The discrete system (\ref{dp}) can be written as
\be\label{dpa}
\qmatrix{A & -B\cr}\qmatrix{\alpha\cr\beta\cr}=0,
\ee
where $\qmatrix{A & -B\cr}\in \mathrm{M}(N\times(N+ M))$ and
 $\qmatrix{\alpha\cr\beta\cr}\in \mathrm{M}((N+ M)\times 1)$. Let $P\in \mathrm{M}((N+M)\times N')$ be a matrix such that (\ref{dpa}) holds if and only if 
\be\label{dpb} 
 \qmatrix{\alpha\cr\beta\cr}=P\tau
\ee
that is, $P\in \mathrm{M}(N'\times 1)$ is the "basis generating matrix (operator) of the kernel $N( 
\qmatrix{A & B\cr})$".
Such a matrix $P$ always exists and is called the {\it parametrization (operator/matrix)} of the system (\ref{dp}).

We observe that if $Q$ is a matrix such that
\be \label{dpaa}
\qmatrix{A & -B\cr}Q\qmatrix{A & -B\cr}=\qmatrix{A & -B\cr}
\ee
then $P:=I-Q\qmatrix{A & B\cr}$ is a parametrization. 
Especially, (\ref{dpaa}) is valid if $Q=\qmatrix{A & -B\cr}^+$ is the Moore-Penrose pseudo-inverse of $\qmatrix{A & -B\cr}$. Note that when applying (\ref{dpaa}) the dimension (number of rows and columns) of $P$ is $N+M$ but it can be essentially reduced by omitting insignificant elements.

In virtue of (\ref{dpb}) we have $\alpha=P_1\tau,\ \beta=P_2\tau$ for some matrices $P_j,\ j=1,2$ obtained from blocks of $P$. The object function becomes with the help of parameters $\tau$ as
\be 
J=J(\alpha)=J(P_1\tau)=: J(\tau).
\ee
The optimization problem becomes the following: Find the global minimum
\be 
\inf_{\tau\in U_{\rm ad}^d} J(\tau)
\ee
where 
\be 
U_{\rm ad}^d:=\{\tau\in \R^{N'}\ |\ g\approx {\s G}(y,\omega,E)\beta={\s G}(y,\omega,E)P_2\tau\geq 0\}.
\ee
In the case where the basis  $\{\eta_k\}$ is build up of positive step functions (zero-order splines),
the interiors of supports of which are disjoint,
we have
$U_{\rm ad}^d=\{\tau\in \R^{N'}\ |\ \beta=P_2\tau\geq 0\}$.
Using higher order splines would, however, be preferred.
An alternative possibility for taking care of the positivity of the approximative controls $g$ is to add a penalty term of the form $c_{\rm ad}\n{({\s G}(\cdot,\cdot,\cdot)P_2\tau)_-}^p_{T^2(\Gamma_-)^3}$ to the object function.
No explicit inversion of the matrix $A$ is needed.
The essential problem in this approach is in constructing the parametrization $P$,
approximatively, and preferably such that $N'$ (the number of parameters) is small.
Moreover, the algorithms used in this construction should be a iterative schemes,
during which the accuracy can be controlled.
Elements of $P$ which are small enough should be neglected,
such that the dimensionality of $P$ gets decreased and its sparsity gets increased.
Preliminary simulations have shown that this approach works at least
in spatially $2D$-situations (cf. \cite{bomanthesis}, where $N+M\approx 5000,\ N'\approx 100$).
For applying the explained parametrization method,
an initial solution $\tau$ for the optimization can be obtained e.g. as in \cite{bomanthesis}, p. 110 (we omit the details here).

%%%%%%%%%%
\csslast %
Another possibility to avoid inversions of huge matrices would be to utilize in computations the formulas given in Remark \ref{re:calc:2}, that is to compute $\psi=\psi(f,g)$ from
\be\label{nb1}
\psi=\sum_{k=0}^\infty ((-{\bf A}_0+\Sigma)^{-1}K)^k
((-{\bf A}_0+\Sigma)^{-1}(f-(\Sigma-K)(Lg)))+Lg,
\ee
where $(-{\bf A}_0+\Sigma)^{-1}$ can be explicitly obtained from (\ref{nb0}).
Alternatively, one could compute $\psi=\psi(f,g)$ approximately from
(see \eqref{calceq4})
\be\label{nb2}
\psi\approx
\int_0^T\Big[
T(t/n_0)
e^{-(t/n_0)\Sigma(x,\omega,E)}
\sum_{k=0}^{N_0} {1\over{k!}}((t/n_0)K)^k\Big]^{n_0}(f-(\Sigma-K)(Lg)) dt+Lg.
\ee
Substituting one of these expressions into
\[
({\s D}(f,g))(x)=\sum_{j=1}^3\int_{S\times I}\kappa_j(x,E)(\psi_j(f,g))(x,\omega,E) d\omega dE
\]
one acquires the (approximate) dose as a function of $f$ and $g$.
Consequently, the object function $J=J(f,g)$ can be directly calculated from (\ref{of}), (\ref{ofsc1}), (\ref{ofsc2}).
 
The initial solution $\ol g$ (e.g. for external radio therapy)
for applying this computational scheme
is calculated from 
\[
\ol g={1\over c}(\gamma_-(\psi^*))_+,
\]
where $\psi^*$ is solved from the coupled system (see the proof of Theorem \ref{thinit})
\bea\label{nb3}
(-{\bf A}^*+\Sigma^*-K^*)\psi^* +&  c_{\bf T}D^*e_{\bf T}D\psi+c_{\bf C}D^*e_{\bf C}D\psi+c_{\bf N}D^*e_{\bf N}D\psi \nonumber\\
&=c_{\bf T}D^*e_{\bf T}d_{\bf T}+c_{\bf C}D^*e_{\bf C}d_{\bf C}+c_{\bf N}D^*e_{\bf N}d_{\bf N},\nonumber\\
(-{\bf A}+\Sigma-K)\psi, &=0\\
\psi^*_{|\Gamma_+} &=0,\nonumber\\
\psi_{|\Gamma_-} &={1\over c}(\gamma_-(\psi^*))_+, \nonumber
\eea
where the system of equations is equivalent to
\bea\label{nb4}
&\qmatrix{-{\bf A}^*+\Sigma^*-K^* & c_{\bf T}D^*e_{\bf T}D+c_{\bf C}D^*e_{\bf C}D+
c_{\bf N}D^*e_{\bf N}D\cr
0 & -{\bf A}+\Sigma-K \cr}\qmatrix {\psi^*\cr \psi\cr}
\nonumber\\
=&
\qmatrix{c_{\bf T}D^*e_{\bf T}d_{\bf T}+c_{\bf C}D^*e_{\bf C}d_{\bf C}+c_{\bf N}D^*e_{\bf N}d_{\bf N}\cr 0\cr}
\eea
As far as the authors are aware of, computationally effective and stable techniques 
for solving (\ref{nb3}) (for instance using formulas similar to (\ref{nb1}), 
(\ref{nb2})) require further study.

%%%%%%%%%%
\csslast %
Because of their strongly forward-peaked migration,
it would be reasonable to use the Continuous Slowing Down Approximation (see the introduction) 
in the transport of electrons and positrons.
When the solution of the transport equation is smooth enough (see section \ref{fr}),
higher order spline basis functions (along with related more rapid convergence results)
could be used in numerical techniques like FEM.

%%%%%%%%%%
\csslast %
Except for $J_{\rm DV}$
(see, however, the discussion at the end of Section \ref{optims}), the terms of
the discretized object function are (locally) Lipschitz continuous.
Nonetheless, while they are convex,
the terms $J_{\bf T}$, $J_{\bf C}$, $J_{\bf N}$, $J_{\rm ad}$, $J_{\rm sc}$
are not differentiable in general,
except for the case $p=2$ in which case $J_{\bf T}$ and $J_{\rm sc}$ are differentiable
(see Theorem \ref{pof}).
The term $J_{\rm DV}$, however, is non-convex,
and therefore a global optimization strategy is needed if this
constraint is to be taken into account in the treatment planning.
There exist several global optimization algorithms well suited for Lipschitz continuous
(not necessarily differentiable) object functions (e.g. \cite{pinter}).
Large dimensionality of the related (discretized) object function's variables is, however,
a limiting factor for the application of these methods in practice.  

%%%%%%%%%%
\csslast %
Multicriteria optimization and related (interactive) decision making can be applied to the treatment planning applying the presented optimization schemes
(\cite{ruotsalainen}). In addition, we remark that optimization can be used simultaneously for external and internal therapy (which is not likely applied in practise).

%%%%%%%%%%
\csslast %
As we mentioned in the introduction,
in the case of external radiotherapy the incoming flux
(or fluence) $g$ can be essentially expressed using \emph{beam parameters},
which is to be understood include relevant (controllable) variables
like the energy of the incoming beam,
multileaf collimator (MLC) leaf positions, the jaw positions as well as rotational parameters related to the gantry and collimator rotations etc.
(this is by no means intended to be an exhaustive list).
The dose optimization problem can then be put in the form where the object function 
is expressed in terms of beam parameters.

This approach has the advantage that device constraints can be taken into account at an early stage of the treatment planning. The main disadvantage, however,
is that the resulting object function is likely to be highly \emph{multiextremal},
and so effective global optimization algorithms are fundamental
for the success of such an approach.
Notice that the approach given here enables to optimize besides of position, the energies and angles of
incoming flux(es) since $g=g(y,\omega,E)$.  

%%%%%%%%%%
\csslast %
Stochastic aspects (arising e.g. from delivery processes or patient motions during the treatment) can be taken into account by using as the transport model the so-called {\it stochastic Boltzmann transport equation}.
Matters like inverse treatment planning interpreted as an optimal control (boundary) problem, existence of optimal control and its computation, exact controllability and so on, can be then considered in the (more general) framework of the stochastic calculus.
For a glimpse of some recent advances in the context of stochastic BTE and its controllability, we refer e.g. to \cite{Lu13} and the references therein.
Issues of exact controllability (and observability) are considered there for time-dependent monokinetic single particle transport equation.

%%%%%%%%%%
\csslast %
We emphasize that in the computations of the object function,
with the exception of the additional terms $J_{\rm sc}$ and $J_{\rm ad}$,
one only needs to know of the dose distribution
\[
D(x)=\sum_{j=1}^3\int_S\int_I\kappa_j(x,\omega)\psi_j(x,\omega,E)d\omega dE
\]
which is a kind of a moment. 
It might thus be possible to develop iterative approximative methods for calculating the dose without explicitly solving $\psi$.
These techniques lead to recursive computations of some tensors, which also seem to have a physical meaning.

%%%%%%%%%%%%%%%%%%%%%%%%%%%%%%%%%%%%%%%%%%%%%%%%%%%%%%%%%%%%%%%%%%%%%%%%

\end{document}